\documentclass[reqno]{amsart}
\usepackage{amsmath, amssymb, amsthm, amsfonts, amsgen}
\usepackage[centertags]{amsmath}
\usepackage{indentfirst}
\usepackage{fancyhdr}
\usepackage{graphicx}
\usepackage{psfrag}
\usepackage{esint,color}
\usepackage[export]{adjustbox}
\numberwithin{equation}{section}
\usepackage[foot]{amsaddr}

\usepackage[toc,page]{appendix}

\newcommand{\ds}{\displaystyle}
\newcommand{\R}{{\mathbb{R}}}

\newcommand{\N}{{\mathbb{N}}}
\newcommand{\dx}{\,dx}
\newcommand{\dt}{\,dt}
\newcommand{\dy}{\,dy}
\newcommand{\dz}{\,dz}
\newcommand{\dH}{\,d\mathcal{H}^{n-1}(y)}
\newcommand{\ie}{{; \it i.e., }}

\newcommand{\wto}{\rightharpoonup}

\newcommand{\E}{\mathcal{E}}
\newcommand{\F}{\mathcal{F}}
\newcommand{\MS}{\mathcal{MS}}

\let\e=\varepsilon
\let\r=\varrho
\let\d=\delta
\let\O=\Omega
\let\a=\alpha
\let\b=\beta
\let\G=\Gamma

\newtheorem{definition}{Definition}[section]

\newtheorem{theorem}[definition]{Theorem}
\newtheorem{proposition}[definition]{Proposition}

\theoremstyle{definition}
\newtheorem{remark}[definition]{Remark}

\begin{document}

\title[Second-order Ambrosio-Tortorelli functional]
{
  Second-order edge-penalization in \\the Ambrosio-Tortorelli functional}

\author[M. Burger, T. Esposito, and C.I. Zeppieri]{Martin Burger$^1$, Teresa Esposito$^1$, and Caterina Ida Zeppieri$^1$}
\address{$^1$Institut f\"ur Numerische und Angewandte Mathematik, Westf\"alische Wilhelms-Universit\"at M\"unster, Einsteinstr. 62, 48149 M\"unster, Germany.}
\email{martin.burger@wwu.de, teresa.esposito@uni-muenster.de, caterina.zeppieri@uni-muenster.de}



\begin{abstract}
We propose and study two variants of the Ambrosio-Tortorelli functional where the first-order penalization of the edge variable $v$ is replaced by a second-order term depending on the Hessian or on the Laplacian of $v$, respectively. 
We show that both the variants as above provide an elliptic approximation of the Mumford-Shah functional in the sense of $\Gamma$-convergence.

In particular the variant with the Laplacian penalization can be implemented without any difficulties compared to the standard Ambrosio-Tortorelli functional. The computational results indicate several advantages however. First of all, the diffuse approximation of the edge contours appears smoother and clearer for the minimizers of the second-order functional. { Moreover, the convergence of alternating minimization algorithms seems improved for the new functional.} We also illustrate the findings with several computational results.

\vspace{3pt}\noindent {\bf Keywords}:
Ambrosio-Tortorelli approximation, free-discontinuity problems, $\G$-convergence, variational image segmentation.

\vspace{3pt} \noindent {\it 2000 Mathematics Subject
Classification:} 49J45, 74G65, 68U10, 49M99.
\end{abstract}

\maketitle

%



\vspace*{-12pt}

\section{Introduction}

\noindent The minimization of the Mumford-Shah functional \cite{MumfordShah} is one of the most prominent and successful approaches in mathematical image processing, which has received considerable attention from a practical (see e.g. \cite{bar,cheng,tsai}), computational (see e.g. \cite{gobbino,hintermueller,tsai}), as well as theoretical point of view (see e.g. \cite{kawohl,MorelSolimini} and references therein). This approach led moreover to several fruitful variants of the original model such as the Chan-Vese functional \cite{ChanVese} or region-based variants with realistic noise modelling \cite{sawatzky}. 

According to the Mumford-Shah approach and to its later development proposed by De Giorgi and Ambrosio \cite{DGA}, 
the relevant contours of the objects in a picture are interpreted as the discontinuity set $S_{\bar u}$ of a function $\bar u$ approximating a given image datum and minimizing the functional
\begin{equation}\label{I:MS}
MS(u)=\a\int_\O |\nabla u|^2\dx +\b\,\mathcal{H}^{n-1}(S_u), 
\end{equation}
among all functions $u$ belonging to the space of special functions of bounded variation $SBV(\O)$. (The constants $\a,\b>0$ in \eqref{I:MS} are tuning parameters.)
The minimization of \eqref{I:MS} leads to a so-called free-discontinuity problem which is notoriously difficult to be solved numerically in a robust and efficient way.  An alternative approach frequently used to bypass these difficulties is to relax the minimization problem as above replacing the Mumford-Shah functional by a sequence of elliptic functionals which approximate $MS$ in a suitable variational sense. 
A first approximation of \eqref{I:MS} was studied by Ambrosio and Tortorelli \cite{AT90, AT92} who considered a sequence of functionals reminiscent of the Cahn-Hilliard approximation of the perimeter. More precisely, in \cite{AT92} the authors introduced the family of elliptic functionals
\begin{equation} \label{ATfunct}
AT_\e(u,v)=\a \int_\O (v^2+\eta_\e)|\nabla u|^2\dx +\frac{\b}{2}\int_\O\left(\frac{(v-1)^2}{\e}+\e|\nabla v|^2\right)\dx, 
\end{equation}
with $\e>0$ and $0<\eta_\e\ll \e$, defined on pair of functions $u,v \in W^{1,2}(\O)$, $v\nabla u\in L^{2}(\O;\R^n)$, where the additional variable $v$ encodes an approximation of the discontinuity set of $\bar u$. Indeed, loosely speaking, one expects that a minimizing $v_\e$ is close to $1$ in large regions of $\O$ and it deviates from $1$, being close to $0$, only in an $\e$-neighbourhood of $S_{\bar u}$ (where $\nabla u_\e$ tends to be very large). In this way we get that $v_\e$ approaches $1-\chi_{S_{\bar u}}$ as $\e\to 0$. This heuristic argument is made rigorous in \cite{AT90,AT92} using the language of $\Gamma$-convergence and following earlier ideas developed by Modica and Mortola in the seminal papers \cite{M,MM}.

Clearly the above approximation of the Mumford-Shah functional is not the only possible one and in particular different variants of \eqref{ATfunct} can be considered in order to enhance the Ambrosio-Tortorelli approximation scheme. In this perspective, a computational and practical improvement to the existing scheme would be desirable in some cases, e.g. when it is difficult to compute global minimizers of the original functional. These considerations also motivate the analysis carried out in the present paper. More precisely we are interested in replacing the term $\e|\nabla v|^2$ in \eqref{ATfunct} by a second-order term depending on the Hessian or on the Laplacian of $v$. These second-order penalizations are strongly related to some second-order Cahn-Hilliard-type functionals used to approximate the perimeter \cite{FMa, HPS} (see also the more recent \cite{CDMFL, CSZ}). 
  At a first glance a higher-order approximation seems counterintuitive, since one expects convergence to the first-order perimeter term in the limit. However, we shall observe in computational experiments that the second-order approximation may have several advantages. First of all the stronger smoothing behaviour of the second-order term leads to smoother approximations of $v$, which can lead to increased robustness in practice. In particular certain structures that are larger than typical noise but not yet of interest for the segmentation can be suppressed, e.g. freckles in the segmentation of the face in a portrait as illustrated in Figure \ref{lentigginifig}. Moreover an increase in robustness is visible in the computation as well, standard alternating minimization algorithms appear to converge faster for the second-order model and do not get stuck in undesired local minima, which is the case for the first-order model in a significant number of cases.	Another interesting aspect is that due to the missing maximum principle in higher order equations the optimal value of the variable $v$ is not bounded between zero and one anymore. In particular locations where $v$ is larger than one can certainly be identified as edges and due the specific shape of the optimal profile one can build two-sided approximations of the edge set in several cases.
	
	\begin{figure}[!!!h]
\includegraphics[width=0.2\textwidth,frame]{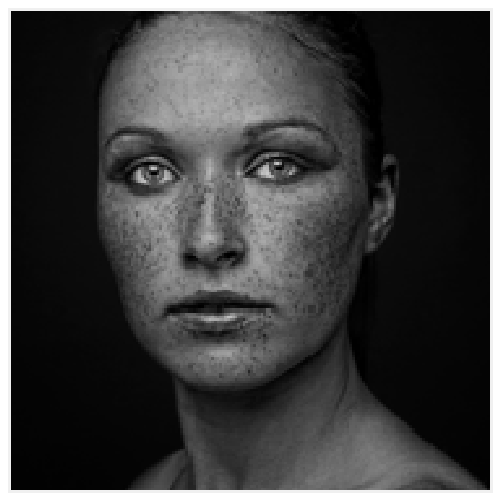}
\includegraphics[width=0.2\textwidth,frame]{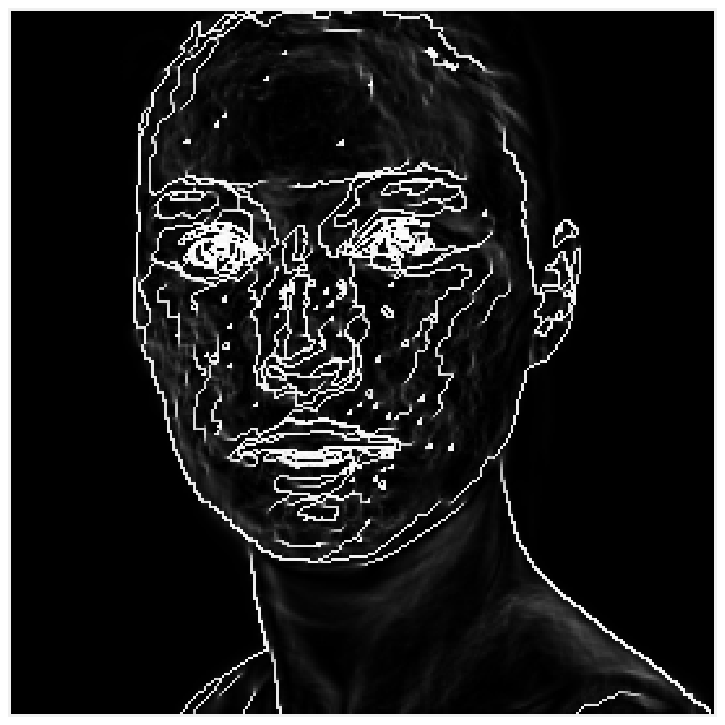}
\includegraphics[width=0.2\textwidth,frame]{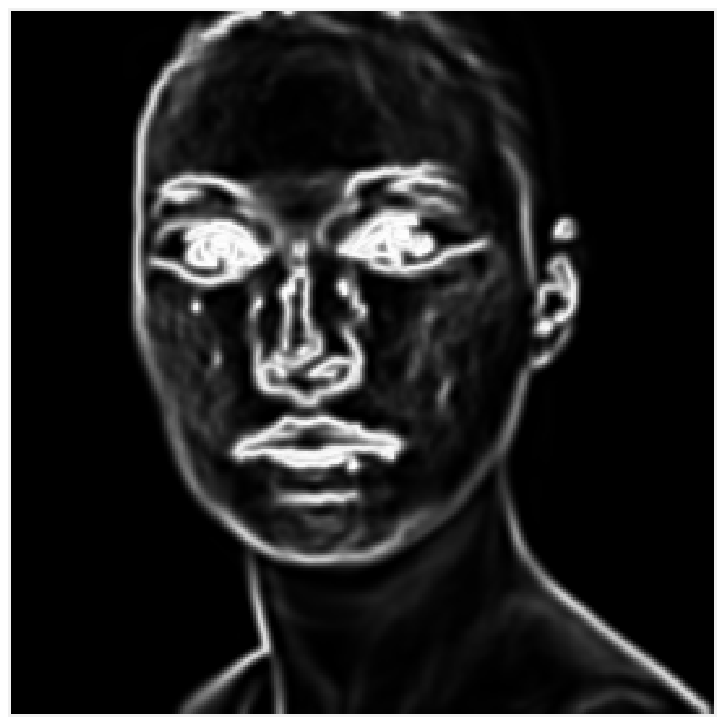}
\caption{Sisse image, courtesy of S\o ren Udby (from left to right): Image $g$, resulting $v$ in the Ambrosio-Tortorelli model, resulting $v$ in the second-order model. }
\label{lentigginifig}
\end{figure}

Then, assuming now that $v\in W^{2,2}(\O)$ (and dropping the ``lower-order terms'') we consider the functionals 
\begin{equation*}
F_\e(u,v)=\a\int_\O v^2|\nabla u|^2\dx +\frac{\b}{2\sqrt{2}}\int_\O\left(\frac{(v-1)^2}{\e}+{\e}^3|\nabla^2 v|^2 \right)\dx, 
\end{equation*}
as well as the functionals  
\begin{equation*}
E_\e(u,v)=\a\int_\O v^2|\nabla u|^2\dx +\frac{\b}{2\sqrt{2}}\int_\O\left(\frac{(v-1)^2}{\e}+{\e}^3|\Delta v|^2 \right)\dx,
\end{equation*}
for which we additionally assume $(v-1)\in W^{1,2}_0(\O)$. We show that when $\e\to 0$ both the families as above approximate $MS$ in the sense of $\Gamma$-convergence (see Theorem \ref{gconv-nD} and Theorem \ref{conv_lap} for the precise statements). 

In both cases the most delicate part in the proof is the lower bound inequality. The one-dimensional case is first considered and as for the Ambrosio-Tortorelli functional it turns out to contain the main features of the problem. In dimension one $F_\e$ and $E_\e$ clearly coincide (if we ignore the boundary condition, which in this case plays no role in the lower bound). When we estimate from below the energy contribution of a sequence $(u_\e,v_\e)$ with equi-bounded energy $F_\e$ (or $E_\e$), we first appeal to the Gagliardo-Nirenberg interpolation inequality to obtain the necessary {\it a priori} bound on the first derivative $v_\e'$. In its turn, this bound allows us to deduce that the limit $u$ is a piecewise Sobolev function, which thus has a finite number of discontinuities. Then the main difference with respect to the Ambrosio-Tortorelli analysis is that now the so-called ``Modica-Mortola trick'' cannot be applied. Therefore to find the minimal cost of a transition between $0$ and $1$, and occurring in an $\e$-layer around each discontinuity point of $u$, a careful analysis is needed (see Theorem \ref{t:min-pb} and Theorem \ref{gconv-1D}). Then, the upper bound inequality follows by an explicit construction which yields a recovery sequence $(u_\e,v_\e)$ satisfying  $(v_\e -1) \in W^{1,2}_0$.    

In dimension $n>1$ the two functionals $F_\e$ and $E_\e$ are in general different and the lower bound inequality is established in the two cases by means of two different arguments. Specifically, the form of the functionals $F_\e$ (and in particular the presence of the full Hessian) makes it possible to use a well-known integral-geometric argument, the so-called slicing procedure, which allows us to reduce the $n$-dimensional problem to the one-dimensional case and hence to conclude. 
On the other hand, when we deal with $E_\e$ we cannot exploit the one-dimensional reduction argument as above because of a symmetry breaking  due to the presence of the Laplacian. Therefore in this case a different procedure based on the blow-up method of Fonseca and M\"uller \cite{FM} is employed. In Theorem \ref{conv_lap} (see also Proposition \ref{lower_bd}) we show, however, that the symmetry breaking at $\e>0$ disappears in the limit. This is done again using Gagliardo-Nirenberg interpolation and then appealing to  standard elliptic regularity (where the boundary condition on $v$ plays a role). Then in both cases the upper bound inequality follows by a standard density argument and by explicit construction.
Finally, the $\G$-convergence results are complemented with the corresponding results about the convergence of the associated minimization problems (Theorem \ref{eqcoerc} and Theorem \ref{eqcoerc-1}).           

\section{Notation and preliminaries}

\noindent In this section we set a few notation and recall some preliminary results we employ in what follows.

Let $n\geq 1$; if not otherwise specified, throughout the paper $\O \subset \R^n$ denotes an open bounded set with Lipschitz boundary. 
The Lebesgue measure and the $k\hbox{-}$dimensional Hausdorff measure on $\R^n$ are denoted by $\mathcal{L}^n$ and $\mathcal{H}^k$, respectively. The scalar product of $x,y\in\mathbb{R}^n$ is denoted by $\langle x,y\rangle$ and the euclidean norm by $|x|$, whereas $A\cdot B$ denotes the product between two suitable matrices $A,B$. For each $\nu\in\mathbb{S}^{n-1}:=\{x\in\R^n\colon |x|=1\}$, $Q^\nu$ denotes the open unit cube centered at the origin with one face orthogonal to $\nu$; if $x_0\in\R^n$ and $\r>0$, then $Q_\r^\nu(x_0):=x_0+\r\,Q^\nu$. If $\nu$ belongs to the canonical basis of $\R^n$, we omit the dependence on $\nu$ and we simply write $Q_\r(x_0)=x_0+\r\,Q$, with $Q:=(-\frac12,\frac12)^n$.

Let $\mathcal{M}_b(\O)$ be the set of all bounded Radon measures on $\O$; if $\mu_k,\,\mu\in \mathcal{M}_b(\O)$, we say that $\mu_k\wto^*\mu$ weakly$^*$ in $\mathcal{M}_b(\O)$ if
\begin{equation*}
\int_\O\varphi\,d\mu_k \rightarrow\int_\O\varphi\,d\mu\qquad\text{ for every }\varphi\in C_0^0(\O).
\end{equation*}
Let $1\leq p \leq +\infty$ and $k\in \N$, we use standard notation for the Lebesgue and Sobolev spaces $L^p(\O)$ and $W^{k,p}(\O)$. 
For the general theory of special functions of bounded variation we refer the reader to the monograph \cite{AFP}; here we only recall some useful notation and definitions. For every $u\in SBV(\O)$, $\nabla u$ denotes the approximate gradient of $u$, $S_u$ the approximate discontinuity set of $u$, $\nu_u$ the generalized normal to $S_u$, and $u^+$ and $u^-$ are the traces of $u$ on $S_u$. We also consider the larger space of the generalized special functions of bounded variation on $\O$, $GSBV(\O)$, which is made of all the functions $u\in L^1(\O)$ whose truncation $u^M:=(u\land M)\vee(-M)$ belongs to $SBV(\O)$ for every $M\in\mathbb{N}$. We also consider the spaces
$$
SBV^2(\O) =\{u\in SBV(\O)\colon \nabla u\in L^2(\O) \text{ and } \mathcal{H}^{n-1}(S_u)<+\infty\}
$$
and
$$
GSBV^2(\O)=\{u\in GSBV(\O)\colon \nabla u\in L^2(\O) \text{ and } \mathcal{H}^{n-1}(S_u)<+\infty\}.
$$
We have 
$$
SBV^2(\O) \cap L^\infty(\O)=GSBV^2(\O) \cap L^\infty(\O).
$$
Since we heavily use it in what follows, we recall here the Gagliardo-Nirenberg interpolation inequality (see e.g. \cite[Theorems 4.14 and 4.15]{A}). 
\begin{proposition}\label{interpol-ineq} 
Let $U$ be a bounded open subset of $\R^n$ with Lipschitz boundary and let $\e_0>0$. Then there exists a constant $c_0(\e_0,U)>0$ such that
\begin{equation*}
c_0\,\e\int_U|\nabla v|^2\dx \leq \frac1\e\int_U v^2\dx+\e^3\int_U|\nabla^2 v|^2\dx,
\end{equation*}
for every $\e\in (0,\e_0]$ and for every $v\in W^{2,2}(U)$.
\end{proposition}
Moreover, we also recall two {\it a priori} estimates for the Laplace operator (see \cite[Theorem 1, Section 6.3.1]{E} and \cite[Theorem 3.1.2.1 and Remark 3.1.2.2]{Gr}) that we use in Section \ref{secondmodel}.
\begin{proposition}\label{ellipt-reg}
Let $U$ be a bounded open subset of $\R^n$. Then
\begin{itemize}
\item[(i)] for each open subset $V\subset\subset U$ there exists a constant $c(U,V)>0$ such that
\begin{equation*}
\|v\|_{W^{2,2}(V)}\leq c(U,V)\left(\|\Delta v\|_{L^2(U)}+\|v\|_{L^2(U)}\right),
\end{equation*}
for all $v\in W^{2,2}(U)$;
\item[(ii)] if in addition $U$ has $C^2$ boundary, then there exists a constant $c(U)>0$ such that
\begin{equation}\label{ellipt-est}
\|v\|_{W^{2,2}(U)}\leq c(U)\|\Delta v\|_{L^2(U)},
\end{equation}
for all $v\in W^{2,2}(U)\cap W^{1,2}_0(U)$.
\end{itemize}
\end{proposition}

Throughout the paper the parameter $\e$ varies in a strictly decreasing sequence of positive real numbers converging to zero.

Let $\a,\,\b>0$; we consider the functionals $\F_\e,\E_\e \colon L^1(\O)\times L^1(\O) \longrightarrow [0,+\infty]$ defined as
\begin{equation}\label{energy-hes}
\F_\e(u,v):=
\begin{cases} 
\ds \a\int_\O v^2|\nabla u|^2\dx +\frac{\b}{2\sqrt{2}}\int_\O\left(\frac{(v-1)^2}{\e}+{\e}^3|\nabla^2 v|^2 \right)\dx & \text{if}\;(u,v)\in W^{1,2}(\O)\times W^{2,2}(\O)\\
&\text{and}\; v\nabla u \in L^2(\O;\R^n), 
\cr\cr
+\infty & \text{otherwise},
\end{cases}
\end{equation}
and
\begin{equation}\label{energy-lap}
\E_\e(u,v):=
\begin{cases}
\ds \a\int_\O v^2 |\nabla u|^2\dx+\frac{\b}{2\sqrt{2}}\int_\O \left(\frac{(v-1)^2}{\e}+\e^3|\Delta v|^2\right)\dx & \text{if}\;(u,v)\in W^{1,2}(\O)\times W^{2,2}(\O)\\
&\text{and}\; v\nabla u \in L^2(\O;\R^n), 
\cr\cr
+\infty & \text{otherwise}.
\end{cases}
\end{equation}
%

We also consider the Ambrosio-Tortorelli functionals $\mathcal{AT}_\e\colon L^1(\O)\times L^1(\O) \longrightarrow [0,+\infty]$ 
\begin{equation}\label{f:AT}
\mathcal{AT}_\e(u,v):=
\begin{cases} 
\ds \a\int_\O v^2|\nabla u|^2\dx+\frac{\b}{2}\int_\O\left(\frac{(v-1)^2}{\e}+\e|\nabla v|^2\right)\dx & \text{if}\; (u,v)\in W^{1,2}(\O)\times W^{1,2}(\O)\\
&\text{and}\; v\nabla u \in L^2(\O;\R^n), 
\cr\cr
+\infty & \text{otherwise},
\end{cases}  
\end{equation}
and the Mumford-Shah functional $\MS \colon L^1(\O)\times L^1(\O) \longrightarrow [0,+\infty]$
\begin{equation}\label{Mu-Sh}
\MS(u,v):=
\begin{cases}
\ds \a\int_\O |\nabla u|^2\dx +\b\,\mathcal{H}^{n-1}(S_u) & \text{if}\; u\in GSBV^2(\O)\;\text{and}\;v=1\;\text{a.e.} ,\\
+\infty & \text{otherwise}.
\end{cases}
\end{equation}

\section{Optimal profile problem}
\noindent In this section we study the optimal profile problem
\begin{multline}\label{opt-prof}
\text{\it\textbf{m}}:= \inf\bigg\{\int_0^{+\infty} \left((f-1)^2+(f'')^2\right)\dt\colon f\in W^{2,2}_{\text{loc}}(0,+\infty),
\\
\quad f(0)=f'(0)=0,\;f(t)=1\; \text{if}\; t>M\; \text{for some}\; M>0 \bigg\}.
\end{multline}
The constant {\it\textbf{m}} represents the minimal cost, in terms of the unscaled, one-dimensional Modica-Mortola contribution in (\ref{energy-hes}) and (\ref{energy-lap}), for a transition from the value $0$ to the value $1$ on the positive real half-line.
\begin{theorem}\label{t:min-pb}
Let $\textbf{m}$ be as in (\ref{opt-prof}); then
\begin{equation}\label{min_prob}
\text{\it\textbf{m}}=\min\left\{\int_0^{+\infty} \left((f-1)^2+(f'')^2\right)\dt\colon f\in W^{2,2}_{\text{loc}}(0,+\infty),\;f(0)=f'(0)=0,\;\lim_{t\to +\infty}f(t)=1 \right\},
\end{equation}
moreover, $\textbf{m}=\sqrt{2}$.
\end{theorem}
\begin{proof}
Let
$$
\tilde{\text{\it\textbf{m}}}:=\inf\left\{\int_0^{+\infty} \left((f-1)^2+(f'')^2\right)\dt\colon f\in W^{2,2}_{\text{loc}}(0,+\infty),\;f(0)=f'(0)=0,\;\lim_{t\to +\infty}f(t)=1 \right\}.
$$
By solving the associated Euler-Lagrange equation it is easily shown that $\tilde{\text{\it\textbf{m}}}$ is attained at
\begin{equation}\label{minimizer}
f(t)= 1+\sqrt2 e^{-\frac{t}{\sqrt{2}}}\cos\left(\frac{t}{\sqrt{2}}+\frac\pi4\right)
\end{equation}
(see Figure \ref{opt-prof-fig}). Moreover a direct computation yields $\tilde{\text{\it\textbf{m}}}=\sqrt{2}$.

We clearly have $\tilde{\text{\it\textbf{m}}} \leq {\text{\it\textbf{m}}}$\,, then to achieve \eqref{min_prob} it only remains to show the opposite inequality\ie $\tilde{\text{\it\textbf{m}}} \geq {\text{\it\textbf{m}}}$. To this end, we suitably modify $f$ so to obtain a test function for ${\text{\it\textbf{m}}}$\,. Let $x_i\to +\infty$ as $i\to +\infty$; it is easy to check that
\begin{equation}\label{prop_f}
\lim_{i\to +\infty}f(x_i)=1,\quad\lim_{i\to +\infty}f'(x_i)=0,
\end{equation}
with $f$ is as in (\ref{minimizer}).

We introduce the auxiliary function $\mathcal{G}:\R^2\longrightarrow [0,+\infty)$ given by
\begin{equation*}
\mathcal{G}(w,z):=\inf\left\{\int_0^1\left((g-1)^2+(g'')^2\right)\dt\colon g\in C^2([0,1]),\;g(0)=w,\;g(1)=1,\;g'(0)=z,\;g'(1)=0\right\};
\end{equation*}
testing $\mathcal{G}$ with a third-degree polynomial satisfying the boundary conditions, one can easily show that
\begin{equation}\label{lim-g}
\lim_{(w,z)\to(1,0)}\mathcal{G}(w,z)=0.
\end{equation}
Fix $\eta>0$ and let $g$ be a test function for $\mathcal{G}(f(x_i),f'(x_i))$ such that
\begin{equation*}
\int_0^1\left((g-1)^2+(g'')^2\right)\dt\leq \mathcal{G}(f(x_i),f'(x_i))+\eta.
\end{equation*}
Set $g_i(t):=g(t-x_i)$ and define
\begin{equation*}
f_i(t):=
\begin{cases}
\ds f(t) & \text{if }\; 0\leq t\leq x_i ,\cr
g_i(t) & \text{if }\; x_i\leq t\leq x_i+1 ,\cr
1 & \text{if }\; t\geq x_i+1.
\end{cases}
\end{equation*}
Then $f_i$ is admissible for {\text{\it\textbf{m}}} and we have
\begin{eqnarray*}
\tilde{\text{\it\textbf{m}}} & = & \int_0^{+\infty}\left((f-1)^2+(f'')^2\right)\dt\geq \int_0^{x_i}\left((f-1)^2+(f'')^2\right)\dt\\
& = & \int_0^{+\infty}\left((f_i-1)^2+(f_i'')^2\right)\dt -\int_{x_i}^{x_i+1}\left((g_i-1)^2+(g_i'')^2\right)\dt\\
& \geq & {\text{\it\textbf{m}}} - \mathcal{G}(f(x_i),f'(x_i))-\eta.
\end{eqnarray*}
Invoking (\ref{prop_f}) and (\ref{lim-g}), we conclude by first letting $i\to +\infty$ and then $\eta \to 0^+$.
\end{proof}

\begin{figure}[!!!h]
\includegraphics[scale=0.5]{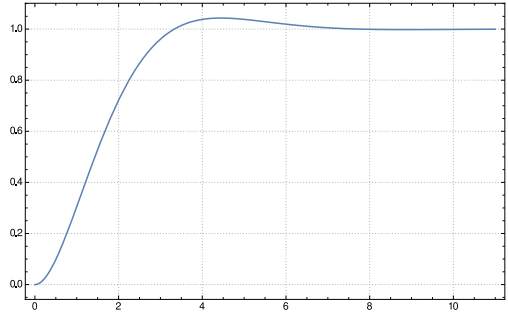}
\caption{The optimal profile $f$.}
\label{opt-prof-fig}
\end{figure}

\begin{remark}
{\rm For $d\in \R$, set
\begin{equation*}
{\text{\it\textbf{m}}}_d:=\min\left\{\int_0^{+\infty} \left((f-1)^2+(f'')^2\right)\dt\colon f\in W^{2,2}_{\text{loc}}(0,+\infty),\;f(0)=d,\;f'(0)=0,\;\lim_{t\to+\infty}f(t)=1\right\}.
\end{equation*}
A direct computation gives 
\begin{equation}\label{d-optprof}
{\text{\it\textbf{m}}}_d=\sqrt{2}(d-1)^2,
\end{equation} 
hence $\lim_{d\to 0}{\text{\it\textbf{m}}}_d=\sqrt{2}$.}
\end{remark}

\section{The first model}\label{sect3}
\noindent In this section we study the $\G\hbox{-}$convergence of the functionals $\F_\e$ defined in (\ref{energy-hes}). The one-dimensional case is considered first: this is the object of Subsection \ref{s:1D}. Then, Subsection \ref{n-dim} is devoted to the $n$-dimensional case, with $n\geq 2$. Finally, in Subsection \ref{s:min-pb} we deal with the equicoercivity of a suitable modification of $\F_\e$ and with the convergence of the associated minimization problems.

\medskip

In all that follows $c$ denotes a generic positive constant which may vary from line to line within the same formula. 

\subsection{The one-dimensional case}\label{s:1D}
Let $-\infty<a<b<+\infty$ and let $\F_\e\colon L^1(a,b)\times L^1(a,b)\longrightarrow [0,+\infty]$ be the one-dimensional version of the functional in (\ref{energy-hes})\ie
\begin{equation*}
\F_\e(u,v):=
\begin{cases}
\ds \a\int_a^b v^2(u')^2\dt+\frac{\b}{2\sqrt{2}}\int_a^b\left(\frac{(v-1)^2}{\e}+\e^3(v'')^2 \right)\dt & \text{if}\; (u,v)\in W^{1,2}(a,b)\times W^{2,2}(a,b) ,\\
+\infty & \text{otherwise}.
\end{cases}
\end{equation*}
Notice that in this case the condition $v u'\in L^2(a,b)$ is automatically satisfied for any pair $(u,v)\in W^{1,2}(a,b)\times W^{2,2}(a,b)$.

\medskip

We have the following $\G$-convergence result. 
\begin{theorem}\label{gconv-1D}
The sequence $(\F_\e)$ $\G\hbox{-}$converges, with respect to the $(L^1(a,b)\times L^1(a,b))\,\hbox{-}$topology, to the functional $\MS\colon L^1(a,b)\times L^1(a,b)\longrightarrow [0,+\infty]$ given by
\begin{equation*}
\MS(u,v):=
\begin{cases}
\ds \a\int_a^b (u')^2\dt +\b\,\#(S_u) & \text{if}\; u\in SBV^2(a,b)\;\text{and}\;v=1\;\text{a.e.} ,\\
+\infty & \text{otherwise},
\end{cases}
\end{equation*}
where $\#(S_u)$ denotes the number of discontinuity points of $u$ in $(a,b)$.
\end{theorem}
\begin{proof}

For the sake of notation, for any open set $U\subset (a,b)$ we consider the localized functionals
\begin{equation}\label{energy-1Dloc}
\F_\e(u,v,U):=
\begin{cases}
\ds \a\int_U v^2(u')^2\dt+\frac{\b}{2\sqrt{2}}\int_U\left(\frac{(v-1)^2}{\e}+\e^3(v'')^2 \right)\dt & \text{if}\;(u,v)\in W^{1,2}(U)\times W^{2,2}(U) ,\\
+\infty & \text{otherwise}.
\end{cases}
\end{equation}
We divide the proof into two steps in which we analyze separately the liminf-inequality and the limsup-inequality. 

\textit{Step 1: liminf-inequality}. Let $(u,v)\in L^1(a,b)\times L^1(a,b)$ and $(u_\e,v_\e)\subset L^1(a,b)\times L^1(a,b)$ be such that $(u_\e, v_\e)\rightarrow (u, v)$ in $L^1(a,b)\times L^1(a,b)$. We want to prove that
\begin{equation}\label{liminf-1D}
\liminf_{\e\to 0}\F_\e(u_\e,v_\e)\geq\MS(u,v).
\end{equation}
Clearly it is enough to consider the case 
$$
\liminf_{\e\to 0} \F_\e(u_\e,v_\e)<+\infty.
$$ 
Then, up to subsequences, $\|v_\e-1\|_{L^2(a,b)}\leq c\,\sqrt\e$, from which we immediately deduce that $v=1$ a.e. in $(a,b)$.

According to Proposition \ref{interpol-ineq} there exists a positive constant $c_0>0$ such that for $\e$ sufficiently small we have
\begin{equation*}
c_0\,\int_a^b \e (v_\e')^2\dt\leq \int_a^b \frac{(v_\e-1)^2}{\e}\dt+\int_a^b \e^3(v_\e'')^2\dt.
\end{equation*}
Therefore
\begin{equation*}
\int_a^b\left(\frac{(v_\e-1)^2}{\e}+\e\,(v_\e')^2\right)\dt\leq c \quad \text{for $\e>0$ sufficiently small};
\end{equation*} 
hence we can apply \cite[Lemma 6.2 and Remark 6.3]{B}, with $Z=\{1\}$ and $W(s)=(s-1)^2$ to conclude that there exists a finite set $S$ such that, for every fixed open set $U \subset \subset (a,b)\setminus S$ and for $\e>0$ sufficiently small, $1/2<v_\e<3/2$ on $U$. For every such fixed $U$ we have
\begin{equation*}
\frac{\a}{4}\,\sup_{\e>0}\int_U (u_\e')^2\dt\leq\a\,\sup_{\e>0}\int_a^b v_\e^2 (u_\e')^2\dt<+\infty;
\end{equation*}
thus $u\in W^{1,2}(U)$ and $u_\e\wto u$ in $W^{1,2}(U)$. Moreover, since $v_\e\to 1$ in $L^2(U)$ and $u_\e'\wto u'$ in $L^2(U)$, then $u_\e'v_\e\wto u'$ in $L^1(U)$; hence we have 
\begin{equation}\label{est1}
\a\int_U(u')^2\dt\leq\a\,\liminf_{\e\to 0}\int_U v_\e^2(u_\e')^2\dt\leq\a\,\liminf_{\e\to 0}\int_a^b v_\e^2(u_\e')^2\dt.
\end{equation}
By the arbitrariness of $U$, (\ref{est1}) can be rewritten as
\begin{equation}\label{est2}
\a\int_{(a,b)\setminus(S+[-\eta,\eta])}(u')^2\dt\leq\a\,\liminf_{\e\to 0}\int_{(a,b)\setminus(S+[-\eta,\eta])} v_\e^2(u_\e')^2\dt\leq\a\,\liminf_{\e\to 0}\int_a^b v_\e^2(u_\e')^2\dt,
\end{equation}
for every $\eta>0$. This allows us to conclude that $u\in SBV^2(a,b)$ and $S_u\subset S$.

Let $N:=\#(S_u)$, $S_u:=\{t_1,\dots,t_N\}$ with $t_1<t_2<\dots<t_N$, and consider pairwise disjoint intervals $I_i=(a_i,b_i)\subset(a,b)$ with $t_i\in I_i$; we want to show that, for every $i\in\{1,\dots,N\}$,
\begin{equation}\label{liminf-i}
\liminf_{\e\to 0}\F_\e(u_\e,v_\e,I_i)\geq \b.
\end{equation}
To this end, fix $i\in\{1,\dots,N\}$. Let $t_i\in I_i'\subset\subset I_i$ and set $m_i:=\liminf_{\e\to 0}\,\inf_{t\in I_i'}\,(v_\e(t))^2$. Since $I_i'$ contains a discontinuity point of $u$, we want to show that it is convenient for $v_\e$ to be ``close'' to zero in $I_i'$\ie that $m_i=0$. We argue by contradiction assuming that $m_i>0$. Then for every $t\in I_i'$, we have
\begin{equation*}
\liminf_{\e\to 0}\,\frac{(v_\e(t))^2}{m_i}\geq 1,
\end{equation*}
therefore
\begin{equation*}
\int_{I_i'}(u_\e')^2\dt\leq\frac{1}{m_i}\int_{I_i'}v_\e^2(u_\e')^2\dt\leq\frac{\F_\e(u_\e,v_\e)}{m_i}\leq c\,,
\end{equation*}
so that $u_\e\wto u$ in $W^{1,2}(I_i')$ and $S_u\cap I_i'=\emptyset$ which contradicts the hypothesis $t_i\in I_i'$. Hence, we must have $m_i=0$. As a consequence there exists a sequence $(s_\e^i)\subset I_i'$ such that 
\begin{equation}\label{lim-vs}
\lim_{\e\to 0}v_\e(s_\e^i)=0.
\end{equation}
On the other hand, up to subsequences (not relabelled), $v_\e\to 1$ a.e. in $(a,b)$ as $\e\to 0$. Therefore, there exist $\tilde{r}_\e^i$, $r_\e^i\in I_i$ such that $\tilde{r}_\e^i<s_\e^i<r_\e^i$ and
\begin{equation}\label{lim-vr}
\lim_{\e\to 0} v_\e(\tilde{r}_\e^i)=1,\quad \lim_{\e\to 0} v_\e(r_\e^i)=1.
\end{equation}
Moreover, again appealing to the interpolation inequality Proposition \ref{interpol-ineq} we may deduce that $\e v_\e'\to 0$ in $L^1(a,b)$. Indeed
\begin{eqnarray*}
\int_a^b\e |v_\e'|\dt & \leq & \e^{1/2}(b-a)^{1/2}\left(\int_a^b \e\,(v_\e')^2\dt\right)^{1/2}\\
& \leq & \e^{1/2} (b-a)^{1/2}\left(\frac{1}{c_0}\,\int_a^b\left(\frac{(v_\e-1)^2}{\e}+\e^3(v_\e'')^2\right)\dt\right)^{1/2}\\
&\leq& c\,\e^{1/2} \to 0\quad \text{as $\e\to 0$.}
\end{eqnarray*}
Then up to subsequences (not relabelled) it is not restrictive to suppose that
\begin{equation*}
\lim_{\e\to 0}\e\, v_\e'(\tilde{r}_\e^i) =0 \quad \text{and}\quad\lim_{\e\to 0}\e\, v_\e'(r_\e^i)=0.
\end{equation*}
Let $\tilde{s}_\e^i$ be a minimum point for $v_\e$ in $[\tilde{r}_\e^i,r_\e^i]$. In view of (\ref{lim-vs}) and (\ref{lim-vr}) we deduce that $\tilde{s}_\e^i\in(\tilde{r}_\e^i,r_\e^i)$, for $\e$ sufficiently small. Then, since $v_\e\in C^1([\tilde{r}_\e^i,r_\e])$ we have $v_\e'(\tilde{s}_\e^i)=0$. Moreover $v_\e(\tilde{s}_\e^i)\leq v_\e(s_\e^i)<v_\e(r_\e^i)$, hence (\ref{lim-vr}) implies that $v_\e(\tilde{s}_\e^i)<1$ for $\e$ small.

We have
\begin{eqnarray}
\F_\e(u_\e,v_\e,I_i) & \geq & \frac{\b}{2\sqrt{2}}\int_{I_i}\left(\frac{(v_\e-1)^2}{\e}+\e^3(v_\e'')^2 \right)\dt\nonumber\\
& \geq & \frac{\b}{2\sqrt{2}}\,\left(\int_{\tilde{r}_\e^i}^{\tilde{s}_\e^i}\left(\frac{(v_\e-1)^2}{\e}+\e^3(v_\e'')^2 \right)\dt + \int_{\tilde{s}_\e^i}^{r_\e^i}\left(\frac{(v_\e-1)^2}{\e}+\e^3(v_\e'')^2 \right)\dt\right).\label{est3}
\end{eqnarray}
We now estimate from below the term
\begin{equation*}
\int_{\tilde{s}_\e^i}^{r_\e^i}\left(\frac{(v_\e-1)^2}{\e}+\e^3(v_\e'')^2 \right)\dt = \int_0^{\frac{r_\e^i-\tilde{s}_\e^i}{\e}}\left((w_\e-1)^2+(w_\e'')^2\right)\dz,
\end{equation*}
where $z=(t-\tilde{s}_\e^i)/\e$ and $w_\e(z):=v_\e(\e z+\tilde{s}_\e^i)$. To this end let $g_{\e,i}\in C^2([0,1])$ be an admissible function for $\mathcal{G}(v_\e(r_\e^i),\e\, v_\e'(r_\e^i))$\ie $g_{\e,i}(0)=v_\e(r_\e^i)$, $g_{\e,i}(1)=1$, $g'_{\e,i}(0)=\e\, v'_\e(r_\e^i)$, $g'_{\e,i}(1)=0$.  By construction 
\begin{equation*}
\lim_{\e\to 0}g_{\e,i}(0)=1 \quad\text{and}\quad\lim_{\e\to 0}g'_{\e,i}(0)= 0,
\end{equation*}
hence by (\ref{lim-g}) we infer
\begin{equation}\label{lim-gv}
\lim_{\e\to 0}\mathcal{G}(v_\e(r_\e^i),\e\, v_\e'(r_\e^i))=0.
\end{equation}
Let $(\tilde{v}_{\e,i})$ be the sequence defined as
\begin{equation*}
\tilde{v}_{\e,i}(z):=
\begin{cases}
w_\e(z) & \text{if }\; 0\leq z\leq \ds\frac{r_\e^i-\tilde{s}_\e^i}{\e} ,\cr
g_{\e,i}\left(z-\ds\frac{r_\e^i-\tilde{s}_\e^i}{\e}\right) & \text{if }\;\ds\frac{r_\e^i-\tilde{s}_\e^i}{\e}\leq z\leq\frac{r_\e^i-\tilde{s}_\e^i}{\e}+1 ,\cr
1 & \text{if }\;z\geq\ds\frac{r_\e^i-\tilde{s}_\e^i}{\e}+1.
\end{cases}
\end{equation*}
By definition of $g_{\e,i}$ it follows that $(\tilde{v}_{\e,i})\subset W^{2,2}_{\text{loc}}(0,+\infty)$. Since $(\tilde{v}_{\e,i})$ is a test function for $\textit{\textbf{m}}_{v_\e(\tilde{s}_\e^i)}$ (with $\textit{\textbf{m}}_{v_\e(\tilde{s}_\e^i)}$ as in (\ref{d-optprof}) with $d=v_\e(\tilde{s}_\e^i)$), we have
\begin{eqnarray*}
\int_0^{\frac{r_\e^i-\tilde{s}_\e^i}{\e}}\left((w_\e-1)^2+(w_\e'')^2\right)\dz  & = & \int_0^{+\infty} \left((\tilde{v}_{\e,i}-1)^2+(\tilde{v}_{\e,i}'')^2\right)\dz-\mathcal{G}(v_\e(r_\e^i),\e\, v_\e'(r_\e^i))\\
& \geq & \textit{\textbf{m}}_{v_\e(\tilde{s}_\e^i)}-\mathcal{G}(v_\e(r_\e^i),\e\, v_\e'(r_\e^i)).
\end{eqnarray*}
A similar argument applies to the term 
\begin{equation*}
\int_{\tilde{r}_\e^i}^{\tilde{s}_\e^i}\left(\frac{(v_\e-1)^2}{\e}+\e^3(v_\e'')^2 \right)\dt
\end{equation*}
in (\ref{est3}). Since $\textit{\textbf{m}}_{v_\e(\tilde{s}_\e^i)}=\sqrt{2}(v_\e(\tilde{s}_\e^i)-1)^2$ and $v_\e(\tilde{s}_\e^i)\leq v_\e(s_\e^i)<1$, we have
\begin{eqnarray*}
\mathcal{F_\e}(u_\e,v_\e,I_i) & \geq & \frac{\b}{\sqrt{2}}\,\textit{\textbf{m}}_{v_\e(\tilde{s}_\e^i)} -\frac{\b}{\sqrt{2}}\,\mathcal{G}(v_\e(r_\e^i),\e\, v_\e'(r_\e^i)))\geq \frac{\b}{\sqrt{2}}\,\min_{d\leq v_\e(s_\e^i)}\textit{\textbf{m}}_d-\frac{\b}{\sqrt{2}}\,\mathcal{G}(v_\e(r_\e^i),\e\, v_\e'(r_\e^i)))\\
& = & \frac{\b}{\sqrt{2}}\,\textit{\textbf{m}}_{v_\e(s_\e^i)}-\frac{\b}{\sqrt{2}}\,\mathcal{G}(v_\e(r_\e^i),\e\, v_\e'(r_\e^i)))=\b\,(v_\e(s_\e^i)-1)^2-\frac{\b}{\sqrt{2}}\,\mathcal{G}(v_\e(r_\e^i),\e\, v_\e'(r_\e^i)))
\end{eqnarray*}
and (\ref{liminf-i}) follows from (\ref{lim-gv}) letting $\e\to 0$.

Since the intervals $I_i$ are pairwise disjoint, gathering (\ref{est2}) and (\ref{liminf-i}), we get
\begin{eqnarray*}
\liminf_{\e\to 0}\F_\e(u_\e,v_\e) & \geq & \a\,\liminf_{\e\to 0}\int_a^b v_\e^2(u_\e')^2\dt +\frac{\b}{2\sqrt{2}}\,\liminf_{\e\to 0}\int_a^b \left(\frac{(v_\e-1)^2}{\e}+\e^3(v_\e'')^2 \right)\dt\\
& \geq & \a\int_{(a,b)\setminus(S+[-\eta,\eta])}(u')^2\dt+\b\,\#(S_u).
\end{eqnarray*}
Hence the liminf-inequality (\ref{liminf-1D}) follows from the arbitrariness of $\eta$.

\medskip

\textit{Step 2: limsup-inequality}. Let $u\in SBV^2(a,b)$, $S_u:=\{t_1,\dots,t_N\}$ with $t_1<t_2<\dots<t_N$, and set $t_0:=a$, $t_{N+1}:=b$, $\d_0:=\min_i\{t_{i+1}-t_i\colon i=0,\dots,N\}$. For $i=1,\dots,N$, define $I_i:=\left[\frac{t_{i-1}+t_i}{2},\frac{t_i+t_{i+1}}{2}\right]$.

We now construct a recovery sequence for the $\G\hbox{-}$limit. 

Let $\eta>0$; there exist a function $f_\eta\in W^{2,2}_{\text{loc}}(0,+\infty)$ and a constant $M_\eta >0$ such that $f_\eta(0)=f_\eta'(0)=0$, $f_\eta(t)=1$ for all $ t>M_\eta$, and
\begin{equation}\label{est4}
\int_0^{M_\eta} \left((f_\eta-1)^2+(f_\eta'')^2\right)\dt \leq \sqrt{2}+\eta.
\end{equation}
Let $\xi_\e>0$ be such that $\xi_\e/\e\to 0$ as $\e\to 0$, then for $\e$ sufficiently small $\xi_\e<\frac{\d_0}{2}$. For $i=1,\dots,N$ let $\varphi_\e^i$ be a cut-off function between $\left(t_i-\frac{\xi_\e}{2},t_i+\frac{\xi_\e}{2}\right)$ and $(t_i-\xi_\e,t_i+\xi_\e)$\ie $\varphi_\e^i\in C_c^\infty(t_i-\xi_\e,t_i+\xi_\e)$, $0\leq\varphi_\e^i\leq 1$ and $\varphi_\e^i\equiv 1$ on $\left(t_i-\frac{\xi_\e}{2},t_i+\frac{\xi_\e}{2}\right)$. Define the sequence
\begin{equation*}
u_\e(x):=u(x)\left(1-\sum_{i=1}^N\varphi_\e^i(x)\chi_{I_i}(x)\right);
\end{equation*}
then $(u_\e)\subset W^{1,2}(a,b)$ and $u_\e\to u$ in $L^1(a,b)$ by the dominated convergence Theorem.

Fix $T>M_\eta$; then $\frac{\d_0-2\xi_\e}{2\e}>T$, for $\e$ small. Define the sequence
\begin{equation*}
v_{\e}(t):=
\begin{cases}
0 & \text{if }\; |t-t_i|\leq\xi_\e ,\cr
\ds f_\eta\left(\frac{|t-t_i|-\xi_\e}{\e}\right) & \text{if }\;\xi_\e\leq|t-t_i|\leq\xi_\e+\e T ,\cr
1 & \text{if }\;t\in(\{|t-t_i|\geq\xi_\e+\e T\}\cap I_i)\cup\left[a,\frac{a+t_1}{2}\right]\cup\left[\frac{t_N+b}{2},b\right],
\end{cases}
\end{equation*}
for $i=1,\cdots,N$, where $f_\eta$ is as in (\ref{est4}).
It is immediate to check that $(v_\e)\subset  W^{2,2}(a,b)$ and $v_\e\to 1$ in $L^1(a,b)$. Moreover for every $i=1,\dots, N$, we have
\begin{equation}\label{c:ls}
\int_{I_i} v_\e^2 (u_\e')^2\dt = \int_{I_i\setminus [t_i-\xi_\e,t_i+\xi_\e]} v_\e^2 (u')^2\dt
\end{equation}
and
\begin{eqnarray}\label{est5}\nonumber
&\quad&\ds\int_{I_i}\left(\frac{(v_\e-1)^2}{\e}+\e^3(v_\e'')^2\right)\dt\\
\qquad & = & \int_{t_i-\xi_\e-\e T}^{t_i-\xi_\e}\left(\frac{1}{\e}\left(f_\eta\left(\frac{t_i-t-\xi_\e}{\e}\right)-1\right)^2+\e^3\left(f_\eta''\left(\frac{t_i-t-\xi_\e}{\e}\right)\right)^2\right)\dt\nonumber\\
& & +\int_{t_i+\xi_\e}^{t_i+\xi_\e+\e T}\left(\frac{1}{\e}\left(f_\eta\left(\frac{t-t_i-\xi_\e}{\e}\right)-1\right)^2+\e^3\left(f_\eta''\left(\frac{t-t_i-\xi_\e}{\e}\right)\right)^2\right)\dt+\int_{t_i-\xi_\e}^{t_i+\xi_\e}\frac{1}{\e}\dt\nonumber\\
& = & 2\int_0^T\left((f_\eta(z)-1)^2+(f_\eta''(z))^2\right)\dz+2\frac{\xi_\e}{\e}\nonumber\\
& \leq & 2\sqrt{2}+\eta+2\frac{\xi_\e}{\e}.
\end{eqnarray}
Therefore gathering (\ref{c:ls}) and (\ref{est5}) gives
\begin{eqnarray*}
\F_\e(u_\e,v_\e) &=& \a\sum_{i=1}^N\int_{I_i} v_\e^2(u')^2\dt+\frac{\b}{2\sqrt{2}}\,\sum_{i=1}^N\int_{I_i}\left(\frac{(v_\e-1)^2}{\e}+\e^3(v_\e'')^2\right)\dt
\\
&\leq &\a\sum_{i=1}^N \int_{I_i\setminus [t_i-\xi_\e,t_i+\xi_\e]} v_\e^2 (u')^2\dt+ \b N +c \frac{\xi_\e}{\e}+\tilde\eta. 
\end{eqnarray*}
Finally, invoking the dominated convergence Theorem yields
\begin{equation*}
\limsup_{\e\to 0}\F_\e(u_\e,v_\e)\leq\a\int_a^b(u')^2\dt+\b\,\#(S_u)+\tilde\eta
\end{equation*}
and thus the thesis.
\end{proof}

\subsection{The $n$-dimensional case}\label{n-dim}
Let $n\geq 2$ and let $\F_\e$ be the functional defined as in (\ref{energy-hes}). The following $\G$-convergence result holds true. 
\begin{theorem}\label{gconv-nD}
The sequence $(\F_\e)$ $\G\hbox{-}$converges, with respect to the $(L^1(\O)\times L^1(\O))$-topology, to the functional $\MS\colon L^1(\O)\times L^1(\O)\longrightarrow [0,+\infty]$ given by
\begin{equation*}
\MS(u,v):=
\begin{cases}
\ds \a\int_\O|\nabla u^2|\dx +\b\,\mathcal{H}^{n-1}(S_u) & \text{if}\; u\in GSBV^2(\O)\;\text{and}\;v=1\;\text{a.e.} ,\\
+\infty & \text{otherwise},
\end{cases}
\end{equation*}

\end{theorem}
\begin{proof}
As for the one-dimensional case, if $A$ is an open, $A\subset\O$, we set
\begin{equation*}
\F_\e(u,v,A):=
\begin{cases}
\ds \a\int_A v^2|\nabla u|^2\dx+\frac{\b}{2\sqrt{2}}\int_A\left(\frac{(v-1)^2}{\e}+{\e}^3|\nabla^2 v|^2 \right)\dx & \text{if}\; (u,v)\in W^{1,2}(A)\times W^{2,2}(A)\\
&\text{and}\; v\nabla u\in L^2(A;\R^n),
\cr\cr
+\infty & \text{otherwise},
\end{cases}
\end{equation*}
and
\begin{equation*}
\MS(u,v,A):=
\begin{cases}
\ds \a\int_A |\nabla u|^2\dx +\b\,\mathcal{H}^{n-1}(S_u\cap A) & \text{if}\; u\in GSBV^2(A)\;\text{and}\;v=1\;\text{a.e.} ,\\\\
+\infty & \text{otherwise}.
\end{cases}
\end{equation*}
We divide the proof into two steps in which we analyze separately the liminf-inequality and the limsup-inequality. 

\textit{Step 1: liminf-inequality}.
Let $A$ be an open subset of $\O$. We recover the lower bound
\begin{equation}\label{liminf-nD}
\G\hbox{-}\liminf_{\e\to 0} \F_\e(u,1,A)\geq\MS(u,1,A)
\end{equation}
from the one-dimensional case, by using the slicing method. To this end we start recalling some useful notation. For each $\xi\in\mathbb{S}^{n-1}$ we consider the hyperplane through the origin and orthogonal to $\xi$,
\begin{equation*}
\Pi^\xi:=\{x\in\R^n\colon \langle x,\xi\rangle=0\},
\end{equation*}
and, for every $y\in\Pi^\xi$, we consider the one-dimensional set
\begin{equation*}
A_{\xi,y}:=\{t\in\R\colon y+t\,\xi\in A\}
\end{equation*}
and the one-dimensional functions $u_{\xi,y}$, $v_{\xi,y}$ defined by
\begin{equation*}
u_{\xi,y}(t):=u(y+t\,\xi),\qquad v_{\xi,y}(t):=v(y+t\,\xi).
\end{equation*}
By Fubini's Theorem we have
\begin{eqnarray*}
&\;& \F_\e(u,v,A)\\
& = & \int_{\Pi^\xi}\int_{A_{\xi,y}}\bigg(\a\,v^2(y+t\xi)|\nabla u(y+t\xi)|^2\\
& & \qquad\qquad\quad+\frac{\b}{2\sqrt{2}}\,\left(\frac{(v(y+t\xi)-1)^2}{\e}+\e^3|\nabla^2 v(y+t\xi)|^2\right)\bigg)\dt\dH\\
& \geq & \int_{\Pi^\xi}\int_{A_{\xi,y}}\bigg(\a\,v^2(y+t\xi)|\langle\nabla u(y+t\xi),\xi\rangle|^2\\
& & \qquad\qquad\quad+\frac{\b}{2\sqrt{2}}\,\left(\frac{(v(y+t\xi)-1)^2}{\e}+\e^3|\langle\nabla^2 v(y+t\xi)\xi,\xi\rangle|^2\right)\bigg)\dt\dH\\
& = & \int_{\Pi^\xi}\F_\e^{\xi,y}(u_{\xi,y},v_{\xi,y},A_{\xi,y})\dH,
\end{eqnarray*}
where for all $\xi$ and $y$, $\F_\e^{\xi,y}$ are the one-dimensional functionals defined as in (\ref{energy-1Dloc}).

Let $(u,v)\in L^1(\O)\times L^1(\O)$ and let $(u_\e,v_\e)\subset L^1(\O)\times L^1(\O)$ be such that $(u_\e,v_\e) \rightarrow (u, v)$ in $L^1(\O)\times L^1(\O)$, and 
\begin{equation}\label{bound-nD}
\liminf_{\e\to 0}\F_\e(u,v,A)<+\infty,
\end{equation} 
then, $v=1$ a.e. in $A$. By Fubini's Theorem and Fatou's Lemma, 
\begin{eqnarray*}
0 & = & \lim_{\e\to 0}\int_A|u_\e-u|\dx\\
& = & \lim_{\e\to 0}\int_{\Pi^\xi}\int_{A_{\xi,y}}|(u_\e)_{\xi,y}-u_{\xi,y}|\dt\dH\\
& \geq & \int_{\Pi^\xi}\liminf_{\e\to 0}\|(u_\e)_{\xi,y}-u_{\xi,y}\|_{L^1(A_{\xi,y})}\dH;
\end{eqnarray*}
hence $(u_\e)_{\xi,y}\to u_{\xi,y}$ in $L^1(A_{\xi,y})$ for $\mathcal{H}^{n-1}\hbox{-}$a.e. $y\in\Pi^{\xi}$ and analogously $(v_\e)_{\xi,y}\to 1$ in $L^1(A_{\xi,y})$ for $\mathcal{H}^{n-1}\hbox{-}$a.e. $y\in\Pi^{\xi}$. Therefore, appealing to the one-dimensional result Theorem \ref{gconv-1D} we have that $u_{\xi,y}\in SBV^2(A_{\xi,y})$, for $\mathcal{H}^{n-1}\hbox{-}$a.e. $y\in\Pi^{\xi}$ and
\begin{eqnarray}\label{slicing}
\liminf_{\e\to 0}\F_\e(u_\e,v_\e,A) & \geq & \liminf_{\e\to 0}\int_{\Pi^\xi}\F^{\xi,y}_\e((u_\e)_{\xi,y},(v_\e)_{\xi,y},A_{\xi,y})\dH\nonumber\\
& \geq & \int_{\Pi^\xi}\liminf_{\e\to 0}\F^{\xi,y}_\e((u_\e){\xi,y},(v_\e)_{\xi,y},A_{\xi,y})\dH\nonumber\\
& \geq & \int_{\Pi^\xi}\left(\int_{A_{\xi,y}}\a\,|u'_{\xi,y}|^2\dt+\b\,\#(S_{u_{\xi,y}}\cap A_{\xi,y})\right)\dH.
\end{eqnarray}
Let $M\in \N$ and consider the truncated functions $u_M:=(-M)\vee (u\land M)$; by (\ref{bound-nD}) and (\ref{slicing}) we have
\begin{eqnarray*}
&\qquad& \int_{\Pi^\xi}|D(u_M)_{\xi,y}|(A_{\xi,y})\dH \\
\qquad & = & \int_{\Pi^\xi}\left(\int_{A_{\xi,y}}|(u_M)'_{\xi,y}|\dt+|(u_M)^+_{\xi,y}-(u_M)^-_{\xi,y}|\#(S_{(u_M)_{\xi,y}})\right)\dH\\
& \leq & \int_{\Pi^\xi}\left(c\,\left(\int_{A_{\xi,y}}|(u_M)'_{\xi,y}|^2\dt\right)^{1/2}+2M\,\#(S_{(u_T)_{\xi,y}})\right)\dH<+\infty.
\end{eqnarray*}
Thus, by virtue of \cite[Theorem 4.1(b)]{B1} we conclude that $u_M\in SBV^2(A)$ for every $M\in \N$, hence the lower semicontinuity of the Mumford-Shah functional entails $u\in GSBV^2(A)$. Moreover, by (\ref{slicing}) and \cite[Theorem 4.1(a)]{B1}, we infer
\begin{equation}\label{lim-xi}
\G\hbox{-}\liminf_{\e\to 0}\F_\e(u,v,A) \geq \a\int_A |\langle\nabla u, \xi\rangle|^2\dx+\b\int_{A\cap S_{u}}|\langle \xi,\nu_{u}\rangle|\,d\mathcal{H}^{n-1}.
\end{equation}
In particular, since (\ref{lim-xi}) holds for each $\xi\in\mathbb{S}^{n-1}$, we have
\begin{equation*}
\G\hbox{-}\liminf_{\e\to 0}\F_\e(u,v,A) \geq \sup_{\xi\in\mathbb{S}^{n-1}}\left\{\a\int_A |\langle\nabla u, \xi\rangle|^2\dx+\b\int_{A\cap S_{u}}|\langle \xi,\nu_{u}\rangle|\,d\mathcal{H}^{n-1}\right\}.
\end{equation*}
The previous estimate can be improved to (\ref{liminf-nD}) by means of a measure theory lemma. Observe that for $u\in GSBV^2(A)$ and $v=1$ the set function $\mu(A):=\G\hbox{-}\liminf_{\e\to 0}\F_\e(u,v,A)$ is superadditive on disjoint open sets. Then, applying \cite[Lemma 15.2]{B} with
\begin{equation*}
\lambda:=\mathcal{L}^n+\mathcal{H}^{n-1}\llcorner S_u,
\end{equation*}
and
\begin{equation*}
\psi_i(x):=
\begin{cases}
\a\,|\langle\nabla u,\xi_i\rangle|^2 & \text{if}\; x\notin S_u ,\\
\b\,|\langle\xi_i,\nu_u\rangle| & \text{if}\; x\in S_u,
\end{cases}
\end{equation*}
where $(\xi_i)$ is a dense sequence in $\mathbb{S}^{n-1}$, we obtain that
\begin{eqnarray*}
\G\hbox{-}\liminf_{\e\to 0}\F_\e(u,v,A)  &\geq&  \int_A\sup_i \psi_i\,d\lambda\\
&=& \a\int_A |\nabla u|^2\dx+\b\,\mathcal{H}^{n-1}(S_u\cap A)\\
&=&\MS(u,v,A).
\end{eqnarray*}
Finally the liminf-inequality follows taking $A=\O$.

\medskip

\textit{Step 2: limsup-inequality}. We want to show that 
\begin{equation}\label{limsup-nD}
\G\hbox{-}\limsup_{\e\to 0}\F_\e (u,v)\leq \MS(u,v),
\end{equation}
whenever $u\in GSBV^2(\O)$ and $v=1$ a.e.\ in $\O$.
To this end it is useful to recall the density result \cite[Theorem $3.9$ and Corollary $3.11$]{C}. Let $u\in GSBV^2(\O)$ then there exists a sequence $(u_j)\subset SBV^2(\O)$ satisfying the following properties:
\begin{enumerate}
\item\label{ess-clos} $S_{u_j}$ is essentially closed\ie $\mathcal{H}^{n-1}(\O\cap(\overline{S_{u_j}}\setminus S_{u_j}))=0$;
\item $\overline{S_{u_j}}$ is the intersection of $\O$ with a finite number of pairwise disjoint closed and convex sets, each contained in a $(n-1)\hbox{-}$dimensional hyperplane, and whose (relative) boundaries are $C^\infty$;
\item\label{reg} $u_j\in W^{k,\infty}(\O\setminus\overline{S_{u_j}})$ for every $k\in\mathbb{N}$;
\item\label{conv-u} $\|u_j-u\|_{L^2(\O)}\rightarrow 0$ as $j\rightarrow+\infty$;
\item\label{conv-grad} $\|\nabla u_j-\nabla u\|_{L^{2}(\O)}\rightarrow 0$ as $j\rightarrow+\infty$;
\item\label{conv-hausd} $|\mathcal{H}^{n-1}(S_{u_j})-\mathcal{H}^{n-1}(S_u)|\rightarrow 0$ as $j\rightarrow+\infty$.
\end{enumerate}
Denote by $\mathcal{W}(\O)$ the class of all functions for which conditions (\ref{ess-clos})-(\ref{reg}) hold. Then, by a standard density  argument it is enough to prove \eqref{limsup-nD} when $u$ belongs to $\mathcal{W}(\O)$. Indeed assume (\ref{limsup-nD}) holds true in $\mathcal{W}(\O)$. If $u\in GSBV^2(\O)$ then there exists a sequence $(u_j)\subset\mathcal{W}(\O)$ satisfying (\ref{conv-u})-(\ref{conv-hausd}). Hence it holds
\begin{equation*}
\G\hbox{-}\limsup_{\e\to 0}\F_\e(u,v)\leq \liminf_{j\to+\infty}\left(\G\hbox{-}\limsup_{\e\to 0}\F_\e(u_j,v)\right)\leq\lim_{j\to+\infty}\MS(u_j,v)=\MS(u,v),
\end{equation*}
where the first inequality is due to the lower semicontinuity of the $\G\hbox{-}$limsup, whereas the last equality follows from (\ref{conv-grad}) and (\ref{conv-hausd}).

We now prove (\ref{limsup-nD}) for a function $u\in \mathcal{W}(\O)$. Assume first that $\overline{S_u}=\O\cap K$ where $K$ is a closed convex set contained in an $(n-1)\hbox{-}$dimensional hyperplane $\Pi$ with normal $\nu$. Let $p\colon\R^n\to\Pi$ be the orthogonal projection on $\Pi$, $d(x):=\text{dist}(x,\Pi)$, and for any $\d>0$ set $K^\d:=\{x\in\Pi\colon \text{dist} (x,K)\leq\d\}$.

Let $\eta>0$; then there exist a function $f_\eta\in W^{2,2}_{\text{loc}}(0,+\infty)$ and a constant $M_\eta >0$ such that $f_\eta(0)=f_\eta'(0)=0$, $f_\eta(t)=1$ $\forall t>M_\eta$ and
\begin{equation}\label{est6}
\int_0^{M_\eta} \left((f_\eta-1)^2+(f_\eta'')^2\right)\dt\leq\sqrt{2}+\eta.
\end{equation}
Fix $T>M_\eta$ and let $\xi_\e>0$ be such that $\xi_\e/\e\to 0$ as $\e\to 0$; set
\begin{align*}
A_\e & :=\{x\in\R^n\colon p(x)\in K^\e,\;d(x)\leq\xi_\e\}
\\
B_\e & :=\{x\in\R^n\colon p(x)\in K^{2\e},\;d(x)\leq\xi_\e+\e T\}
\end{align*}
(see Figure \ref{figure2}). 
\begin{figure}[!h]
\centering
\psfrag{ex1}[cr][cr]{$\xi_\e+\e T$}
\psfrag{ex2}[cr][cr]{$\xi_\e$}
\psfrag{ex3}[cr][cr]{$-\xi_\e$}
\psfrag{ex4}[cr][cr]{$-\xi_\e-\e T$}
\psfrag{ex5}[cr][cr]{$\e$}
\psfrag{ex6}[cr][cr]{$\e$}
\psfrag{ex7}[c][c]{$S_u$}
\psfrag{ex8}[c][c]{$A_\e$}
\psfrag{ex9}[c][c]{$B_\e$}
\psfrag{ex10}[c][c]{$\Pi$}
\psfrag{ex11}[cr][cr]{$\nu$}
\includegraphics[scale=0.3]{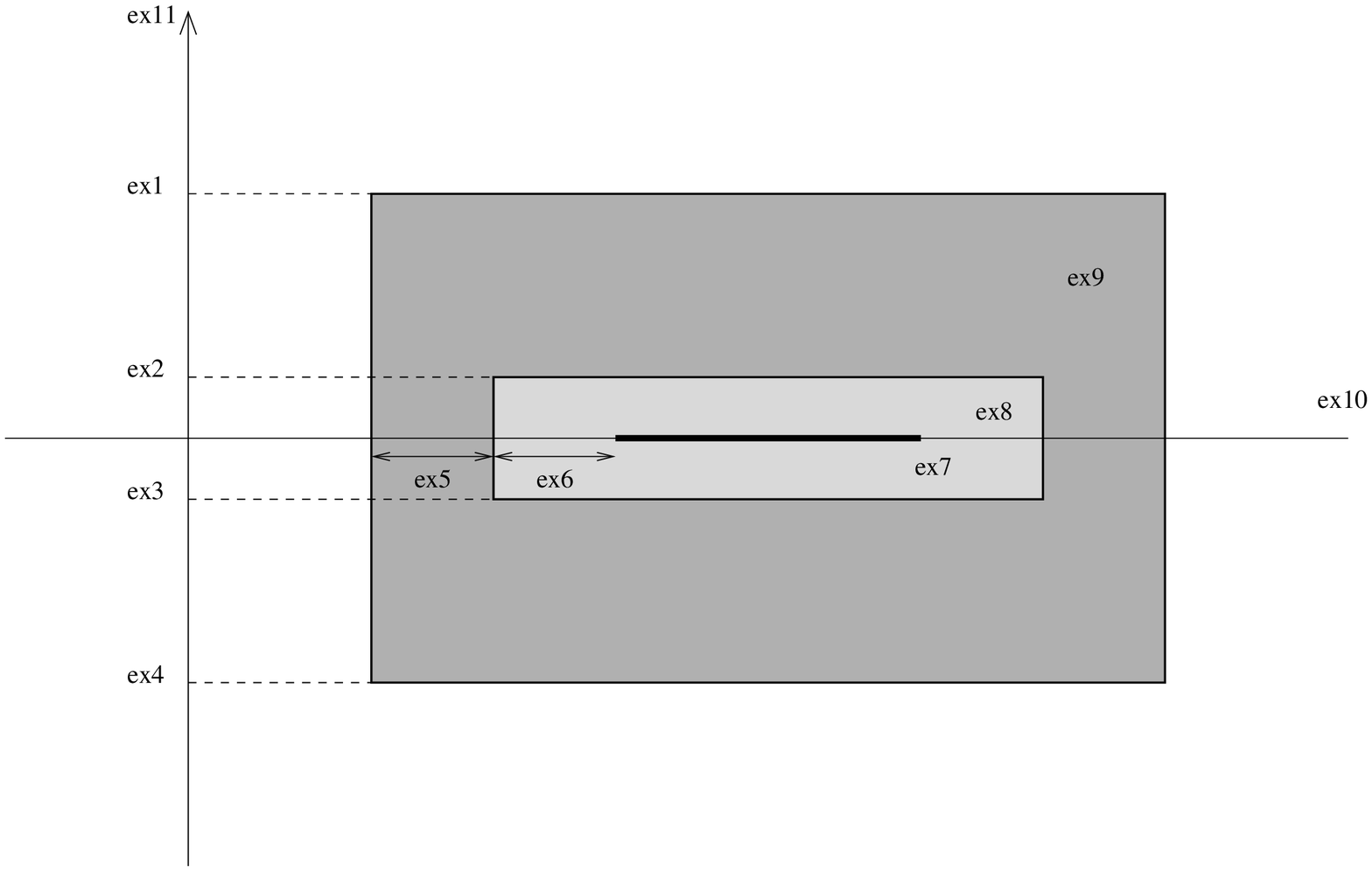}
\caption{The sets $A_\e$ and $B_\e$.}
\label{figure2}
\end{figure}
Consider moreover 
\begin{equation*}
C_\e :=\{x\in\R^n\colon p(x)\in K^{\e/2},\;d(x)\leq\xi_\e/2\}
\end{equation*}
and let $\varphi_\e$ be a cut-off function between $C_\e$ and $A_\e$\ie $\varphi_\e\in C_c^\infty(A_\e)$, $0\leq\varphi_\e\leq 1$ and $\varphi_\e\equiv 1$ on $C_\e$. Define the sequence
\begin{equation*}
u_\e(x):=u(x)(1-\varphi_\e(x));
\end{equation*}
then $(u_\e)\subset W^{1,2}(\O)$ and $u_\e\to u$ in $L^1(\O)$, by the Lebesgue dominated convergence Theorem.

Let $\gamma_\e$ be a cut-off function between $K^\e$ and $K^{2\e}$\ie $\gamma_\e\in C_c^\infty(K^{2\e})$, $0\leq\gamma_\e\leq 1$, $\gamma_\e\equiv 1$ on $K^\e$ with
\begin{equation}\label{bound-grad}
\|\nabla\gamma_\e\|_{L^\infty(\Pi)}\leq \frac{c}{\e},
\end{equation}
\begin{equation}\label{bound-hes}
\|\nabla^2\gamma_\e\|_{L^\infty(\Pi)}\leq\frac{c}{\e^2},
\end{equation}
for some $c>0$. Let $h_\e\colon\R\to\R$ be the function defined by
\begin{equation*}
h_\e(t):=
\begin{cases}
0 & \text{if }\; |t|\leq\xi_\e ,\\
f_\eta\left(\ds\frac{|t|-\xi_\e}{\e}\right) & \text{if }\; \xi_\e\leq|t|\leq\xi_\e+\e T ,\\
1 & \text{if }\; |t|\geq\xi_\e+\e T,
\end{cases}
\end{equation*}
where $f_\eta$ is as in (\ref{est6}). Let
\begin{equation}\label{v_eps}
v_\e(x):=\gamma_\e(p(x))h_\e(d(x))+(1-\gamma_\e(p(x)));
\end{equation}
by construction $(v_\e)\subset W^{2,2}(\O)\cap L^\infty(\O)$ so that $v_\e \nabla u_\e \in L^2(\O;\R^n)$; moreover $v_\e\rightarrow 1$ in $L^1(\O)$.

We have
\begin{align}\nonumber
\limsup_{\e\to 0}\F_\e(u_\e,v_\e) & = \limsup_{\e\to 0}\int_\O \left(\a\,v_\e^2|\nabla u_\e|^2+\frac{\b}{2\sqrt{2}}\,\left(\frac{(v_\e-1)^2}{\e}+{\e}^3|\nabla^2 v_\e|^2\right) \right)\dx\\\nonumber
 & \leq \a\,\limsup_{\e\to 0}\int_\O v_\e^2|\nabla u_\e|^2\dx + \frac{\b}{2\sqrt{2}}\,\limsup_{\e\to 0}\int_{A_\e} \left(\frac{(v_\e-1)^2}{\e}+{\e}^3|\nabla^2 v_\e|^2 \right)\dx \\\nonumber
 & \quad+ \frac{\b}{2\sqrt{2}}\,\limsup_{\e\to 0}\int_{B_\e\setminus A_\e} \left(\frac{(v_\e-1)^2}{\e}+{\e}^3|\nabla^2 v_\e|^2 \right)\dx \\ \nonumber
 & \quad+\frac{\b}{2\sqrt{2}}\,\limsup_{\e\to 0}\int_{\O\setminus B_\e} \left(\frac{(v_\e-1)^2}{\e}+{\e}^3|\nabla^2 v_\e|^2 \right)\dx\\[2pt]\label{ls1}
 &\quad =: \a I_1+\frac{\b}{2\sqrt{2}}(I_2+I_3+I_4).
\end{align}
We now compute separately the four terms $I_1, I_2, I_3$, and $I_4$. 

Since $v_\e\equiv 0$ on $A_\e$ and $u_\e\equiv u$ on $\O\setminus A_\e$, we immediately deduce
\begin{align}\label{ls2}
I_1 = \limsup_{\e\to 0}\int_{\O\setminus A_\e}v_\e^2|\nabla u|^2\dx =\int_\O|\nabla u|^2\dx;
\end{align}
moreover the term $I_2$ gives no contribution to the computation, indeed 
\begin{align}\nonumber
I_2 & = \limsup_{\e\to 0}\frac{1}{\e}\int_{A_\e} \dx \\\nonumber
& =\limsup_{\e\to 0}\frac{1}{\e}\int_{-\xi_\e}^{\xi_\e} dt\int_{K^\e} d\mathcal{H}^{n-1}(y)\\\label{ls3}
& = 2\limsup_{\e\to 0}\frac{\xi_\e}{\e}\,\mathcal{H}^{n-1}(K^\e)=0,
\end{align}
where we have used the fact that $\xi_\e\ll\e$ and $\mathcal{H}^{n-1}(K^\e)\rightarrow\mathcal{H}^{n-1}(K)<+\infty$ as $\e\rightarrow 0$.

Since $v_\e\equiv 1$ on $\O\setminus B_\e$, we also have $I_4=0$. Then it only remains to compute $I_3$. To this end it is convenient to decompose $B_\e\setminus A_\e$ as the union of two sets, $D_\e$ and $E_\e$, defined as follows
\begin{align*}
D_\e & :=\{x\in\R^n\colon p(x)\in K^\e,\;\xi_\e\leq d(x) \leq\xi_\e+\e T\}\\
E_\e & :=\{x\in\R^n\colon p(x)\in K^{2\e}\setminus K^\e,\;d(x)\leq\xi_\e+\e T\}
\end{align*}
(see Figure \ref{figure3}).

\begin{figure}[!h]
\centering
\psfrag{ex1}[cr][cr]{$\xi_\e+\e T$}
\psfrag{ex2}[cr][cr]{$\xi_\e$}
\psfrag{ex3}[cr][cr]{$-\xi_\e$}
\psfrag{ex4}[cr][cr]{$-\xi_\e-\e T$}
\psfrag{ex5}[c][c]{$\e$}
\psfrag{ex6}[c][c]{$\e$}
\psfrag{ex7}[c][c]{$E_\e$}
\psfrag{ex8}[c][c]{$D_\e$}
\psfrag{ex9}[c][c]{$\Pi$}
\psfrag{ex10}[cr][cr]{$\nu$}
\includegraphics[scale=0.35]{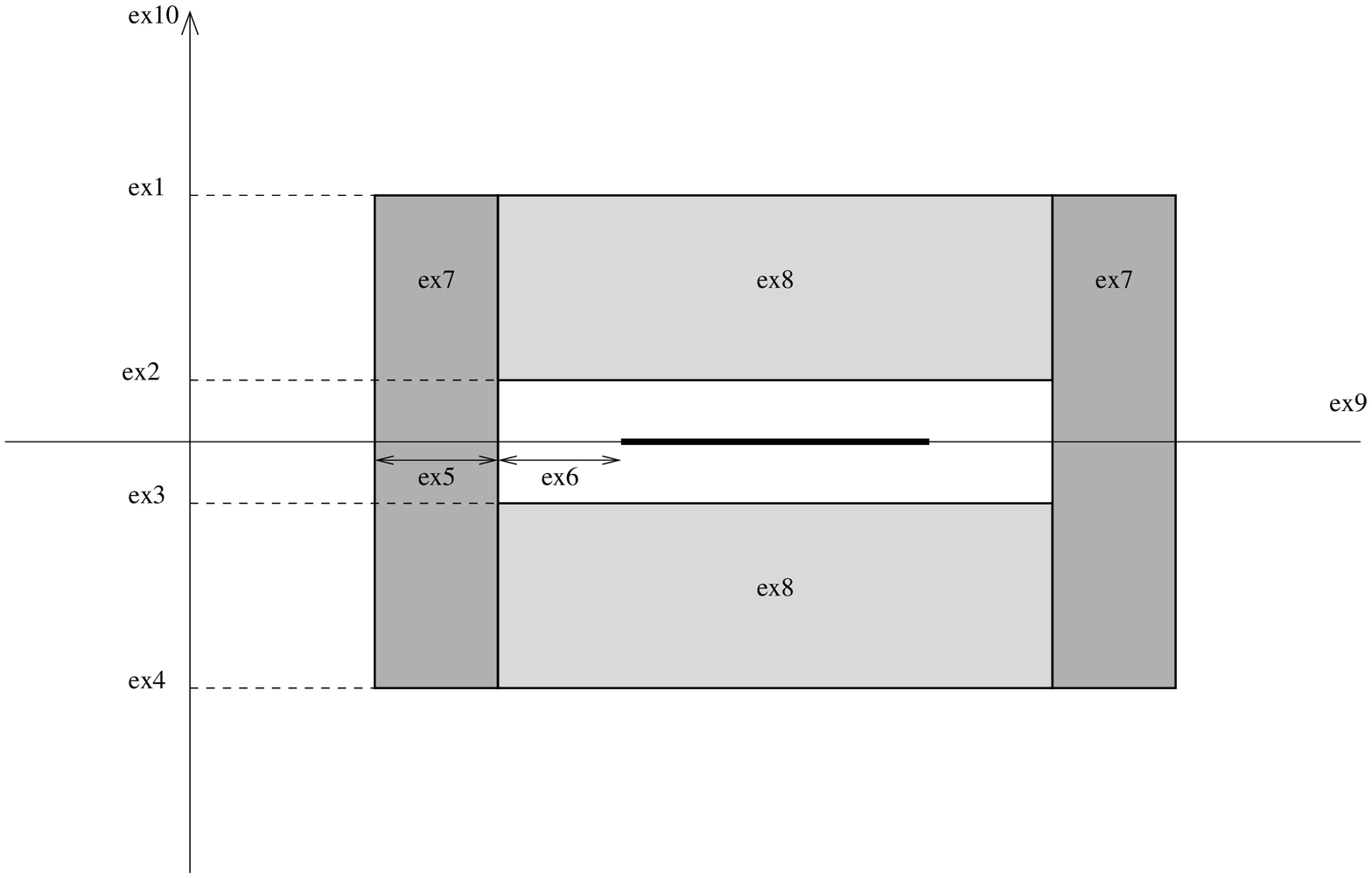}
\caption{The sets $D_\e$ and $E_\e$.}
\label{figure3}
\end{figure}
On $D_\e$ the function $v_\e= f_\eta\left(\frac{d(x)-\xi_\e}{\e}\right)$, moreover $\nabla d(x)=\pm \nu$ hence we have
\begin{eqnarray}\nonumber
&\quad &\ds\limsup_{\e\to 0}\int_{D_\e} \left(\frac{(v_\e-1)^2}{\e}+{\e}^3|\nabla^2 v_\e|^2 \right)\dx\\\nonumber
&=&\ds \limsup_{\e\to 0} \int_{D_\e} \left(\frac{1}{\e}\left(f_\eta\left(\frac{d(x)-\xi_\e}{\e}\right)-1\right)^2+\e^3\left|f_\eta''\left(\frac{d(x)-\xi_\e}{\e}\right)\frac{\nu^T\cdot\nu}{\e^2}\right|^2\right)\dx\\\nonumber
&\leq&\ds2\,\limsup_{\e\to 0}\int_{K^\e}\int_{\xi_\e}^{\xi_\e+\e T}\left(\frac{1}{\e}\left(f_\eta\left(\frac{t-\xi_\e}{\e}\right)-1\right)^2+ \frac{1}{\e}\left(f_\eta''\left(\frac{t-\xi_\e}{\e}\right)\right)^2\right)\dt\dH\\\nonumber
&=&\ds 2\, \limsup_{\e\to 0}\int_{K^\e}\int_0^T((f_\eta(t)-1)^2+(f_\eta''(t))^2)\dt\dH\\\label{ls4}
&\leq &\ds (2\sqrt{2}+\tilde{\eta})\,\limsup_{\e\to 0}\mathcal{H}^{n-1}(K^\e)\, =\,(2\sqrt{2}+\tilde{\eta})\,\mathcal{H}^{n-1}(K).
\end{eqnarray}
Hence finally to achieve the limsup-inequality we have to show that
\begin{equation}\label{ls5}
\limsup_{\e\to 0}\int_{E_\e} \left(\frac{(v_\e-1)^2}{\e}+{\e}^3|\nabla^2 v_\e|^2 \right)\dx=0.
\end{equation}
To do this we consider the further decomposition of $E_\e$ as the union of two sets $V_\e$ and $W_\e$ defined by
\begin{align*}
V_\e & :=\{x\in\R^n\colon p(x)\in K^{2\e}\setminus K^\e,\;d(x)\leq\xi_\e\}\\
W_\e & :=\{x\in\R^n\colon p(x)\in K^{2\e}\setminus K^\e,\;\xi_\e\leq d(x)\leq\xi_\e+\e T\}
\end{align*}
(see Figure \ref{figure4}).
\begin{figure}[!h]
\centering
\psfrag{ex1}[cr][cr]{$\xi_\e+\e T$}
\psfrag{ex2}[cr][cr]{$\xi_\e$}
\psfrag{ex3}[cr][cr]{$-\xi_\e$}
\psfrag{ex4}[cr][cr]{$-\xi_\e-\e T$}
\psfrag{ex5}[c][c]{$\e$}
\psfrag{ex6}[c][c]{$W_\e$}
\psfrag{ex7}[c][c]{$V_\e$}
\psfrag{ex8}[c][c]{$\Pi$}
\psfrag{ex9}[cr][cr]{$\nu$}
\includegraphics[scale=0.35]{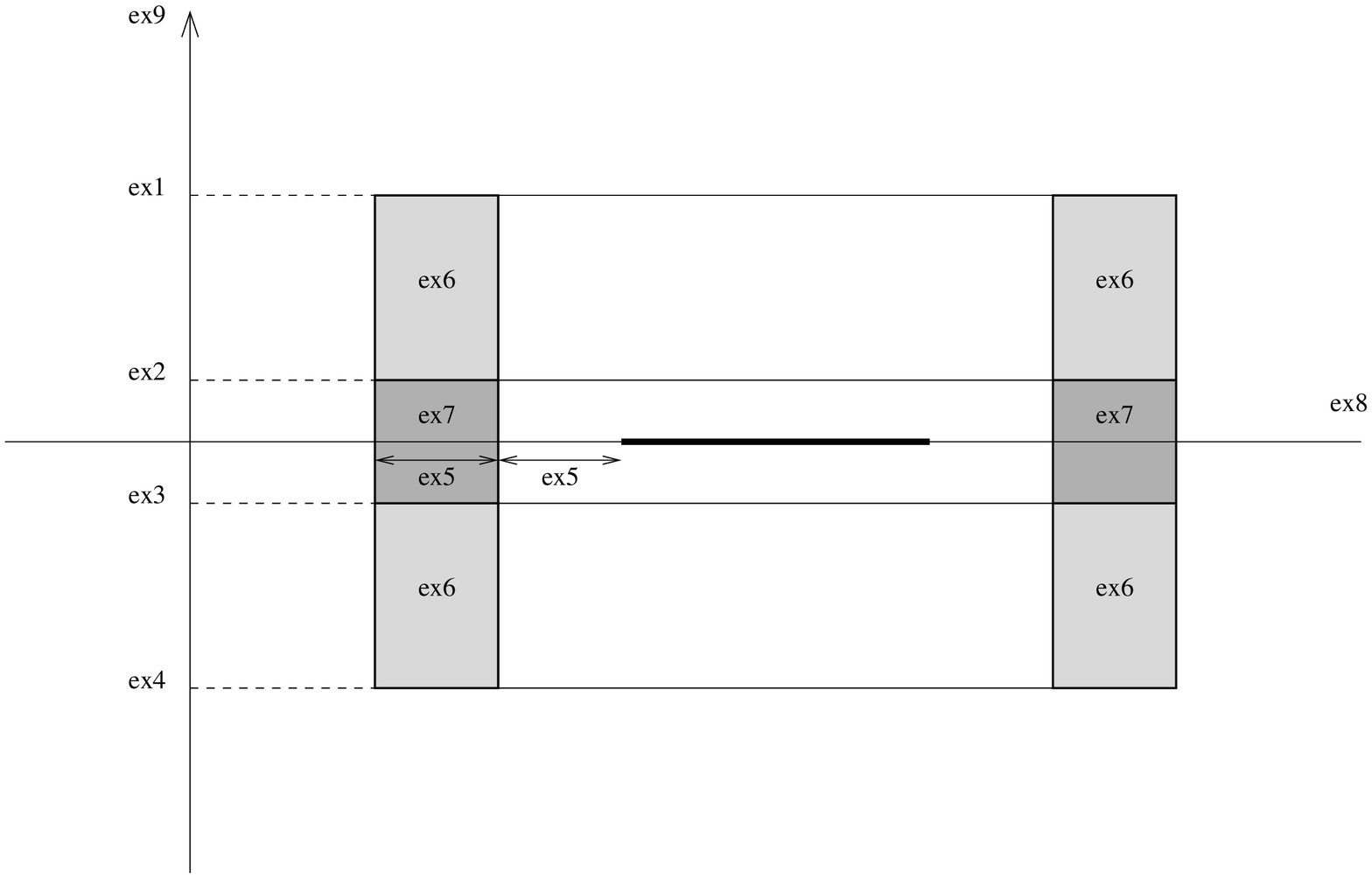}
\caption{The sets $V_\e$ and $W_\e$.}
\label{figure4}
\end{figure}
Since $h_\e(d(x))\equiv 0$ on $V_\e$, we have $v_\e(x)=1-\gamma_\e(p(x))$ for $x\in V_\e$. Then, if we denote by $D_p(x)$ the Jacobian matrix of $p$ evaluated at $x$, we get that $\|D_p(x)\|_{L^\infty(\R^n;\R^{n\times n})}\leq 1$ and
\begin{eqnarray*}
&\quad& \int_{V_\e}\left( \ds\frac{(v_\e-1)^2}{\e}+{\e}^3|\nabla^2 v_\e|^2 \right)\dx\\
\qquad\qquad\qquad & = & \int_{V_\e}\left(\frac{\gamma^2_\e(p(x))}{\e}+\e^3 \left|-D_p(x)^\text{T}\cdot \nabla^2\gamma_\e(p(x))\cdot D_p(x)\right|^2\right)\dx \\
& \leq & \int_{K^{2\e}\setminus K^\e}\int_{-\xi_\e}^{\xi_\e} \left(\frac{\gamma^2_\e(y)}{\e}+\e^3 \left|\nabla^2\gamma_\e(y)\right|^2\right)\dt\dH\\
& \leq & c\,\frac{\xi_\e}{\e}\,\mathcal{H}^{n-1}(K^{2\e}\setminus K^{\e}) \to 0 \quad \text{as}\; \e\to 0.
\end{eqnarray*}
where in the last inequality we have used (\ref{bound-hes}).

Also
\begin{equation*}
\limsup_{\e\to 0}\int_{W_\e} \left( \ds\frac{(v_\e-1)}{\e}+{\e}^3|\nabla^2 v_\e|^2 \right)\dx=0,
\end{equation*}
indeed $v_\e(x)=\gamma_\e(p(x))f_\eta\left(\frac{d(x)-\xi_\e}{\e}\right)+(1-\gamma_\e(p(x)))$ on $W_\e$ and we have
%
\begin{eqnarray*}
&\quad&\int_{W_\e}\left( \ds\frac{(v_\e-1)^2}{\e}+{\e}^3|\nabla^2 v_\e|^2 \right)\dx\\
& = & 2\,\int_{K^{2\e}\setminus K^\e}\int_{\xi_\e}^{\xi_\e +\e T}\bigg(\frac{1}{\e}\left(\gamma_\e(y)f_\eta\left(\frac{
t-\xi_\e}{\e}\right)-\gamma_\e(y)\right)^2\\
&& + \e^3\bigg|D_p(y+t\nu)^\text{T}\cdot \nabla^2\gamma_\e(y)\cdot D_p(y+t\nu)\,\left(f_\eta\left(\frac{t-\xi_\e}{\e}\right)-1\right)\\
&& +\frac{1}{\e}\,f'_\eta\left(\frac{t-\xi_\e}{\e}\right)\,\big((\nabla\gamma_\e(y)\cdot D_p(y+t\nu))^\text{T}\cdot \nabla d(y+t\nu) +\nabla d(y+t\nu)^T\cdot\nabla\gamma_\e(y)\cdot D_p(y+t\nu)\big)\\
&& +\gamma_\e(y)\,f''_\eta\left(\frac{t-\xi_\e}{\e}\right)\frac{\nu^T\cdot\nu}{\e^2}\bigg|^2\bigg)\dt\dH\\
& \leq & \frac{c}{\e}\,\e\,\mathcal{H}^{n-1}(K^{2\e}\setminus K^\e)=c\,\mathcal{H}^{n-1}(K^{2\e}\setminus K^\e) \to 0 \quad \text{as}\; \e \to 0,
\end{eqnarray*}
where to establish the last inequality we have used (\ref{bound-grad}) and (\ref{bound-hes}). 

Thus gathering \eqref{ls1}-\eqref{ls5} we finally deduce
\begin{equation*}
\limsup_{\e\to 0}\F_\e(u_\e,v_\e)\leq \a\int_\O |\nabla u|^2\dx+(\b+\tilde\eta)\,\mathcal{H}^{n-1}(S_u)\qquad\forall\tilde\eta>0,
\end{equation*}
and hence the limsup inequality.

We now consider the general case in which $\overline{S_u}=\O\cap\ds\sqcup_{i=1}^r K_i$, with $K_i$ closed and convex set contained in an $(n-1)\hbox{-}$dimensional hyperplane $\Pi^i$, with normal $\nu_i$. Let $p_i:\R^n\to\Pi^i$ be the orthogonal projection on $\Pi^i$, $d_i(x):=\text{dist}(x,\Pi^i)$, and for every $\d>0$ set $K_i^\d:=\{x\in\Pi^i\colon \text{dist}(x,K_i)\leq\d\}$.

Let $f_\eta$ be as in (\ref{est6}) and fix $T>M_\eta$; for $0<\xi_\e \ll\e$ consider the sets
\begin{align*}
A_\e^i & :=\{x\in\R^n\colon p_i(x)\in K^\e_i,\;d_i(x)\leq\xi_\e\}\\
B_\e^i & :=\{x\in\R^n\colon p_i(x)\in K^{2\e}_i,\;d_i(x)\leq\xi_\e+\e T\}.
\end{align*}
Arguing as above we can construct a recovery sequence $(u_\e^i,v_\e^i)$ for $\F_\e(u,v,K_i)$. Then, let $\e_0>0$ be such that
\begin{equation}\label{cond-e}
T<\frac{\d_0-2\,\xi_\e}{2\e},\qquad\forall\, 0<\e<\e_0,
\end{equation}
where $\d_0:=\min\{\text{dist}(K_i,K_j)\colon i,j=1,\dots,r,\,i\neq j\}$; for $\e\in(0,\e_0)$ set
\begin{equation*}
\hat A_\e:=\bigcup_{i=1}^r A_\e^i,\quad \hat B_\e:=\bigcup_{i=1}^r B_\e^i
\end{equation*}
and consider the sequences $(u_\e)\subset W^{1,2}(\O)$ and $(v_\e)\subset W^{2,2}(\O)\cap L^\infty(\O)$ defined by
\begin{equation}\label{def-u}
u_\e(x):= 
\begin{cases}
u_\e^i & \text{if }\; x\in A_\e^i,\cr
u & \text{if }\; x\notin \hat A_\e,
\end{cases} 
\end{equation}
and
\begin{equation}\label{def-v}
v_\e(x):= 
\begin{cases}
v_\e^i & \text{if }\; x\in B_\e^i,\cr
1 & \text{if }\; x\notin \hat B_\e.
\end{cases} 
\end{equation}
Condition (\ref{cond-e}) ensures that $B_\e^i\cap B_\e^j=\emptyset$ for all $i\neq j$ and consequently that the sequences in (\ref{def-u}) and (\ref{def-v}) are well-defined. Moreover we have $(u_\e, v_\e)\to (u,1)$ in $L^1(\O)\times L^1(\O)$ and
\begin{eqnarray*}
&\quad&\limsup_{\e\to 0}\F_\e(u_\e,v_\e)\\
\qquad & \leq & \a\,\limsup_{\e\to 0}\int_{\O} v_\e^2|\nabla u_\e|^2\dx+\frac{\b}{2\sqrt{2}}\,\bigg(\limsup_{\e\to 0}\int_{\hat A_\e}\left(\frac{(v_\e-1)^2}{\e}+{\e}^3|\nabla^2 v_\e|^2 \right)\dx\\
& & +\limsup_{\e\to 0}\int_{\hat B_\e \setminus \hat A_\e}\left(\frac{(v_\e-1)^2}{\e}+{\e}^3|\nabla^2 v_\e|^2 \right)\dx+\limsup_{\e\to 0}\int_{\O\setminus \hat B_\e}\left(\frac{(v_\e-1)}{\e}+{\e}^3|\nabla^2 v_\e|^2 \right)\dx\bigg)\\
& \leq & \a\int_\O |\nabla u|^2\dx+ \frac{\b}{2\sqrt{2}}\,\sum_{i=1}^r \bigg(\limsup_{\e\to 0}\int_{A_\e^i}\left(\frac{(v_\e-1)^2}{\e}+{\e}^3|\nabla^2 v_\e|^2 \right)\dx\\
& & +\limsup_{\e\to 0}\int_{B_\e^i\setminus A_\e^i}\left(\frac{(v_\e-1)^2}{\e}+{\e}^3|\nabla^2 v_\e|^2 \right)\dx+\limsup_{\e\to 0}\int_{\O\setminus B_\e^i}\left(\frac{(v_\e-1)^2}{\e}+{\e}^3|\nabla^2 v_\e|^2 \right)\dx\bigg)\\
& = & \a\int_\O |\nabla u|^2\dx+\b\,\sum_{i=1}^r\mathcal{H}^{n-1}(K_i)\\
& = & \a\int_\O |\nabla u|^2\dx+\b\,\mathcal{H}^{n-1}(S_u)=\MS(u,1).
\end{eqnarray*}
\end{proof}

\begin{remark} \label{tildeFepsilon}
{\rm Let $f\in W^{2,2}_0(\O)$ then it is immediate to show that
 $$
 \|\Delta f\|_{L^2(\O)}=\|\nabla^2 f\|_{L^2(\O;\R^{n\times n})}.
 $$
Therefore, if $u\in W^{1,2}(\O)$, $(v-1)\in W^{2,2}_0(\O)$, and $v \nabla u \in L^2(\O;\R^n)$ the functionals $\F_\e$ can be rewritten as
\begin{equation*}
\F_\e(u,v)=\a\int_\O v^2|\nabla u|^2\dx+\frac{\b}{2\sqrt{2}}\int_\O\left(\frac{(v-1)^2}{\e}+{\e}^3|\Delta v|^2 \right)\dx.
\end{equation*}}
\end{remark}
\noindent Let $\tilde{\F}_\e \colon L^1(\O)\times L^1(\O) \longrightarrow [0,+\infty]$ denote the functionals 
\begin{equation*} 
\tilde{\F}_\e(u,v):=
\begin{cases}
\ds\a\int_\O v^2|\nabla u|^2\dx+\frac{\b}{2\sqrt{2}}\int_\O\left(\frac{(v-1)^2}{\e}+{\e}^3|\Delta v|^2 \right)\dx & \text{if}\; (u,v-1)\in W^{1,2}(\O)\times W^{2,2}_0(\O)\\
& \text{and}\; v \nabla u \in L^2(\O;\R^n),
\cr\cr
+\infty & \text{otherwise},
\end{cases}
\end{equation*}
then the following theorem is an immediate consequence of Remark \ref{tildeFepsilon} and Theorem \ref{gconv-nD}.
\begin{theorem}
The sequence $(\tilde{\F}_\e)$ $\G\hbox{-}$converges, with respect to the $(L^1(\O)\times L^1(\O))$-topology, to the functional $\MS$, defined as in (\ref{Mu-Sh}).
\end{theorem}
\begin{proof}
Since every sequence $(v_\e)$ with equibounded energy strongly converges to $1$ in $L^2(\O)$, the liminf inequality is trivial. Hence, it only remains to prove the limsup inequality. To this end we notice that the recovery sequence exhibited in Theorem \ref{gconv-nD} is also a recovery sequence for the functional $\tilde{\F}_\e$. Indeed it is immediate to check that $(v_\e)$ defined as in (\ref{def-v}) (see also \eqref{v_eps}) satisfies $(v_\e-1)\in W^{2,2}_0(\O)$.
\end{proof}

\subsection{Convergence of minimization problems}\label{s:min-pb}
Let $\eta_\e>0$ be such that $\eta_\e/\e\to 0$ as $\e\to0$. Let $\gamma>0$ and $g\in L^2(\O)$ be given and for every $(u,v)\in W^{1,2}(\O)\times W^{2,2}(\O)$ consider the functionals  
\begin{equation}\label{pert-coer}
\F_\e(u,v)+\eta_\e\int_\O |\nabla u|^2\dx+\gamma\int_\O |u-g|^2\dx.
\end{equation}
The following result holds true. 

\begin{theorem}\label{eqcoerc}
For every fixed $\e>0$ there exists a minimizing pair $(\tilde u_\e,\tilde v_\e)$ for the problem
\begin{equation*}
M_\e:=\inf\left\{\F_\e(u,v)+\eta_\e\int_\O |\nabla u|^2\dx+\gamma\int_\O |u-g|^2\dx\colon (u,v)\in W^{1,2}(\O)\times W^{2,2}(\O)\right\}.
\end{equation*} 
Moreover, up to subsequences, $(\tilde u_\e,\tilde v_\e)\to (u,1)$ in $L^1(\O)\times L^1(\O)$ where $u$ is a solution to
\begin{equation}
{M}:=\min\left\{\MS(u,1)+\gamma\int_\O|u-g|^2\dx \colon u\in GSBV^2(\O)\right\};
\end{equation}
if $n=1$ then $u\in SBV^2(\O)$. Further, $M_\e \to M$ as $\e\to 0$.
\end{theorem}
\begin{proof}
For fixed $\e>0$ the perturbation term $\eta_\e\int_\O|\nabla u|^2$ makes the functionals in \eqref{pert-coer} coercive with respect to the weak $(W^{1,2}(\O)\times W^{2,2}(\O))$-topology. Indeed, let $(u_k,v_k)\subset W^{1,2}(\O)\times W^{2,2}(\O)$ be such that
$$
\F_\e(u_k,v_k)+\eta_\e\int_\O |\nabla u_k|^2\dx+\gamma\int_\O |u_k-g|^2\dx \to M_\e \quad \text{as }\; k\to +\infty.
$$
As a consequence we deduce 
$$
\|v_k \nabla u_k\|_{L^2(\O;\R^n)} \leq c, \quad \|u_k\|_{W^{1,2}(\O)} \leq c
$$
for some $c>0$ independent of $k$; moreover, by the interpolation inequality Proposition \ref{interpol-ineq} we also have
$\|v_k\|_{W^{2,2}(\O)} \leq c$. Then up to subsequences (not relabelled)
\begin{equation}\label{convergenze}
\begin{cases}
v_k\nabla u_k \wto w & \text{in }\; L^2(\O;\R^n),
\cr
u_k \wto u & \text{in }\; W^{1,2}(\O),
\cr
v_k \wto v & \text{in }\; W^{2,2}(\O) \quad(\Rightarrow \; v_k \to v \;\text{ in }\; W^{1,2}(\O)).
\end{cases}
\end{equation}
Therefore 
$$
v_k \nabla u_k \wto v \nabla u\quad  \text{in }\; L^1(\O;\R^n), 
$$ 
hence by the uniqueness of the weak limit $v \nabla u= w\in L^2(\O;\R^n)$.   
Then, the existence of a minimizing pair $(\tilde u_\e, \tilde v_\e)$ easily follows appealing to the weak lower semicontinuity of the $L^2$ norm and to the direct methods.  

The requirement that $\eta_\e/\e\to 0$ as $\e\to 0$ ensures that  
\begin{equation}\label{gc-pert}
\F_\e(u,v)+\eta_\e\int_\O |\nabla u|^2\dx+\gamma\int_\O |u-g|^2\dx\; \stackrel{\Gamma}{\longrightarrow}\; \MS(u,v) + \gamma\int_\O |u-g|^2\dx
\end{equation}
with respect to the strong $(L^1(\O)\times L^1(\O))$-topology. This can be easily seen arguing as in Theorems \ref{gconv-1D} and \ref{gconv-nD} (now taking $\xi_\e=\sqrt{\eta_\e\,\e}$) and recalling that $\Gamma$-convergence is stable under continuous perturbations.

Moreover, if $(u_\e,v_\e)\subset W^{1,2}(\O)\times W^{2,2}(\O)$ is any sequence such that
\begin{equation*}
\sup_{\e}\left(\F_\e(u_\e,v_\e)+\eta_\e\int_\O |\nabla u_\e|^2\dx+\int_\O |u_\e-g|^2\dx\right)<+\infty,
\end{equation*}
then the interpolation inequality Proposition \ref{interpol-ineq} implies  
\begin{equation*}
\sup_{\e}\left(\mathcal{AT}_\e(u_\e,v_\e)+\eta_\e\int_a^b (u_\e')^2\dt+\int_a^b |u_\e-g|^2\dt\right)<+\infty,
\end{equation*}
where $\mathcal{AT}_\e$ is as in \eqref{f:AT}. Thus \cite[Theorem 1.2]{AT92} immediately yields the equicorecivity of the functionals as in \eqref{pert-coer}. 
Finally, by virtue of  \eqref{gc-pert}, the convergence of the associated minimization problems is ensured by the fundamental property of $\Gamma$-convergence.
\end{proof}


\section{The second model} \label{secondmodel}
\noindent In this section we study the asymptotic behaviour of the energies $\E_\e$ defined in (\ref{energy-lap}). Specifically we prove that, up to imposing suitable boundary conditions on $v$, the $\G\hbox{-}$limit of $\E_\e$ is again given by \eqref{Mu-Sh}. 

\medskip

In what follows $\O$ will be an open bounded subset of $\R^n$ with $C^2$ boundary.

\medskip

Let $\E_\e$ be as in (\ref{energy-lap}); we have
\begin{equation}\label{bound2}
\E_\e(u,v)\leq\a\int_\O v^2|\nabla u|^2\dx+\frac{\b}{2\sqrt{2}}\int_\O\left(\frac{(v-1)^2}{\e}+2\,\e^3|\nabla^2 v|^2\right)\dx,
\end{equation}
for all $(u,v)\in W^{1,2}(\O)\times W^{2,2}(\O)$ such that $v\nabla u \in L^2(\O;\R^n)$.

\medskip

It is convenient to introduce the following notation. Let $\E', \E'' \colon L^1(\O)\times L^\infty(\O)\longrightarrow [0,+\infty]$ be the functionals defined as 
$$
\E'(\cdot, \cdot):=\G\hbox{-}\liminf_{\e\to 0} \E_{\e_j}(\cdot, \cdot),\qquad
\E''(\cdot, \cdot):=\G\hbox{-}\limsup_{\e\to 0} \E_{\e_j}(\cdot, \cdot);
$$
then, by virtue of Theorem \ref{gconv-nD}, we get
\begin{equation*}
\E'(u,v)\leq\E''(u,v)\leq\a\int_\O|\nabla u|^2\dx+2\,\b\,\mathcal{H}^{n-1}(S_u)
\end{equation*}
for all $u\in GSBV^2(\O)$ and for $v=1$ a.e. in $\O$, hence 
$$
GSBV^2(\O)\times\{v=1\;\text{a.e. in}\;\O\} \subset {\rm dom}\, \E'' \subset {\rm dom}\, \E'.
$$
We now apply the blow-up argument of Fonseca-M\"{u}ller \cite{FM} (see also \cite{BFLM}) to obtain the following lower bound inequality for the functionals $\E_\e$.
\begin{proposition}\label{lower_bd}
For every $u\in GSBV^2(\O)$, we have
\begin{equation*}
\G\hbox{-}\liminf_{\e\to 0}\E_\e(u,1)\geq \MS(u,1),
\end{equation*}
with $\MS$ defined as in (\ref{Mu-Sh}).
\end{proposition}
\begin{proof}
Assume first that $u$  belongs to $SBV^2(\O)$. Let $(u_\e,v_\e)\subset L^1(\O)\times L^1(\O)$ be such that $(u_\e,v_\e)\rightarrow (u,1)$ in $L^1(\O)\times L^1(\O)$ and $\sup_\e\E_\e(u_\e,v_\e)<+\infty$. 

For each $\e>0$ consider the measures
\begin{equation*}
\mu_\e:=\left(\a\,v_\e^2|\nabla u_\e|^2+\frac{\b}{2\sqrt{2}}\,\left(\frac{(v_\e-1)^2}{\e}+\e^3|\Delta v_\e|^2\right)\right)\mathcal{L}^n\llcorner\O.
\end{equation*}
By hypothesis $\mu_\e(\O)=\E_\e(u_\e,v_\e)$ is equibounded therefore, up to subsequences (not relabelled), $\mu_\e\wto^*\mu$ where $\mu$ is a non-negative finite Radon measure on $\O$. Using the Radon-Nikod\'{y}m Theorem we decompose $\mu$ into the sum of three mutually orthogonal measures
\begin{equation*}
\mu=\mu_a\mathcal{L}^n+\mu_J\mathcal{H}^{n-1}\llcorner S_u+\mu_s
\end{equation*}
and we claim that
\begin{equation}\label{est_grad}
\mu_a(x_0)\geq\a\,|\nabla u(x_0)|^2\qquad \text{for }\; \mathcal{L}^n\hbox{-}\text{a.e.}\;x_0\in\O
\end{equation}
and
\begin{equation}\label{est_jump}
\mu_J(x_0)\geq \b\qquad\text{for }\;\mathcal{H}^{n-1}\hbox{-}\text{a.e.}\;x_0\in S_u.
\end{equation}
Suppose for a moment that \eqref{est_grad} and \eqref{est_jump} hold true, then to conclude it is enough to consider an increasing sequence of smooth cut-off functions $(\varphi_k)$, such that $0\leq\varphi_k\leq 1$ and $\sup_k\varphi_k(x)=1$ on $\O$, and to note that for every $k\in\N$
\begin{eqnarray*}
\lim_{\e\to 0}\E_\e(u_\e,v_\e) & \geq & \liminf_{\e\to 0} \int_\O \left(\a\,v_\e^2|\nabla u_\e|^2+\frac{\b}{2\sqrt{2}}\,\left(\frac{(v_\e-1)^2}{\e}+\e^3|\Delta  v_\e|^2\right)\right)\varphi_k\dx\\
& = & \int_\O \varphi_k\,d\mu\geq \int_\O \mu_a\varphi_k\dx+ \int_{S_u} \mu_J\varphi_k\,d\mathcal{H}^{n-1}\\
& \geq & \a\int_\O |\nabla u|^2\varphi_k\dx+\b\int_{S_u}\varphi_k\,d\mathcal{H}^{n-1}.
\end{eqnarray*}
Hence, letting $k\to+\infty$ the thesis follows from the monotone convergence Theorem.

We now prove (\ref{est_grad}) and \eqref{est_jump}. We start proving (\ref{est_grad}). To this end let $x_0\in\O$ be a Lebesgue point for $\mu$ with respect to $\mathcal{L}^n$ such that $u$ is approximately differentiable at $x_0$\ie
\begin{equation}\label{L_point}
\mu_a(x_0)=\lim_{\r\to 0}\frac{\mu(Q_\r(x_0))}{\mathcal{L}^n(Q_\r(x_0))}=\lim_{\r\to 0}\frac{\mu(Q_\r(x_0))}{\r^n}
\end{equation}
and
\begin{equation}\label{approx_point}
\lim_{\r\to 0}\frac{1}{\r^{n+1}}\int_{Q_\r(x_0)}|u(x)-u(x_0)-\langle\nabla u(x_0),x-x_0\rangle|\dx=0.
\end{equation}
The Besicovitch derivation Theorem together with the Calder\'{o}n-Zygmund Theorem ensures that (\ref{L_point})-(\ref{approx_point}) hold true for a.e. $x_0\in \O$.

Since $\mu$ is a finite Radon measure we have that $\mu(\partial Q_\r(x_0))=0$ for all $\r>0$ except for a countable set. Thus, being $\chi_{\overline{Q_\r(x_0)}}$ an upper semicontinuous function with compact support in $\O$ (for $\r$ small) \cite[Proposition 1.62(a)]{AFP} yields
\begin{eqnarray*}
\mu_a(x_0) & = & \ds\lim_{\r\to 0}\frac{1}{\r^n}\int_{\overline{Q_\r(x_0)}}d\mu\geq\lim_{\r\to 0}\limsup_{\e\to 0}\frac{1}{\r^n}\,\mu_\e(Q_\r(x_0))\\
&=&\lim_{\r\to 0}\limsup_{\e\to 0}\frac{1}{\r^n}\int_{Q_\r(x_0)}\left(\a\,v_\e^2(x)|\nabla u_\e(x)|^2+\frac{\b}{2\sqrt{2}}\,\left(\frac{(v_\e(x)-1)^2}{\e}+\e^3|\Delta v_\e(x)|^2\right)\right)\dx\\
&=&\lim_{\r\to 0}\limsup_{\e\to 0}\int_Q\bigg(\a\,v_\e^2(x_0+\r y)|\nabla u_\e(x_0+\r y)|^2\\
&&\qquad\qquad\qquad\qquad +\frac{\b}{2\sqrt{2}}\,\left(\frac{(v_\e(x_0+\r y)-1)^2}{\e}+\e^3|\Delta v_\e(x_0+\r y)|^2\right)\bigg)\dy.
\end{eqnarray*}
Now we suitably modify $u_\e$ to obtain a sequence converging in $L^1(Q)$ to the linear function $w_0(y)=\langle\nabla u(x_0),y\rangle$. Set
\begin{equation*}
w_{\e,\r}(y):=\frac{u_\e(x_0+\r y)-u(x_0)}{\r},\qquad v_{\e,\r}(y):=v_\e(x_0+\r y);
\end{equation*}
then, letting first $\e\to 0$ and then $\varrho\to 0$ we get $(w_{\e,\r},v_{\e,\r})\rightarrow (w_0,1)$ in $L^1(Q)\times L^1(Q)$. Moreover we have
\begin{eqnarray*}
\mu_a(x_0) & \geq & \lim_{\r\to 0}\limsup_{\e\to 0}\int_Q\left(\a\,v_{\e,\r}^2(y)|\nabla w_{\e,\r}(y)|^2+\frac{\b}{2\sqrt{2}}\,\left(\frac{(v_{\e,\r}(y)-1)^2}{\e}+\frac{\e^3}{\r^4}|\Delta v_{\e,\r}(y)|^2\right)\right)\dy\\
& \geq & \lim_{\r\to 0}\limsup_{\e\to 0}\E_\e(w_{\e,\r},v_{\e,\r},Q),
\end{eqnarray*}
where in the last inequality we have used that $1/\r>1$ for $\varrho$ small.

By a standard diagonalization argument we can find two positive vanishing sequences $(\e_h),(\r_h)$ such that $(w_{\e_h,\r_h},v_{\e_h,\r_h})\rightarrow(w_0,1)$ in $L^1(Q)\times L^1(Q)$ as $h\to +\infty$, and
\begin{equation*}
\mu_a(x_0)\geq \liminf_{h\to+\infty}\E_{\e_h}(w_{\e_h,\r_h},v_{\e_h,\r_h},Q).
\end{equation*}
Let $Q'\subset\subset Q$; then, gathering Proposition \ref{ellipt-reg}(i) and Theorem \ref{gconv-nD}, we get
\begin{eqnarray*}
\mu_a(x_0) & \geq & \liminf_{h\to+\infty}\int_{Q'}\left(\a\,v_{\e_h,\r_h}^2|\nabla w_{\e_h,\r_h}|^2+\frac{\b}{2\sqrt{2}}\,\left(\frac{(v_{\e_h,\r_h}-1)^2}{\e_h}+c(Q,Q')\e_h^3|\nabla ^2v_{\e_h,\r_h}|^2\right)\right)\dy\\
& \geq & \a\,|\nabla u(x_0)|^2\mathcal{L}^n(Q'),
\end{eqnarray*}
and (\ref{est_grad}) follows letting $Q'\nearrow Q$.

We now prove (\ref{est_jump}). Let $x_0\in S_u$ be a Lebesgue point for $\mu$ with respect to $\mathcal{H}^{n-1}\llcorner S_u$\ie
\begin{equation}\label{Lpj}
\mu_J(x_0)=\lim_{\r\to 0}\frac{\mu(Q^\nu_\r(x_0))}{\mathcal{H}^{n-1}(Q^\nu_\r(x_0)\cap S_u)}=\lim_{\r\to 0}\frac{\mu(Q^\nu_\r(x_0))}{\r^{n-1}},
\end{equation}
where $\nu:=\nu_u(x_0)$. The Besicovitch derivation Theorem ensures that \eqref{Lpj} holds true for $\mathcal{H}^{n-1}\hbox{-}$a.e. $x_0\in S_u$. By the definition of approximate discontinuity point, we can assume that in addition
\begin{equation}\label{app_jump}
\lim_{\r\to 0}\int_{(Q^\nu_\r(x_0))^\pm}|u(x)-u^\pm(x_0)|\dx=0,
\end{equation}
where $(Q^\nu_\r(x_0))^\pm:=\{x\in (Q^\nu_\r(x_0))\colon \pm\langle x-x_0,\nu\rangle > 0\}$.

By following the same argument as before, we get
\begin{eqnarray*}
\mu_J(x_0) & = & \lim_{\r\to 0}\frac{1}{\r^{n-1}}\int_{\overline{Q^\nu_\r(x_0)}}d\mu\geq\lim_{\r\to 0}\limsup_{\e\to 0}\frac{1}{\r^{n-1}}\,\mu_\e(Q^\nu_\r(x_0))\\
& = & \lim_{\r\to 0}\limsup_{\e\to 0}\frac{1}{\r^{n-1}}\int_{Q^\nu_\r(x_0)}\left(\a\,v_\e^2(x)|\nabla u_\e(x)|^2+\frac{\b}{2\sqrt{2}}\,\left(\frac{(v_\e(x)-1)^2}{\e}+\e^3|\Delta v_\e(x)|^2\right)\right)\dx\\
& = & \lim_{\r\to 0}\limsup_{\e\to 0}\r\int_{Q^\nu}\bigg(\a\,v_\e^2(x_0+\r y)|\nabla u_\e(x_0+\r y)|^2\\
& & \qquad\qquad\qquad\qquad+\frac{\b}{2\sqrt{2}}\,\left(\frac{(v_\e(x_0+\r y)-1)^2}{\e}+\e^3|\Delta v_\e(x_0+\r y)|^2\right)\bigg)\dy\\
& = & \lim_{\r\to 0}\limsup_{\e\to 0}\int_{Q^\nu}\bigg(\frac{\a}{\r}\,v_{\e,\r}^2(y)|\nabla u_{\e,\r}(y)|^2\\
&&\qquad\qquad\qquad\qquad+\frac{\b}{2\sqrt{2}}\,\bigg(\frac{\r}{\e}(v_{\e,\r}(y)-1)^2+\left(\frac{\e}{\r}\right)^3|\Delta v_{\e,\r}(y)|^2\bigg)\bigg)\dy\\
& \geq & \lim_{\r\to 0}\limsup_{\e\to 0} \E_{\e/\r}(u_{\e,\r},v_{\e,\r},Q^\nu),
\end{eqnarray*}
where we set $u_{\e,\r}(y):=u_\e(x_0+\r y)$, $v_{\e,\r}(y):=v_\e(x_0+\r y)$. Notice that in view of (\ref{app_jump}), letting first $\e\to 0$ and then $\r\to 0$, we get $(u_{\e,\r},v_{\e,\r})\rightarrow(u_0,1)$ in $L^1(Q^\nu)\times L^1(Q^\nu)$ where
\begin{equation*}
u_0(x):=\begin{cases}
u^+(x_0) & \text{if}\; \langle x-x_0,\nu\rangle\geq 0 ,\\
u^-(x_0) & \text{if}\; \langle x-x_0,\nu\rangle < 0.
\end{cases}
\end{equation*}
Then, a diagonalization argument provides us with two positive vanishing sequences $(\e_h),(\r_h)$ such that $\sigma_h:=\frac{\e_h}{\r_h}\to 0$ as $h\to +\infty$, $(u_{\e_h,\r_h},v_{\e_h,\r_h})\rightarrow(u_0,1)$ in $L^1(Q^\nu)\times L^1(Q^\nu)$ as $h\to +\infty$, and
\begin{equation}\label{l:diag}
\mu_J(x_0)\geq \lim_{h\to+\infty}\E_{\sigma_h}(u_{\e_h,\r_h},v_{\e_h,\r_h},Q^\nu).
\end{equation}
Set $u_h:=u_{\e_h,\r_h}$, $v_h:=v_{\e_h,\r_h}$. 

Since $\E_{\sigma_h}$ is invariant under translations in $u$ and under rotations in $u$ and $v$, it is enough to bound from below the right hand side in \eqref{l:diag} when $\nu=e_n$ and $u_h\to u_t:=t\chi_{\{x_n\geq 0\}}$, $t:=u^+(x_0)-u^-(x_0)$. 

To this end let $\d>0$ and $Q(\d):=(-1/2+\d,1/2-\d)^n$; then,
\begin{equation}\label{lf_lap-d}
\mu_J(x_0)\geq\lim_{h\to +\infty}\E_{\sigma_h}(u_h,v_h,Q)\geq\lim_{h\to+\infty}\E_{\sigma_h}(u_h,v_h,Q(\d)),
\end{equation}
for every $\d>0$ small. We now show that
\begin{equation}\label{liminf_lap}
\lim_{h\to +\infty}\E_{\sigma_h}(u_h,v_h,Q(\d))\geq \b.
\end{equation}
The proof now follows the line of that of \cite[Lemma 3.4]{HPS} where the asymptotic behaviour of a variant of the Modica-Mortola functional is investigated. The idea is to estimate from below the functionals $\E_{\sigma_h}$ in a way which allows us to reduce to the one-dimensional lower bound proved in Theorem \ref{gconv-1D}.

In order to not overburden notation we now drop the index $h$ for the sequences of functions as in \eqref{lf_lap-d}. We assume moreover that $v\in C^\infty(\overline{Q(\d)})$ and we write
\begin{equation*}
\Delta_xv=v_{zz}+\Delta_y v,\quad y\in Q'(\d),\,z\in(-1/2+\d,1/2-\d),
\end{equation*}
with $Q'(\d):=(-1/2+\d,1/2-\d)^{n-1}$. Then, we have 
\begin{equation}\label{est7}
\E_{\sigma_h}(u,v,Q(\d))=\int_{-\frac{1}{2}+\d}^{\frac{1}{2}-\d}\int_{Q'(\d)}\left(\a\,v^2\left(u_z^2+|\nabla_y u|^2\right)+\frac{\b}{2\sqrt{2}}\,\left(\frac{(v-1)^2}{\sigma_h}+\sigma_h^3\,|v_{zz}+\Delta_y v|^2\right)\right)\dy\dz.
\end{equation}
We are going to estimate the right-hand side of (\ref{est7}) from below with a functional that no longer contains partial derivatives with respect to $y$. Since the term involving $\nabla_y u$ is non-negative the only term that we have to estimate is the one containing $\Delta_y v$. For $\eta\in C_0^\infty(Q(\d))$, we write
\begin{eqnarray}\nonumber
&\quad &\int_{-\frac{1}{2}+\d}^{\frac{1}{2}-\d}\int_{Q'(\d)}|v_{zz}+\Delta_y v|^2\eta^2\dy\dz\\\label{est8}
&=&\int_{-\frac{1}{2}+\d}^{\frac{1}{2}-\d}\int_{Q'(\d)}v_{zz}^2\,\eta^2\dy\dz+\int_{-\frac{1}{2}+\d}^{\frac{1}{2}-\d}\int_{Q'(\d)}\left(|\Delta_y v|^2\,\eta^2+2\,v_{zz}\,\eta^2\,\Delta_yv\right)\dy\dz.
\end{eqnarray}
We first estimate the last term in (\ref{est8}) from below. We have
\begin{eqnarray*}
&\quad &\int_{-\frac{1}{2}+\d}^{\frac{1}{2}-\d}\left(\int_{Q'(\d)}v_{zz}\,\eta^2\,\Delta_yv\dy\right)\dz\\
&=&\ds-\int_{-\frac{1}{2}+\d}^{\frac{1}{2}-\d}\int_{Q'(\d)}v_z\,\eta^2\,\Delta_yv_z\dy\dz-2\,\int_{-\frac{1}{2}+\d}^{\frac{1}{2}-\d}\int_{Q'(\d)}v_z\,\eta\,\eta_z\,\Delta_yv\dy\dz\\
&=&\ds\int_{-\frac{1}{2}+\d}^{\frac{1}{2}-\d}\int_{Q'(\d)}\left(|\nabla_yv_z|^2\,\eta^2+2\,v_z\,\eta\,\langle\nabla_yv_z,\nabla_y\eta\rangle\right)\dy\dz-2\,\int_{-\frac{1}{2}+\d}^{\frac{1}{2}-\d}\int_{Q'(\d)}v_z\,\eta\,\eta_z\,\Delta_yv\dy\dz.
\end{eqnarray*}
By Young's inequality we get
\begin{equation*}
2\,\int_{-\frac{1}{2}+\d}^{\frac{1}{2}-\d}\int_{Q'(\d)}v_z\,\eta\,\langle\nabla_y v_z,\nabla_y\eta\rangle \dy\dz\leq\int_{-\frac{1}{2}+\d}^{\frac{1}{2}-\d}\int_{Q'(\d)}|\nabla_yv_z|^2\,\eta^2\dy\dz+\int_{-\frac{1}{2}+\d}^{\frac{1}{2}-\d}\int_{Q'(\d)}v_z^2\,|\nabla_y\eta|^2\dy\dz
\end{equation*}
and
\begin{equation*}
2\,\int_{-\frac{1}{2}+\d}^{\frac{1}{2}-\d}\int_{Q'(\d)}v_z\,\eta\,\eta_z\,\Delta_y v\dy\dz\leq\frac{1}{2}\int_{-\frac{1}{2}+\d}^{\frac{1}{2}-\d}\int_{Q'(\d)}|\Delta_y v|^2\,\eta^2\dy\dz+2\,\int_{-\frac{1}{2}+\d}^{\frac{1}{2}-\d}\int_{Q'(\d)}v_z^2\,\eta_z^2\dy\dz.
\end{equation*}
Using in (\ref{est8}) the two bounds as above we obtain 
\begin{equation}\label{est9}
\int_{-\frac{1}{2}+\d}^{\frac{1}{2}-\d}\int_{Q'(\d)}\left(|\Delta_yv|^2+2\,v_{zz}\,\Delta_yv\right)\eta^2\dy\dz\geq -\int_{-\frac{1}{2}+\d}^{\frac{1}{2}-\d}\int_{Q'(\d)}v_z^2\left(4\,\eta_z^2+2\,|\nabla_y\eta|^2\right)\dy\dz.
\end{equation}
Thus, by (\ref{est7})-(\ref{est9}), we get
\begin{equation*}
\sigma_h^3\int_{-\frac{1}{2}+\d}^{\frac{1}{2}-\d}\int_{Q'(\d)}|v_{zz}+\Delta_yv|^2\,\eta^2\dy\dz\geq \sigma_h^3\int_{-\frac{1}{2}+\d}^{\frac{1}{2}-\d}\int_{Q'(\d)}v_{zz}^2\eta^2\dy\dz-c(\eta)\sigma_h^3\int_{Q(\d)}v_z^2 \dx.
\end{equation*}
Hence, if we assume in addition that $\|\eta\|_\infty\leq 1$, we can conclude that
\begin{eqnarray*}
\E_{\sigma_h}(u,v,Q(\d)) & \geq & \int_{-\frac{1}{2}+\d}^{\frac{1}{2}-\d}\int_{Q'(\d)}\left(\a\,v^2|u_z|^2+\frac{\b}{2\sqrt{2}}\,\left(\frac{(v-1)^2}{\sigma_h}+\sigma_h^3\,v_{zz}^2\right)\right)\eta^2\dy\dz\\
&&-\frac{\b}{2\sqrt{2}}\,c(\eta)\sigma_h^3\int_{Q(\d)}v_z^2 \dx,
\end{eqnarray*}
which holds also true for $v\in W^{2,2}(\O)$, by virtue of the density of $C^\infty(Q)\cap W^{2,2}(\O)$ in $W^{2,2}(\O)$. 

Appealing to Proposition \ref{interpol-ineq} and Proposition \ref{ellipt-reg}(i) we have
\begin{eqnarray*}
\frac{\b}{2\sqrt{2}}\,\sigma_h^3\int_{Q(\d)}(v)_z^2\dx &\leq & c\,\sigma_h^2\frac{\b}{2\sqrt{2}}\,\left(\int_{Q(\d)}\frac{(v-1)^2}{\sigma_h}+\sigma_h^3|\nabla^2 v|^2\dx\right)\\
& \leq & \sigma_h^2\, c(Q,Q(\d))\,\E_{\sigma_h}(u,v,Q);
\end{eqnarray*}
therefore
\begin{eqnarray*}
&\quad& \lim_{h\to+\infty}\E_{\sigma_h}(u_h,v_h,Q(\d))
\\ 
& \geq & \liminf_{h\to+\infty}\int_{-\frac{1}{2}+\d}^{\frac{1}{2}-\d}\int_{Q'(\d)}\left(\a\,v_h^2(u_h)_z^2+\frac{\b}{2\sqrt{2}}\,\left(\frac{(v_h-1)^2}{\sigma_h}+\sigma_h^3(v_h)_{zz}^2\right)\right)\eta^2\dy\dz,
\end{eqnarray*}
for every $\eta \in C^\infty_0(Q(\d))$. Hence if we choose $\hat\d>\d$ and $\eta \in C^\infty_0(Q(\d))$ such that $\eta=1$ in $Q(\hat\d)\subset \subset Q(\d)$, invoking Fatou's Lemma and Theorem \ref{gconv-1D} we get
\begin{eqnarray*}
\mu_J(x_0) & \geq & \lim_{h\to+\infty}\E_{\sigma_h}(u_h,v_h,Q(\hat\d))\\
& \geq & \liminf_{h\to+\infty}\int_{Q'(\hat\d)}\int_{-\frac{1}{2}+\hat\d}^{\frac{1}{2}-\hat\d}\left(\a\,v_h^2(u_h)_z^2+\frac{\b}{2\sqrt{2}}\,\left(\frac{(v_h-1)^2}{\sigma_h}+\sigma_h^3(v_h)_{zz}^2\right)\right)\dz\dy\\
& \geq & \int_{Q'(\hat\d)}\liminf_{h\to +\infty}\int_{-\frac{1}{2}+\hat\d}^{\frac{1}{2}-\hat\d}\left(\a\,v_h^2(u_h)_z^2+\frac{\b}{2\sqrt{2}}\,\left(\frac{(v_h-1)^2}{\sigma_h}+\sigma_h^3(v_h)_{zz}^2\right)\right)\dz\dy\\
& \geq & \b\,\mathcal{H}^{n-1}(Q'(\hat\d)).
\end{eqnarray*}
Then, (\ref{est_jump}) follows by letting $\hat\d\to 0$.

If $u\in GSBV^2(\O)$ the thesis follows by a standard truncation argument. In fact if $u^M:=(u\land M)\vee(-M)$, then $u^M\in SBV^2(\O)$ for all $M\in \mathbb{N}$. Then, appealing to the lower-semicontinuity of $\MS$ and noticing that $\E_\e(\cdot,v)$ (and hence $\E'(\cdot,v)$) decreases by truncation, we immediately get
\begin{equation*}
\E'(u,1)\geq\liminf_{M\to +\infty}\MS(u^M,1)\geq \MS(u,1).
\end{equation*}
\end{proof}
The following result holds true.
\begin{proposition}
We have
\begin{equation}\label{gconv_GSBV}
\G\hbox{-}\lim_{\e\to 0}\E_\e(u,1)=\MS(u,1)
\end{equation}
for every $u\in GSBV^2(\O)$.
\end{proposition}
\begin{proof}
The lower bound inequality is a consequence of Proposition \ref{lower_bd} while the upper bound can be proved by taking the same recovery sequence as in Theorem \ref{gconv-nD}.
\end{proof}

Now define
\begin{equation*}
\tilde{\E}_\e(u,v):=
\begin{cases}
\E_\e(u,v) & \text{if}\; u\in W^{1,2}(\O), (v-1)\in W^{1,2}_0(\O)\cap W^{2,2}(\O)\\
&\text{and}\; v\nabla u\in L^2(\O;\R^n),
\cr\cr
+\infty & \text{otherwise in}\;L^1(\O)\times L^1(\O),
\end{cases}
\end{equation*}
then we can prove the following $\G$-convergence result.
\begin{theorem}\label{conv_lap}
The sequence $(\tilde{\E}_\e)$ $\G\hbox{-}$converges, with respect to the $(L^1(\O)\times L^1(\O))$-topology, to the functional $\MS$ as in (\ref{Mu-Sh}).
\end{theorem}
\begin{proof}
The $\G$-convergence result is a straightforward consequence of Proposition \ref{gconv_GSBV} once we notice that thanks to the boundary conditions satisfied by $v$, we can now invoke Proposition \ref{ellipt-reg}(ii) to get
\begin{equation}\label{lb_ellreg}
\tilde{\E}_\e(u,v)\geq \int_\O\left(\a\,v^2|\nabla u|^2+\frac{\b}{2\sqrt{2}}\,\left(\frac{(v-1)^2}{\e}+c(\O)\e^3|\nabla^2 v|^2\right)\right)\dx,
\end{equation}
which together with (\ref{bound2}) allows us to conclude that the domain of the $\Gamma$-limit is $GSBV^2(\O) \times \{v=1 \; \text{a.e. in}\; \O\}$.
\end{proof}

\begin{remark}
The $C^2$-regularity of $\partial \O$ is only used to invoke Proposition \ref{ellipt-reg}(ii) in order to obtain the estimate from below (\ref{lb_ellreg}). We notice however that for $n=2$, which is the interesting case in numerical simulations, Proposition \ref{ellipt-reg}(ii) holds also true in bounded polygonal open sets (see e.g. \cite[Theorem 2.2.3]{Gr2}). 
\end{remark}

\subsection{Convergence of minimization problems} For every $(u,v)\in W^{1,2}(\O)\times W^{2,2}(\O)$ consider the functionals
\begin{equation}\label{funct2}
\tilde\E_\e(u,v)+\eta_\e\int_\O |\nabla u|^2\dx+\gamma\int_\O |u-g|^2\dx.
\end{equation} 
Appealing to Theorem \ref{conv_lap} and to the fundamental property of $\G$-convergence, also in this case we can prove a result on the convergence of associated minimization problems.  

\begin{theorem}\label{eqcoerc-1}
For every fixed $\e>0$ there exists a minimizing pair $(\tilde u_\e,\tilde v_\e)$ for the problem
\begin{equation*}
M_\e:=\inf\left\{\tilde\E_\e(u,v)+\eta_\e\int_\O |\nabla u|^2\dx+\gamma\int_\O |u-g|^2\dx\colon (u,v)\in W^{1,2}(\O)\times W^{2,2}(\O)\right\}.
\end{equation*} 
Moreover, up to subsequences, $(\tilde u_\e,\tilde v_\e)\to (u,1)$ in $L^1(\O)\times L^1(\O)$ where $u$ is a solution to
\begin{equation}
{M}:=\min\left\{\MS(u,1)+\gamma\int_\O|u-g|^2\dx \colon u\in GSBV^2(\O)\right\};
\end{equation}
if $n=1$ then $u\in SBV^2(\O)$. Further, $M_\e \to M$ as $\e\to 0$.
\end{theorem}
\begin{proof}
The proof follows the line of that of Theorem \ref{eqcoerc} once we notice that the convergence \eqref{convergenze} is now ensured by Proposition \ref{ellipt-reg}(ii). 
\end{proof}
%

\section{Numerical Results}

\noindent In this section we discuss some numerical results for the second-order approximation $\tilde\E_\e(u,v)$\ie we minimize  \eqref{funct2} for reasons of practicability. First of all, the discretization of the Laplacian is more straightforward and leads to more compact schemes compared to the discretization of the Hessian involving mixed derivatives. As we have noticed in Remark \ref{tildeFepsilon} also the Hessian penalization can be rewritten into a Laplacian one under certain conditions, which are however not suitable for a numerical implementation since we need to enforce a solution in $W_0^{2,2}$. The latter means we have to implement simultaneous Dirichlet and Neumann boundary conditions, which is not feasible in finite difference or finite element discretizations without enforcing additional constraints. A second argument comes from the comparison with the first-order Ambrosio-Tortorelli functional, which already includes a Laplacian in the optimality condition, while the optimality condition for \eqref{funct2} changes only to a concatenation of two Laplacians. Hence, the modification of a code for the Ambrosio-Tortorelli functional to the second-order version \eqref{funct2} is straightforward and allows for a comparison of computational efficiency.

We shall report on several computational experiments, starting with simple synthetic images that allow for a detailed study of fine properties such as the realization of the optimal profile already for a rather low number of pixels. Subsequently we investigate the behaviour on a set of natural and biomedical images, highlighting several differences of the second-order approach to the classical Ambrosio-Tortorelli functional. All experiments are carried out in two spatial dimensions, but we mention that extensions to volume data sets are obvious. Since a combination of the parameters $\alpha$, $\beta$ and $\gamma$ is redundant, we choose $\beta=0.3$ in all experiments.

\subsection{Numerical Solution}

In order to minimize the functional in \eqref{funct2} we follow the common strategy of iterative minimization. Hence, given an iterate $(u^k,v^k)$ we subsequently compute
\begin{eqnarray*}
	v^{k+1} &\in& \text{arg}\min_v \tilde\E_\e(u^k,v) \\
  u^{k+1} &\in& \text{arg}\min_u \tilde\E_\e(u,v^{k+1})+\eta_\e\int_\O |\nabla u|^2\dx+\gamma\int_\O |u-g|^2\dx.
\end{eqnarray*}
This yields a descent method for the overall functional, which is based on solving two quadratic minimization problems, respectively the corresponding linear optimality systems in each case. The same method is used to minimize the original Ambrosio-Tortorelli functional in \eqref{ATfunct}. 

The linear equation to be solved for $u^{k+1}$ is in both cases
\begin{equation*}
	- 2 \a \nabla \cdot ((v^{k+1})^2 \nabla u^{k+1})- \eta_\e \Delta u^{k+1} + \gamma u^{k+1} = \gamma g.
\end{equation*}
The linear equation for $v^{k+1}$ is given by 
\begin{equation*}
	2 \a |\nabla u^k|^2 v^{k+1}  +\frac{\b}{\sqrt{2}\e} v^{k+1} + \frac{\b \e^3}{\sqrt{2}} \Delta \Delta v^{k+1}  = \frac{\b}{\sqrt{2}\e} 
\end{equation*}
in the case of the new second-order functional, respectively by
\begin{equation*}
	2 \a |\nabla u^k|^2 v^{k+1}  +\frac{\b}{\e} v^{k+1} -{\b \e} \Delta v^{k+1}  = \frac{\b}{\e} 
\end{equation*}
in the case of the Ambrosio-Tortorelli functional.


We discretize the functionals in \eqref{funct2} and \eqref{ATfunct} by standard finite differences on a rectangular grid, using one-sided (forward) differencing for the gradient and the adjoint one-sided differences for the divergence, hence the usual central differencing for the Laplacian.  For the discretized linear systems we use a direct solver. Convergence diagnostics and stopping rules are based on the relative size of the change in the images\ie  
\begin{equation*}
	e^k = \max\left\{ \frac{\Vert u^{k+1}- u^k\Vert_\infty}{\Vert u^{k+1}\Vert_\infty},\frac{\Vert v^{k+1}- v^k\Vert_\infty}{\Vert v^{k+1}\Vert_\infty} \right\}
\end{equation*}

\subsection{Simple Test Examples}

We start with a simple image showing a one-dimensional structure in the vertical direction. Figure \ref{1dfig} illustrates the results for the parameter settings $\alpha=10^{-2}, \gamma=10^{-3},\e=3*10^{-2}$, which yield a visually optimal result at the given image resolution. The segmentations obtained from the two models (see images of $v$ in the middle) are not distinguishable by eye, a fine difference can be seen however when exploring the level sets where $v$ is slightly larger than one in the second-order model. Note that due to a maximum principle respectively the monotone shape of the optimal profile in the Ambrosio-Tortorelli model the variable $v$ is always less or equal one, while the optimal profile in the second-order model exceeds one, hence the corresponding level set shall provide further information about edge location. This is illustrated in the right-most plot in Figure \ref{1dfig}, from which one observes that $v$ in the second-order model can provide an approximation of the edge set from both sides - an accurate reconstruction can be obtained as the midpoints between the local maxima. This behaviour is also present for larger values of $\e$ as illustrated in the supplementary material.

\begin{figure}[!!!h]
\includegraphics[width=0.2\textwidth,frame]{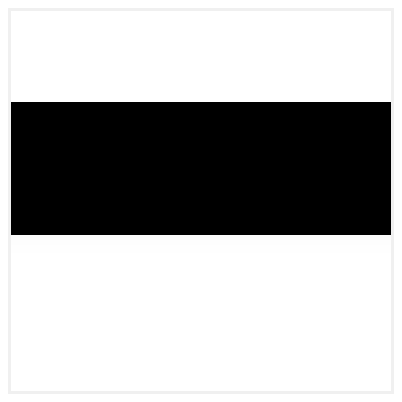}
\includegraphics[width=0.2\textwidth,frame]{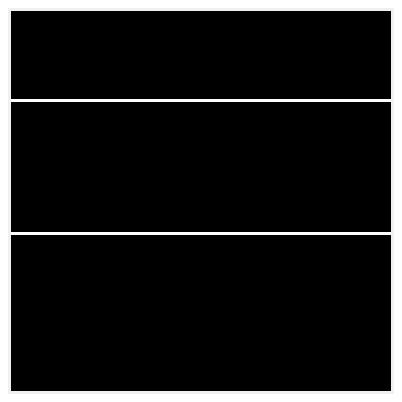}
\includegraphics[width=0.2\textwidth,frame]{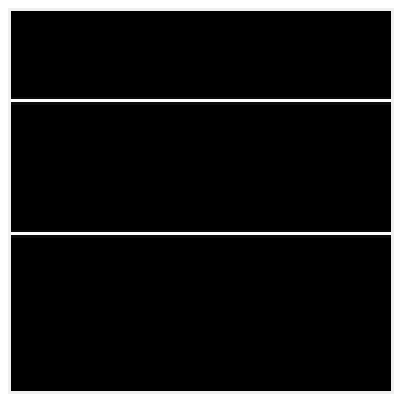}
\includegraphics[width=0.2\textwidth,frame]{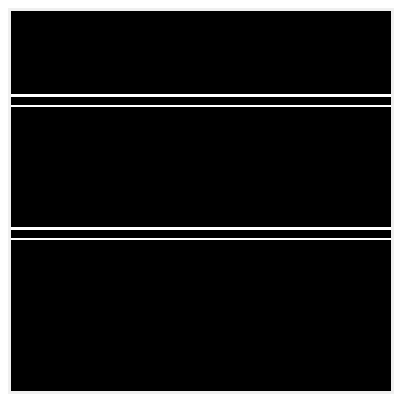}
\caption{One-dimensional structure (from left to right): Image $g$, resulting $v$ in the Ambrosio-Tortorelli model, resulting $v$ in the second-order model, binary plot of the level set $\{ v > 1.005 \}$ in the second-order model. }
\label{1dfig}
\end{figure}

With the same set of parameters we also compute minimizers for images of an ellipse with large ratio between the main directions, illustrated in Figure \ref{ellipsefig} and two overlapping circles, illustrated in Figure \ref{twocirclefig}. Both images are perturbed by additive Gaussian noise to test also the effect of noise on the results. We again plot the resulting minimizers $v$ for both models and the level set for the second-order models. The clean images $u$ yield no visible differences and are shown in the supplementary material for completeness. Overall we observe analogous behaviour as for the one-dimensional example, a remarkable fact is that the level sets of $v$ in the second-order model are able to provide a well separated segmentation of the overlap region.

\begin{figure}[!!!h]
\includegraphics[width=0.2\textwidth,frame]{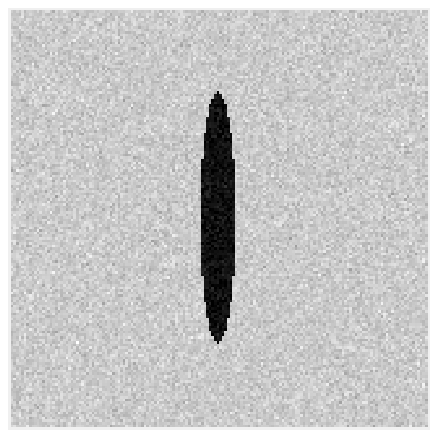}
\includegraphics[width=0.2\textwidth,frame]{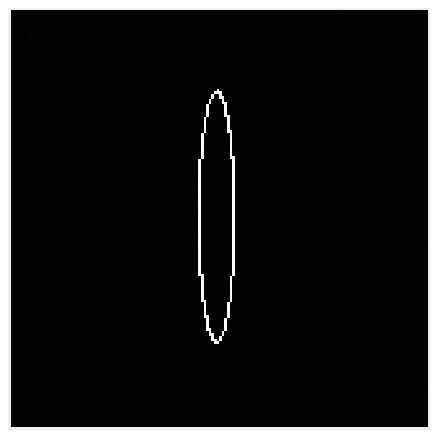}
\includegraphics[width=0.2\textwidth,frame]{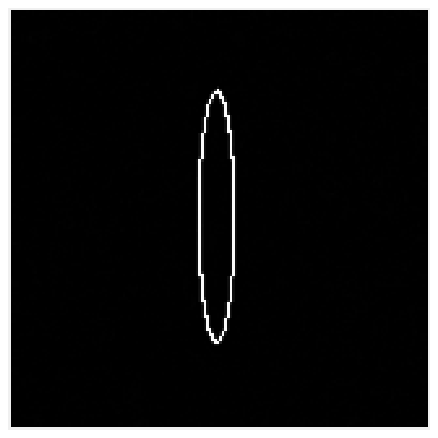}
\includegraphics[width=0.2\textwidth,frame]{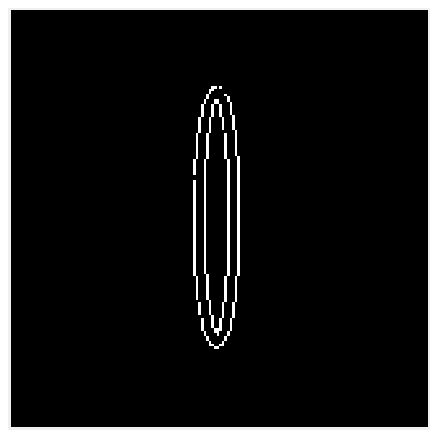}
\caption{Ellipse (from left to right): Image $g$, resulting $v$ in the Ambrosio-Tortorelli model, resulting $v$ in the second-order model, binary plot of the level set $\{ v > 1.005 \}$ in the second-order model. }
\label{ellipsefig}
\end{figure}

\begin{figure}[!!!h]
\includegraphics[width=0.2\textwidth,frame]{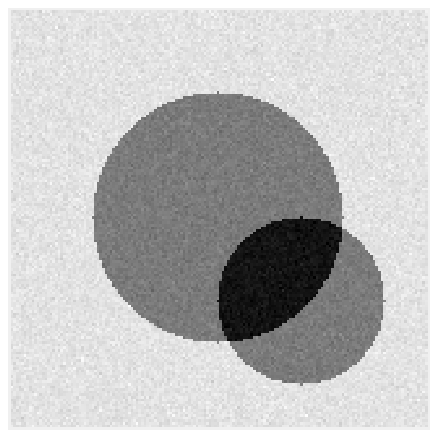}
\includegraphics[width=0.2\textwidth,frame]{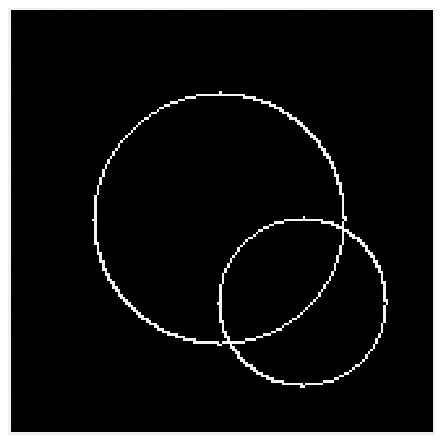}
\includegraphics[width=0.2\textwidth,frame]{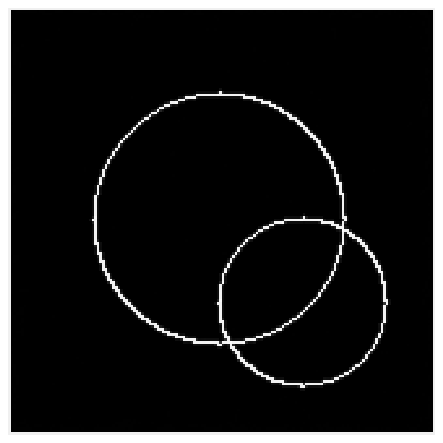}
\includegraphics[width=0.2\textwidth,frame]{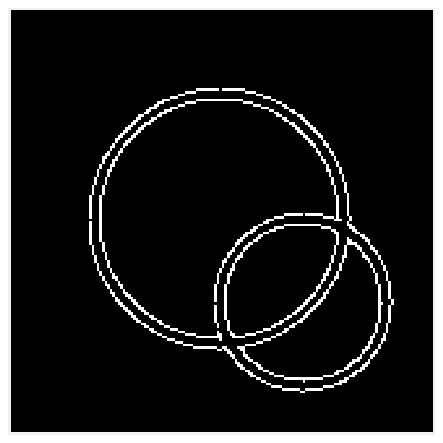}
\caption{Two circles (from left to right): Image $g$, resulting $v$ in the Ambrosio-Tortorelli model, resulting $v$ in the second-order model, binary plot of the level set $\{ v > 1.005 \}$ in the second-order model. }
\label{twocirclefig}
\end{figure}

\subsection{Natural and Biomedical Images}

In the following we report on the results of the second-order model for few examples of natural images as well as different kinds of microscopy images. A general observation on all those images is that at the given resolution (rather small compared to the size of the structures) and the hence possible choices of $\e$, the convergence to the optimal profile is far less pronounced than for the simple images above. Consequently the level set plot does not show an approximation of edges from both sides, for this reason we do not display the plots here. On the other hand we observe more interesting behaviour in the clean images $u$, which we also provide in the supplementary material. Let us mention that the $\Gamma$-convergence to the same minimizer does not mean that the original Ambrosio-Tortorelli functional and the second-order approach yield the same or very similar results on real images, which is due to many effects such as the given finite resolution of the image and choice of $\e$, the details of convergence in $\e$, as well as the level of convergence in the numerical minimization. 

The algorithm was tested on various natural images such as the Kodak image test set (see also the supplementary material). The most pronounced difference between the original Ambrosio-Tortorelli model and the novel second-order version concerns structures at a small scale, which is however still larger than the typical scale of noise. This is well illustrated in a portrait photograph of a person containing freckles (see Figure \ref{lentigginifig}, for parameters $\alpha=3*10^{-2},\gamma=3*10^{-3},\e=7*10^{-2}$ ). Such a behaviour is observed also for a wider range of parameters and seems clearly related to the stronger smoothing of the second-order model in higher frequencies. Another - at least visual - impression confirmed also in other results is that the contours being present in the results of both models appear smoother in the second-order model, which may be advantageous in many cases.

We also report on an effect we obtain for rather large choice of $\e$, i.e. rather far from convergence. This is illustrated in a phase-contrast microscopy image of a mitotic cell (cf. \cite{grah}). Such images are challenging for segmentation algorithms due to halo effects at the border of the cell, while the interior has similar grey value as the surrounding medium. Choosing $\e$ much larger than in the examples before ($\alpha=3*10^{-2},\gamma=3*10^{-3},\e=5*10^{-1}$) we obtain a contour $v$ that actually fills the whole interior. This effect is clearly benefitial for the second-order model, whose result allows a simple tracking of the cell.

\begin{figure}[!!!h]
\includegraphics[width=0.2\textwidth,frame]{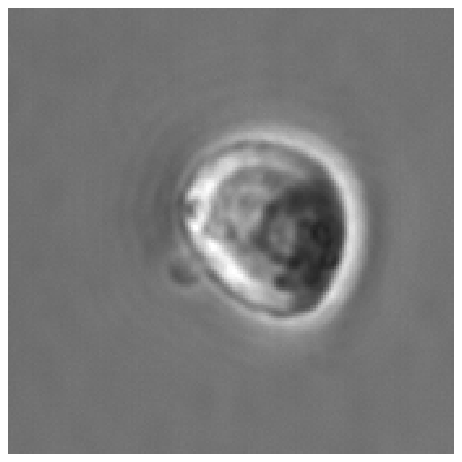}
\includegraphics[width=0.2\textwidth,frame]{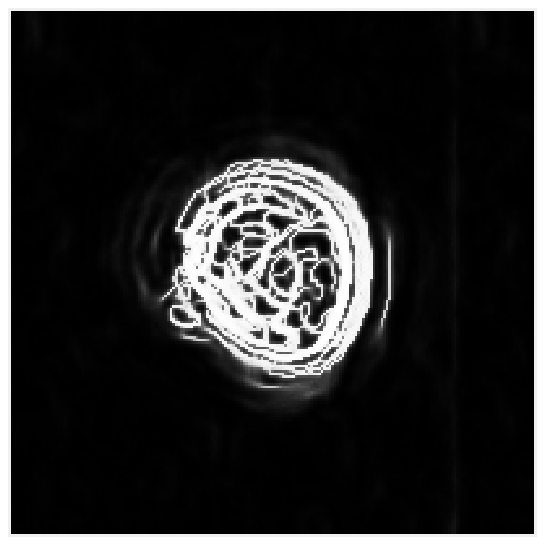}
\includegraphics[width=0.2\textwidth,frame]{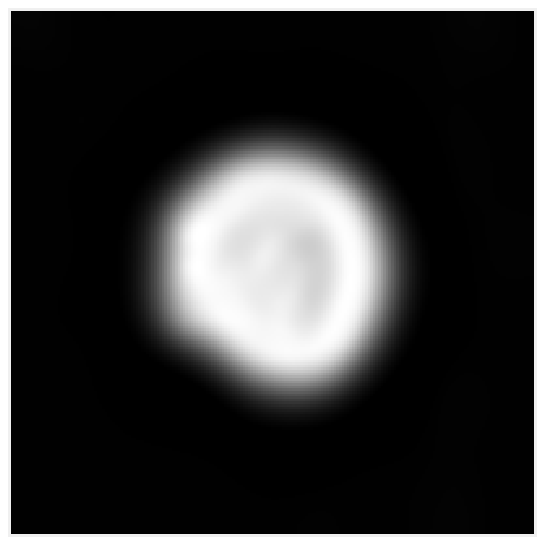}
\caption{Mitotic cell, from \cite{grah} (from left to right): Image $g$, resulting $v$ in the Ambrosio-Tortorelli model, resulting $v$ in the second-order model. }
\label{mitosisfig}
\end{figure}

We finally comment on the convergence behaviour of alternating minimization algorithms routinely used for approximating minimizers of the Ambrosio-Tortorelli functional. The observation made in the majority of our numerical experiments is that the number of iterations needed for fixed accuracy in the second-order model is at least comparable to those for the standard first-order version, in many cases the number of iterations is significantly reduced for the second-order model in particular for real images. In some parameter cases a visual comparison of results indicates that one obtains a global minimizer for the second-order model, while the iteration for the first-order model is stuck in a suboptimal local minimum, which is however difficult to verify. We refer to the supplementary material to a collection of computational investigations of convergence.

\section*{Acknowledgements}

The work of MB has been supported by ERC via Grant EU FP 7 - ERC Consolidator Grant 615216 LifeInverse and by the German Science Foundation DFG via  EXC 1003 Cells in Motion Cluster of Excellence, M\"unster, Germany. The authors thank S\o ren Udby (http://kanonfotografen.wordpress.com/) for permission to use the {\em Sisse} image.

\newpage


\appendix

\section{Supplementary material}

\subsection{Convergence Behaviour of Alternating Minimization Algorithms}

The convergence indicator $e^k$ is plotted vs. the number of iterations in several examples in Figure \ref{convergencefig} and \ref{kodakconvergencefig}. Figure \ref{nonconvergencefig} illustrates two cases of parameters where alternating minimization on the Ambrosio-Tortorelli functional did not converge.

\begin{figure}[!h]
\includegraphics[width=0.3\textwidth,frame]{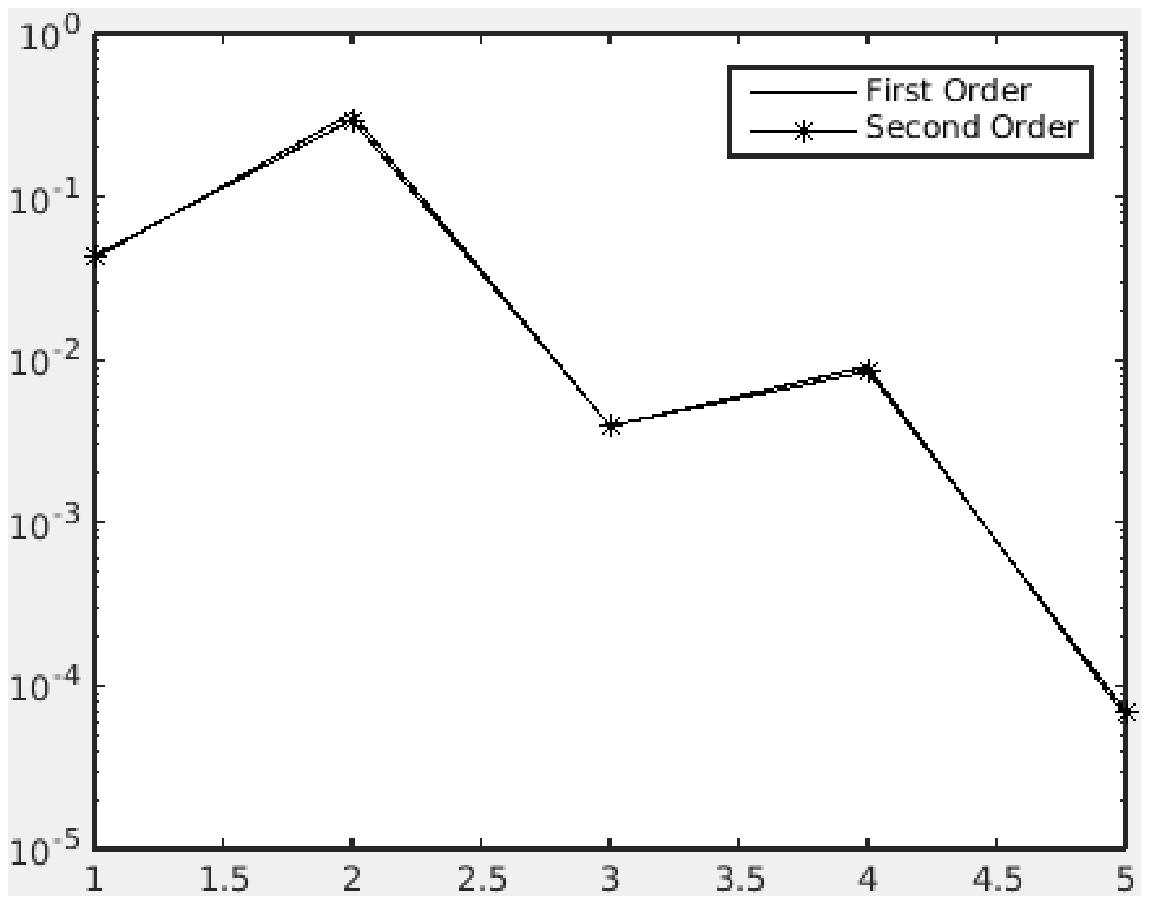}
\includegraphics[width=0.3\textwidth,frame]{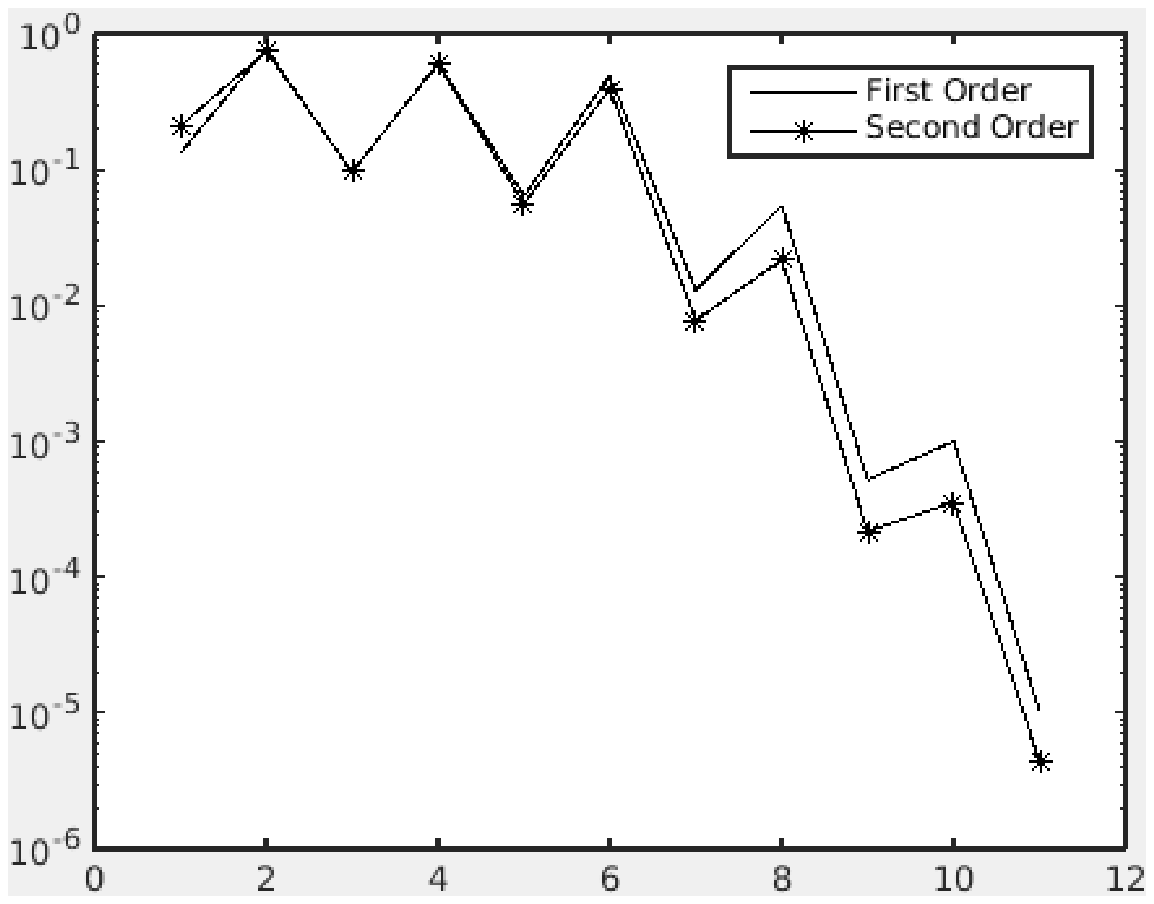}
\includegraphics[width=0.3\textwidth,frame]{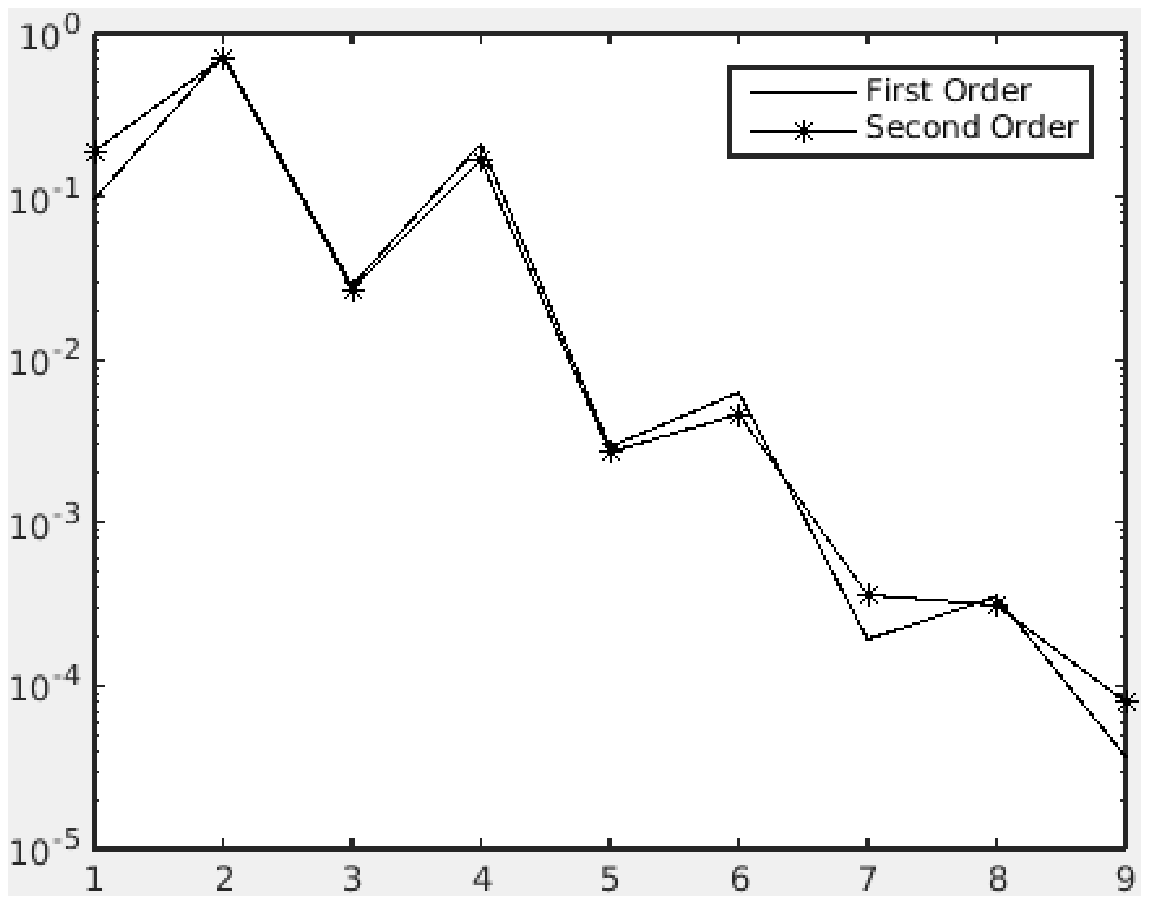}
\caption{Convergence history of $e^k$ vs. number of iterations $k$. One-dimensional example (left, $\alpha=10^{-2}, \gamma=10^{-3},\e=9*10^{-2}$), ellipse (middle, $\alpha=10^{-2}, \gamma=10^{-3},\e=3*10^{-2}$), two circles (right, $\alpha=10^{-2}, \gamma=10^{-3},\e=3*10^{-2}$). }
\label{convergencefig}
\end{figure}

\begin{figure}[!h]
\includegraphics[width=0.3\textwidth,frame]{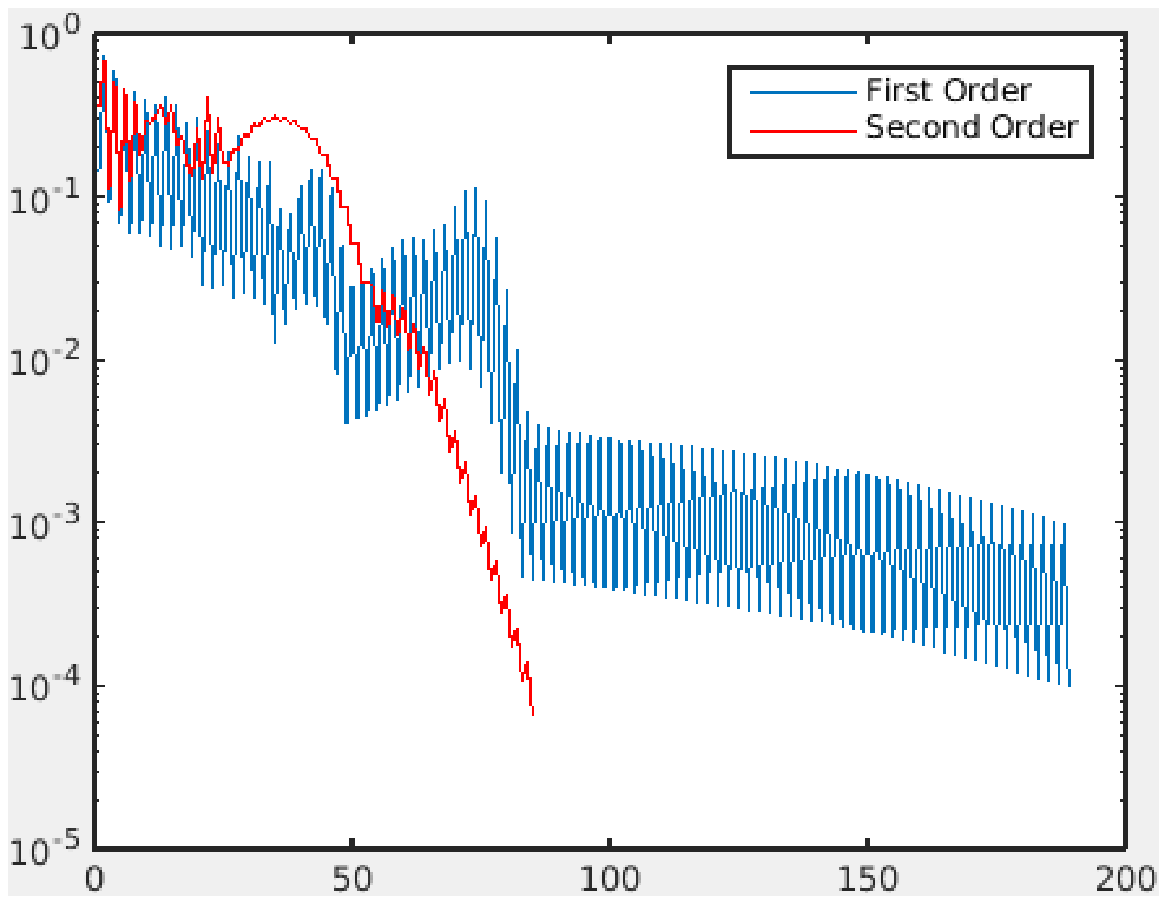}
\includegraphics[width=0.3\textwidth,frame]{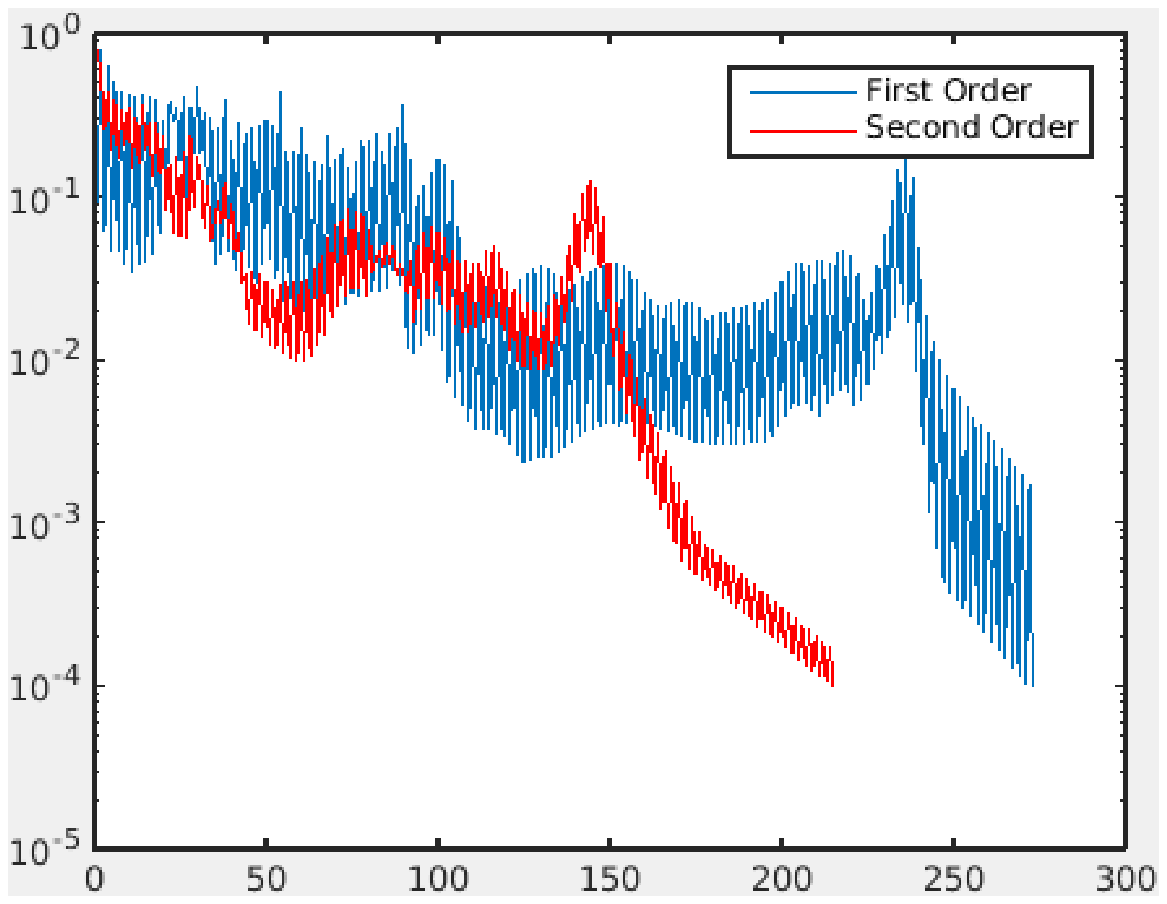}
\includegraphics[width=0.3\textwidth,frame]{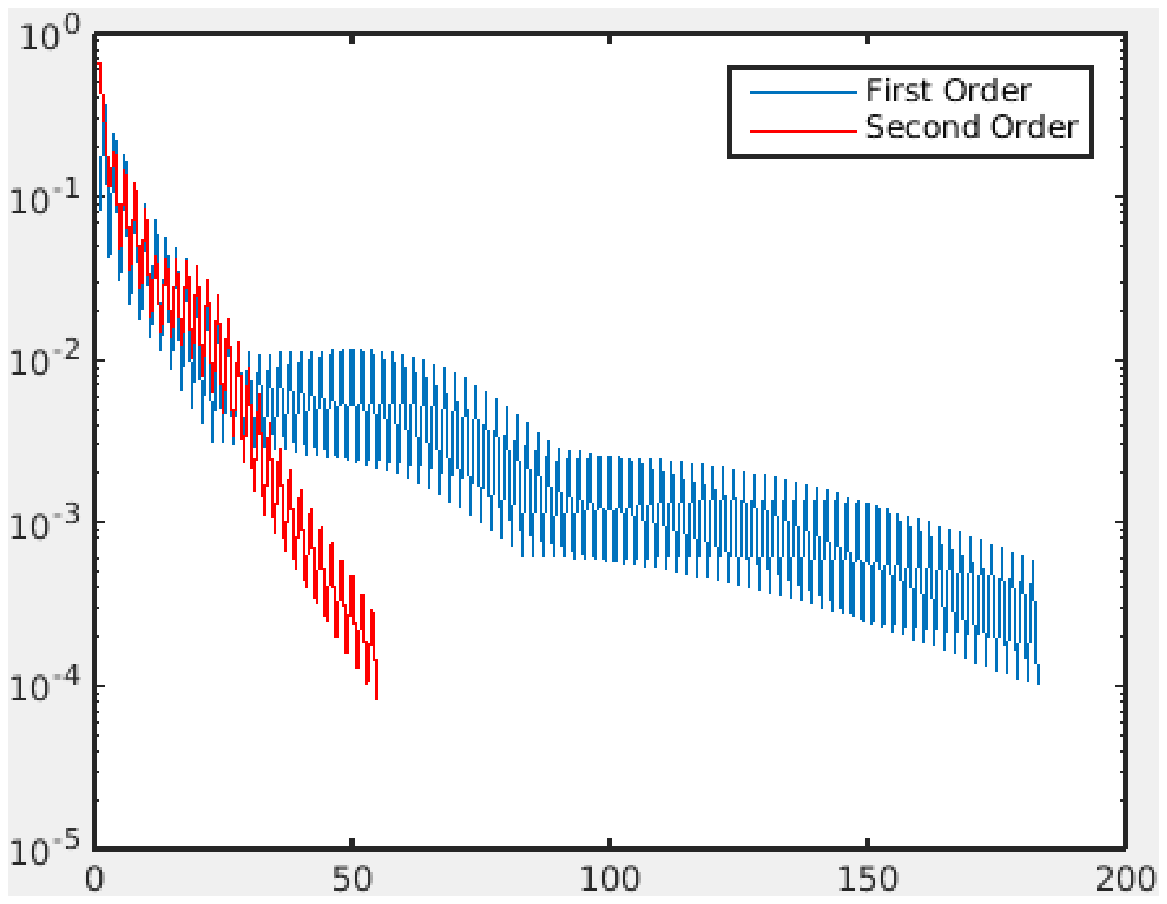}
\caption{Convergence history of $e^k$ vs. number of iterations $k$. Kodak image nr2  (left, $\alpha=10^{-2}, \gamma=10^{-3},\e=3*10^{-2}$), nr 7 (middle, $\alpha=10^{-2}, \gamma=10^{-3},\e=7*10^{-2}$), nr 23 (right, $\alpha=10^{-2}, \gamma=10^{-3},\e=7*10^{-2}$). }
\label{kodakconvergencefig}
\end{figure}

\begin{figure}
\includegraphics[width=0.4\textwidth,frame]{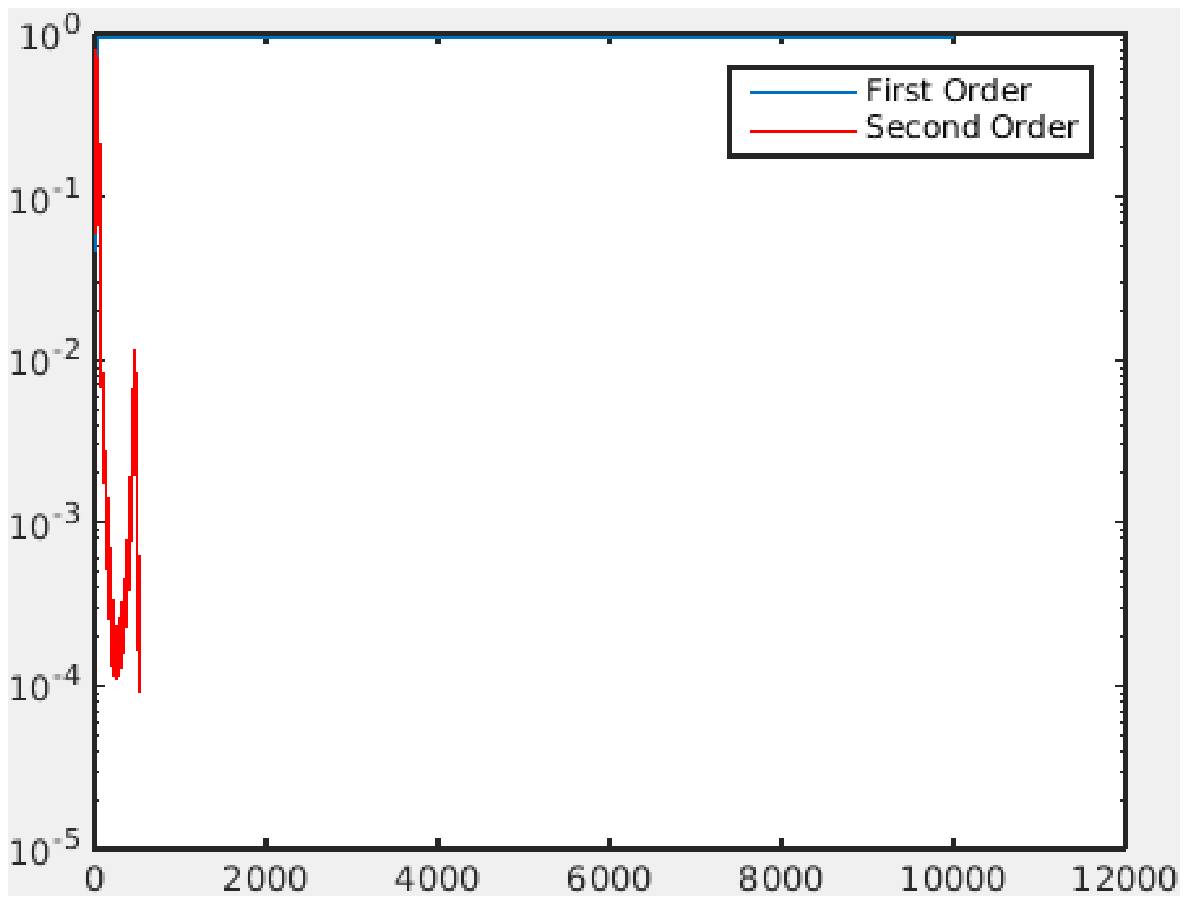}
\includegraphics[width=0.4\textwidth,frame]{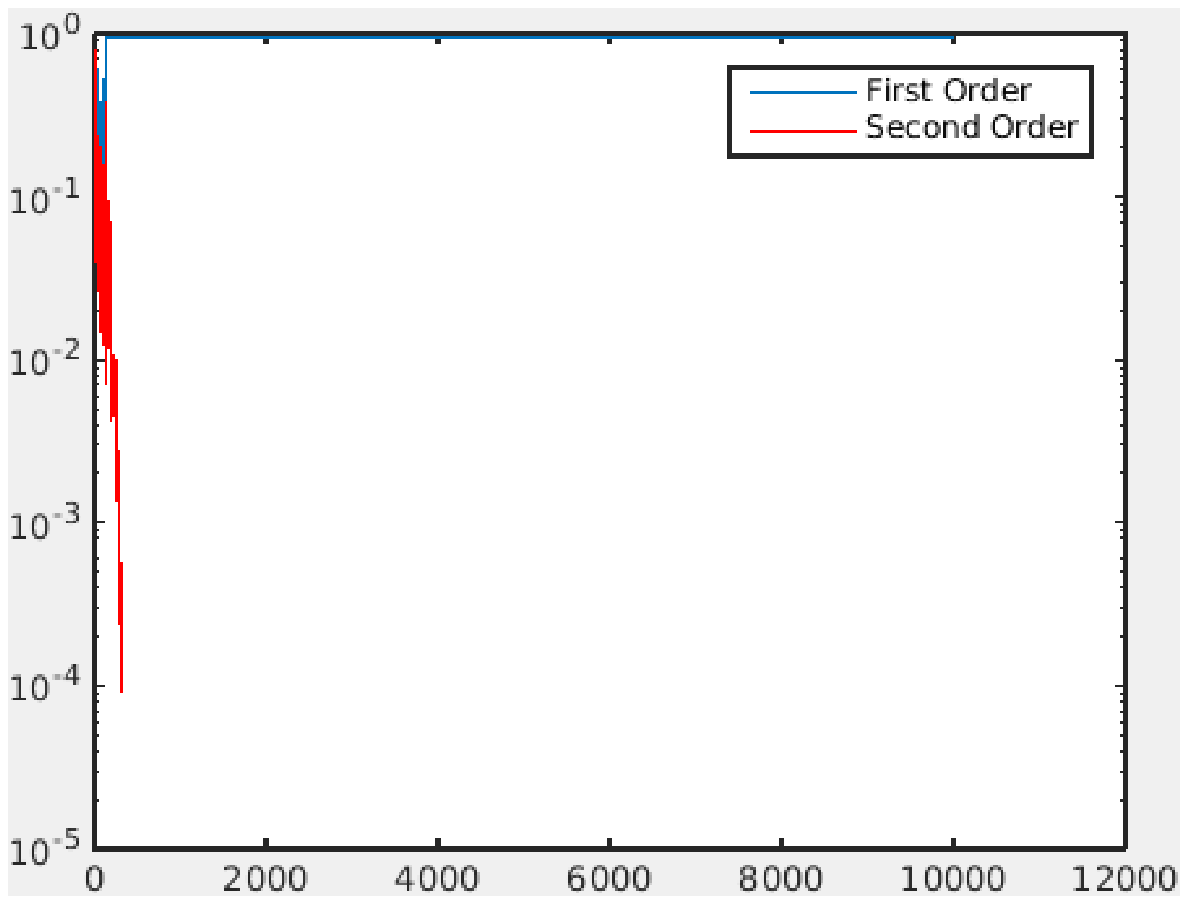}
\caption{Convergence history of $e^k$ vs. number of iterations $k$. Kodak image nr 7 (left, $\alpha=10^{-2},\gamma=7*10^{-4},\e=7*10^{-2}$), Kodak image nr 23 (right, $\alpha=10^{-2},\gamma=7*10^{-4},\e=7*10^{-2}$). }
\label{nonconvergencefig}
\end{figure}

\subsection{One-dimensional Structure}

Figure \ref{1dfiglarge} displays further results for the one-dimensional structure, parameter
$\alpha=10^{-2}, \gamma=10^{-3},\e=9*10^{-2}$.

\begin{figure}
\includegraphics[width=0.25\textwidth,frame]{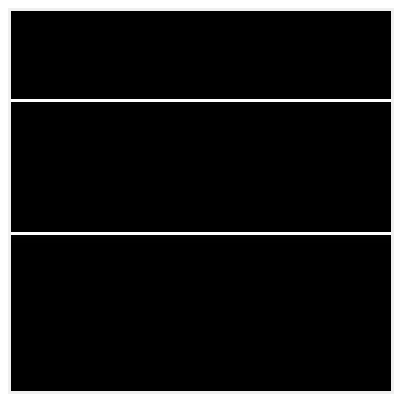}
\includegraphics[width=0.25\textwidth,frame]{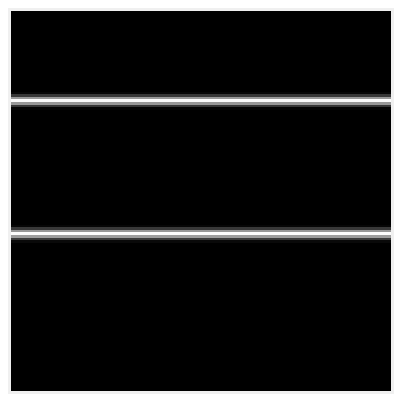}
\includegraphics[width=0.25\textwidth,frame]{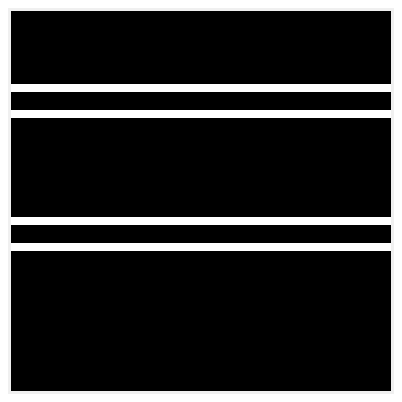}
\caption{One-dimensional structure (from left to right): Image $g$, resulting $v$ in the Ambrosio-Tortorelli model, resulting $v$ in the second-order model, binary plot of the level set $\{ v > 1.005 \}$ in the second-order model, $\e=9*10^{-2}$. }
\label{1dfiglarge}
\end{figure}

\subsection{Results on Kodak Images}

We display some examples of results on the Kodak images 2, 7, and 23, displayed in Figure \ref{kodakdatafig}.  Figures \ref{kodak2avfig}, \ref{kodak2bvfig}, \ref{kodak2cvfig} display the resulting segmentation $v$ for the Kodak image nr 2 with $\alpha=10^{-2}$ and different values of $\gamma$ and $\e$. Figures \ref{kodak7avfig} and \ref{kodak7bvfig} display the results for the Kodak image nr 7 with $\e=7*10^{-2}$ and different values of $\alpha$ and $\gamma$.
Figure \ref{kodak23vfig} displays the resulting $v$ in Kodak image nr 23 for $\alpha=10^{-2},\gamma=10^{-3}, \e=7*10^{-2}$.

\begin{figure}
\includegraphics[width=0.25\textwidth,frame]{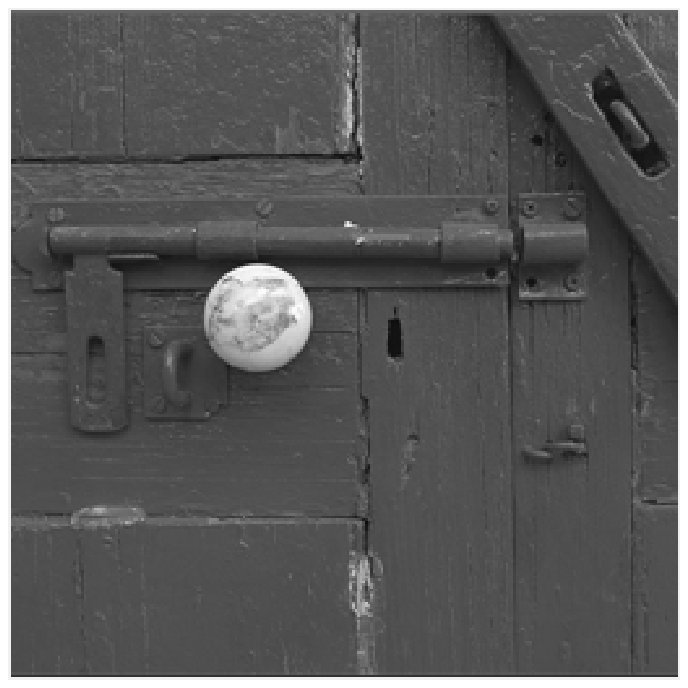}
\includegraphics[width=0.25\textwidth,frame]{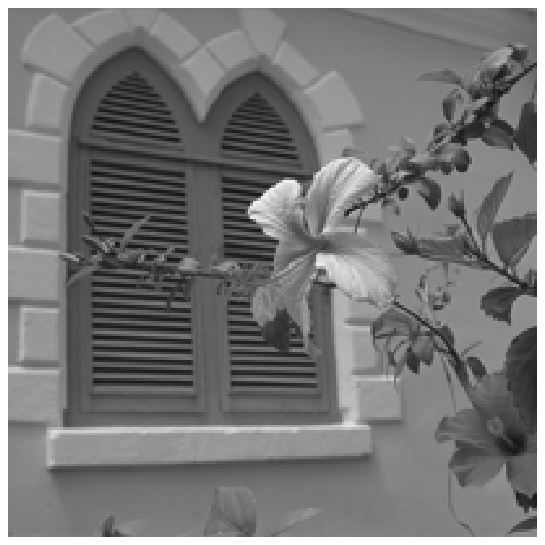}
\includegraphics[width=0.25\textwidth,frame]{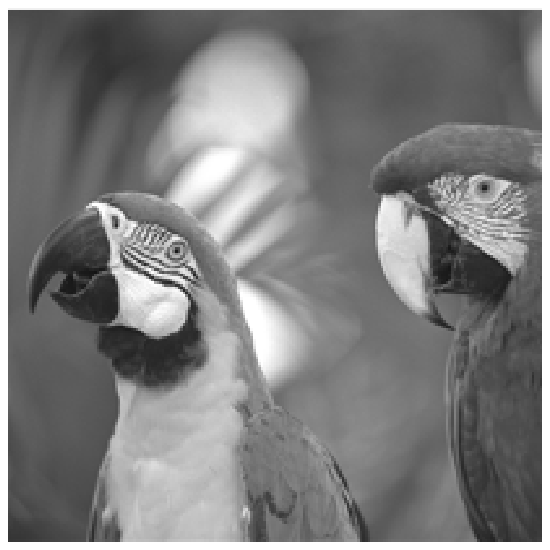}
\caption{Kodak image nr 2 (left), nr 7 (middle) and nr 23 (right). }
\label{kodakdatafig}
\end{figure}

\begin{figure}
\includegraphics[width=0.4\textwidth,frame]{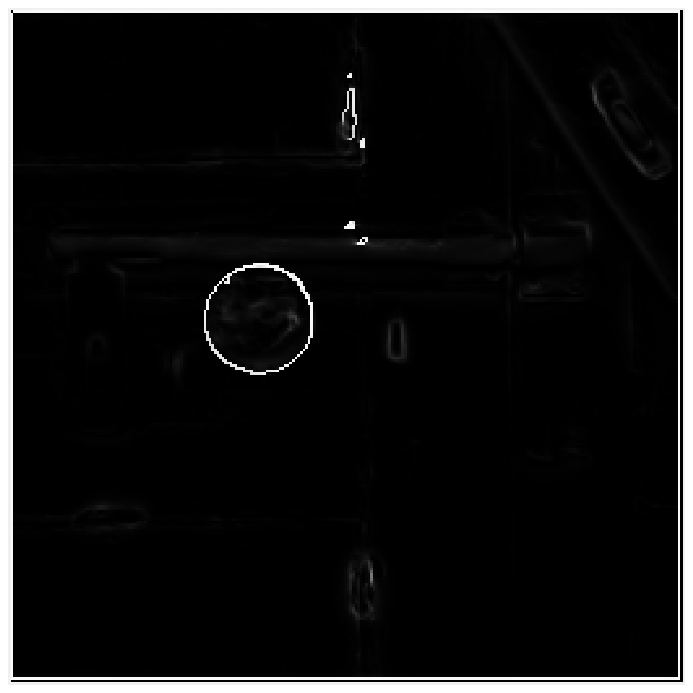}
\includegraphics[width=0.4\textwidth,frame]{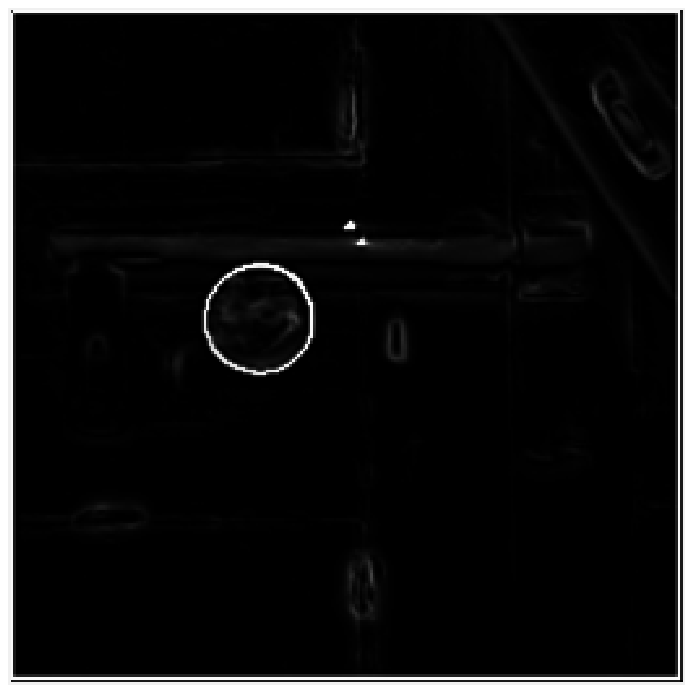}
\caption{Kodak image 2: resulting $v$ in the Ambrosio-Tortorelli model (left) and in the second order model (right), $\gamma=10^{-3},\e=3*10^{-2}$. }
\label{kodak2avfig}
\end{figure}

\begin{figure}
\includegraphics[width=0.4\textwidth,frame]{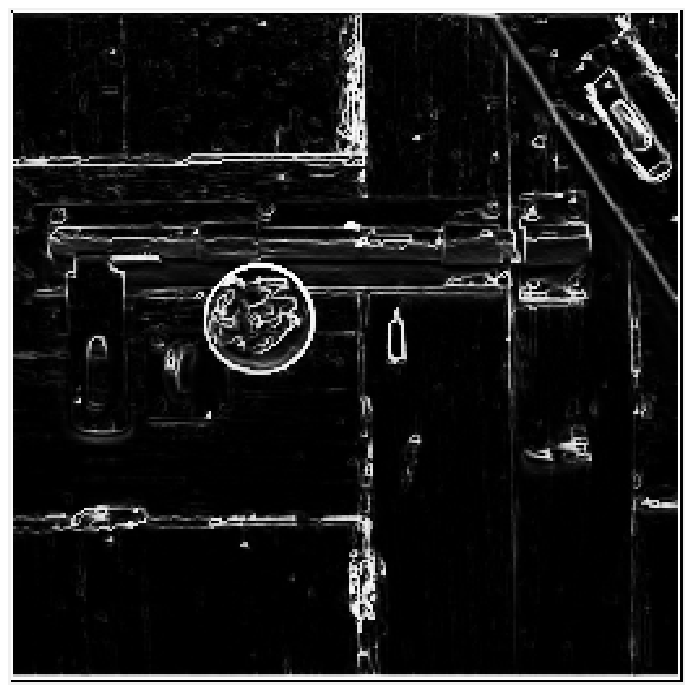}
\includegraphics[width=0.4\textwidth,frame]{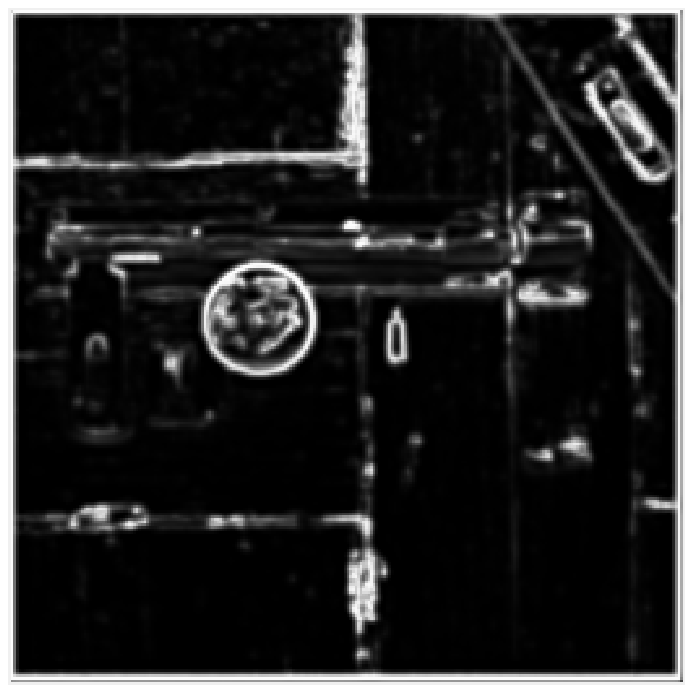}
\caption{Kodak image 2: resulting $v$ in the Ambrosio-Tortorelli model (left) and in the second order model (right), $\gamma=7*10^{-3},\e=6*10^{-2}$. }
\label{kodak2bvfig}
\end{figure}

\begin{figure}
\includegraphics[width=0.4\textwidth,frame]{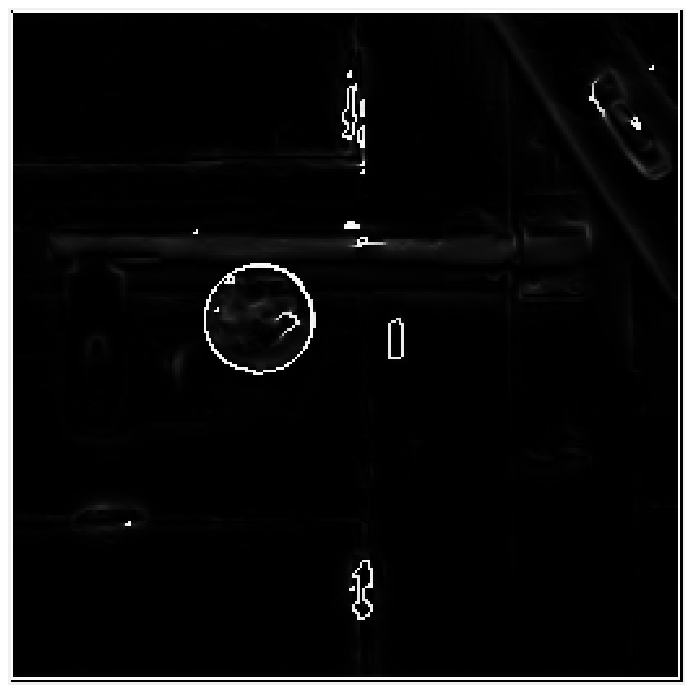}
\includegraphics[width=0.4\textwidth,frame]{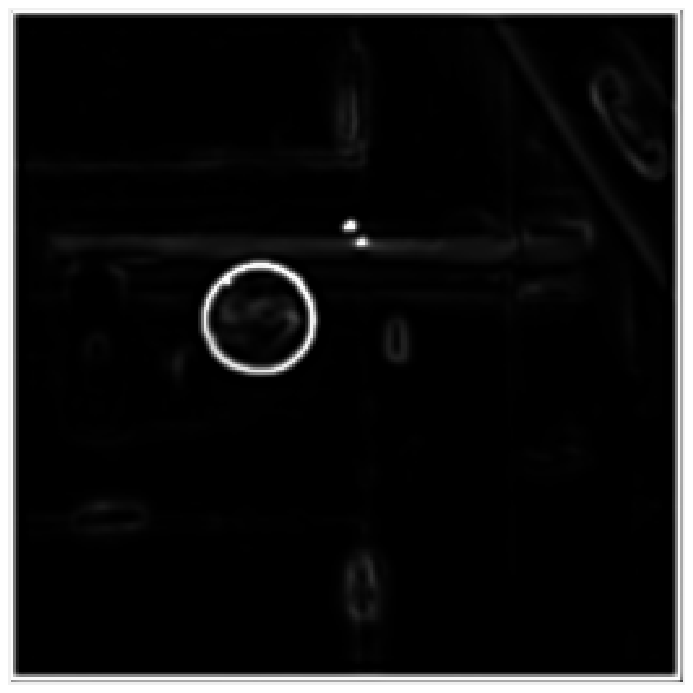}
\caption{Kodak image 2: resulting $u$ in the Ambrosio-Tortorelli model (left) and in the second order model (right), $\gamma=7*10^{-4},\e=6*10^{-2}$. }
\label{kodak2cvfig}
\end{figure}

\begin{figure}
\includegraphics[width=0.4\textwidth,frame]{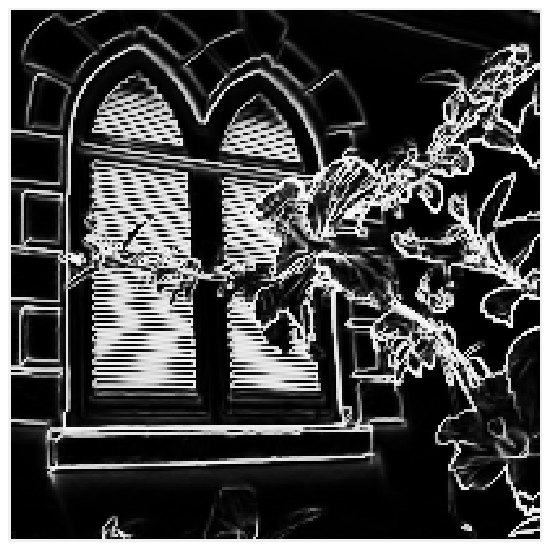}
\includegraphics[width=0.4\textwidth,frame]{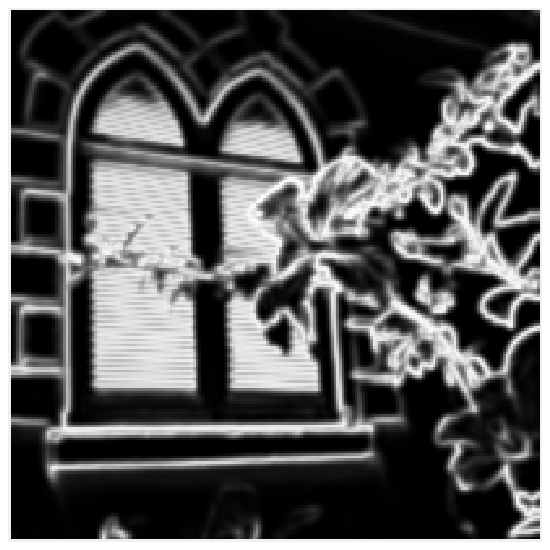}
\caption{Kodak image 7: Resulting $v$ in the Ambrosio-Tortorelli model (left) and in the second order model (right), both with $\alpha=10^{-2}$, $\gamma=7*10^{-3}$. }
\label{kodak7avfig}
\end{figure}

\begin{figure}
\includegraphics[width=0.4\textwidth,frame]{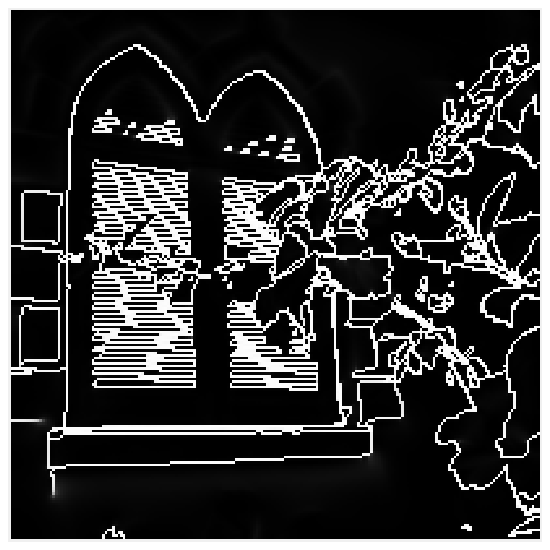}
\includegraphics[width=0.4\textwidth,frame]{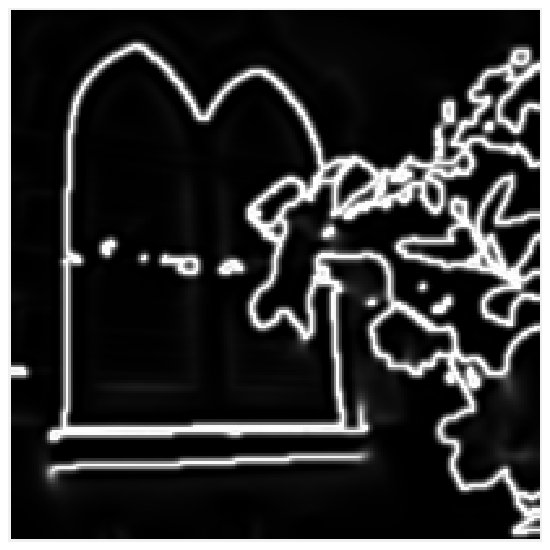}
\caption{Kodak image 7: Resulting $v$ in the Ambrosio-Tortorelli model (left) and in the second order model (right), both with $\alpha=7*10^{-2}$, $\gamma=10^{-3}$. }
\label{kodak7bvfig}
\end{figure}

\begin{figure}
\includegraphics[width=0.4\textwidth,frame]{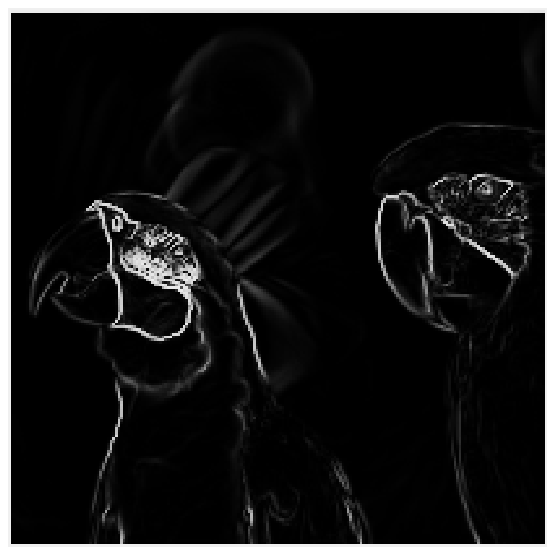}
\includegraphics[width=0.4\textwidth,frame]{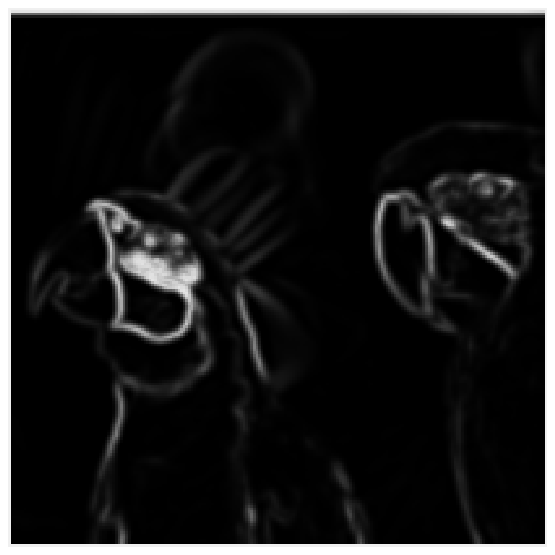}
\caption{Kodak image 23: resulting $v$ in the Ambrosio-Tortorelli model (left) and in the second order model (right). }
\label{kodak23vfig}
\end{figure}

\subsection{Reconstructed Images}

Figures \ref{ellipsefig_2}, \ref{twocirclefig_2}, \ref{sissefig}, and  \ref{mitosisfig_2} display the resulting $u$ in the different models with the parameter settings in the paper. Figures \ref{kodak2afig}, \ref{kodak2bfig}, \ref{kodak2cfig} display results for the Kodak image nr 2 with $\alpha=10^{-2}$ and different values of $\gamma$ and $\e$. Figures \ref{kodak7afig} and \ref{kodak7bfig} display the results for the Kodak image nr 7 with $\e=7*10^{-2}$ and different values of $\alpha$ and $\gamma$.
Figure \ref{kodak23bfig} displays the resulting $u$ in Kodak image nr 23 for both models with parameters $\alpha=10^{-2},\gamma=10^{-3}, \e=7*10^{-2}$.
\begin{figure}
\includegraphics[width=0.4\textwidth,frame]{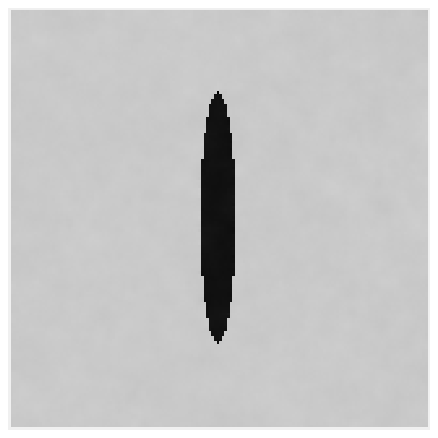}
\includegraphics[width=0.4\textwidth,frame]{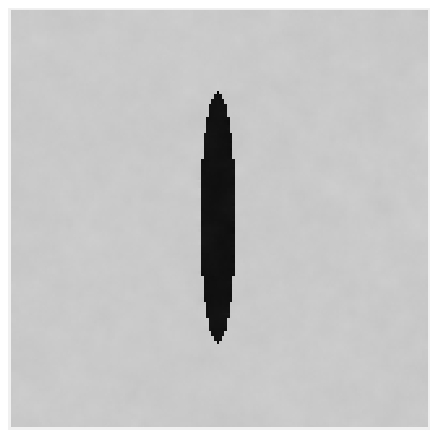}
\caption{Ellipse: resulting $u$ in the Ambrosio-Tortorelli model (left) and in the second order model (right). }
\label{ellipsefig_2}
\end{figure}

\begin{figure}
\includegraphics[width=0.4\textwidth,frame]{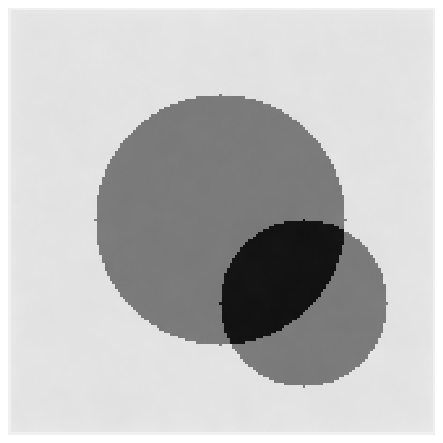}
\includegraphics[width=0.4\textwidth,frame]{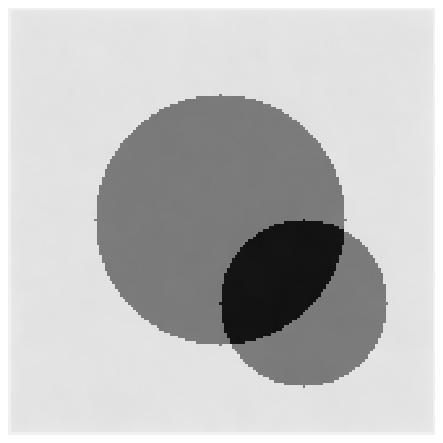}
\caption{Two circles: resulting $u$ in the Ambrosio-Tortorelli model (left) and in the second order model (right). }
\label{twocirclefig_2}
\end{figure}

\begin{figure}
\includegraphics[width=0.4\textwidth,frame]{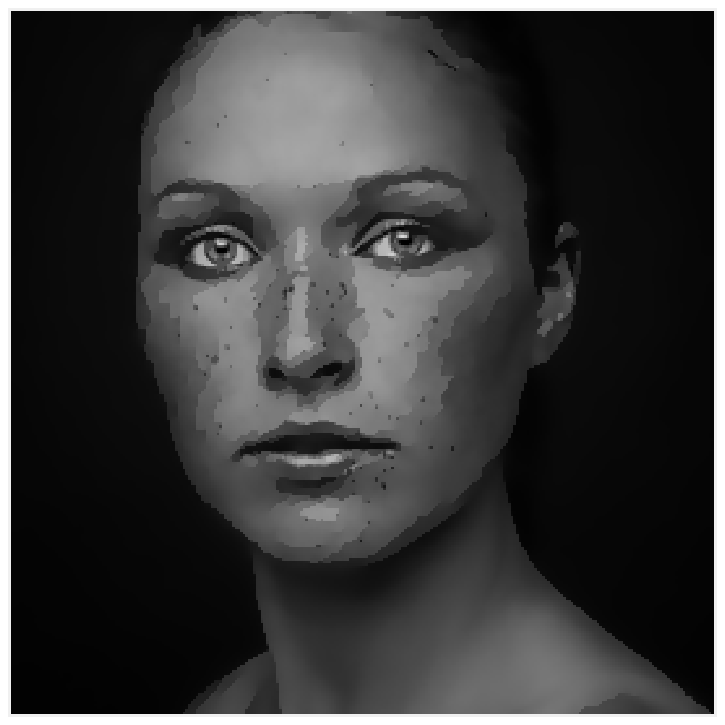}
\includegraphics[width=0.4\textwidth,frame]{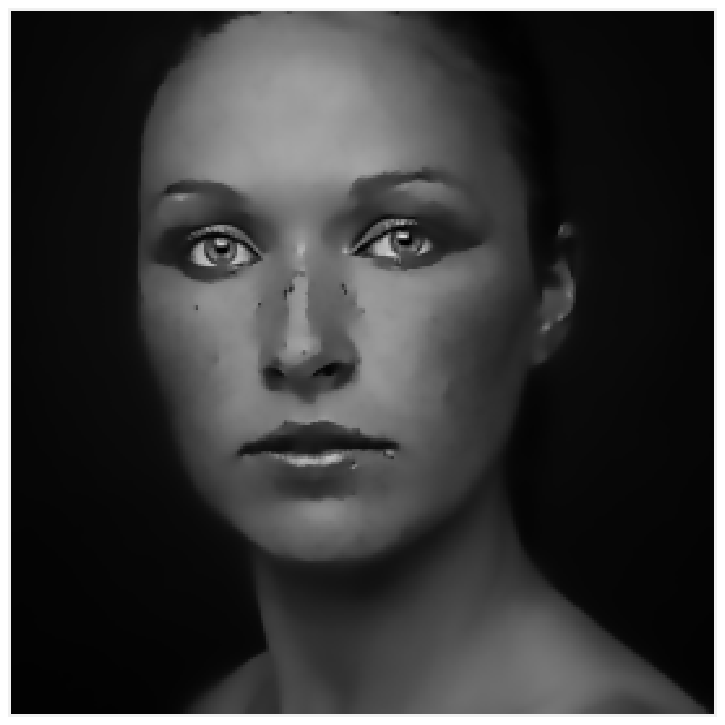}
\caption{Sisse image: resulting $u$ in the Ambrosio-Tortorelli model (left) and in the second order model (right). }
\label{sissefig}
\end{figure}

\begin{figure}
\includegraphics[width=0.4\textwidth,frame]{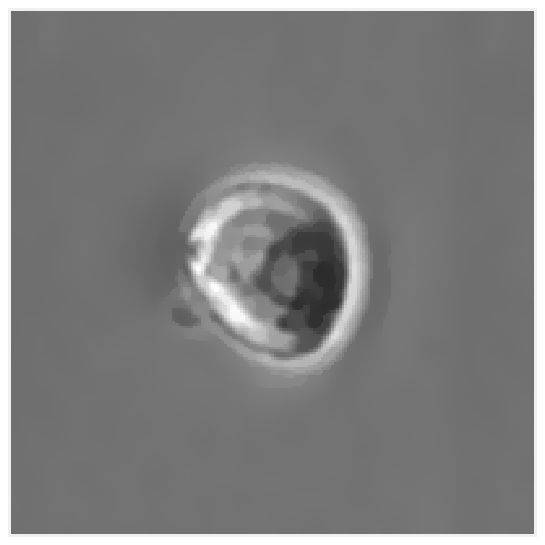}
\includegraphics[width=0.4\textwidth,frame]{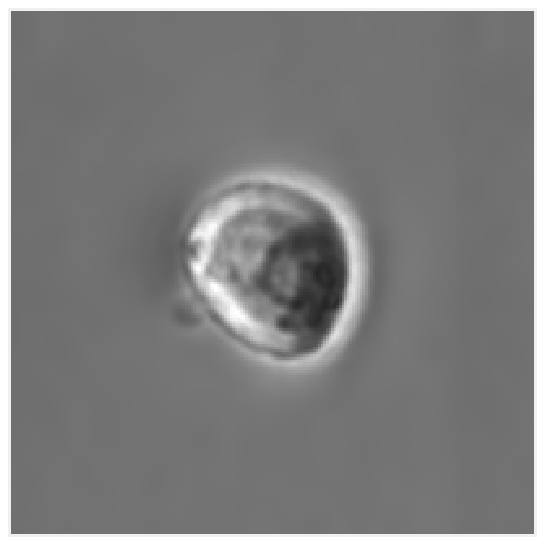}
\caption{Mitosis image: resulting $u$ in the Ambrosio-Tortorelli model (left) and in the second order model (right). }
\label{mitosisfig_2}
\end{figure}

\clearpage

\begin{figure}
\includegraphics[width=0.4\textwidth,frame]{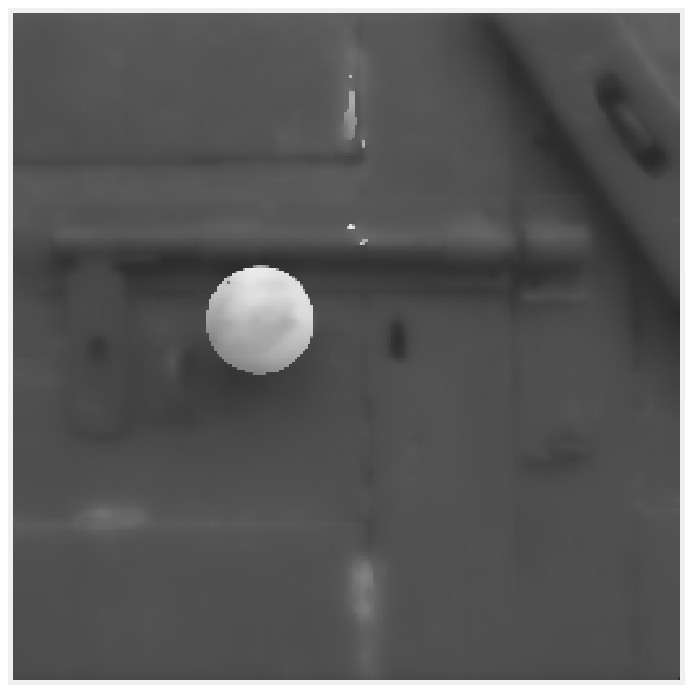}
\includegraphics[width=0.4\textwidth,frame]{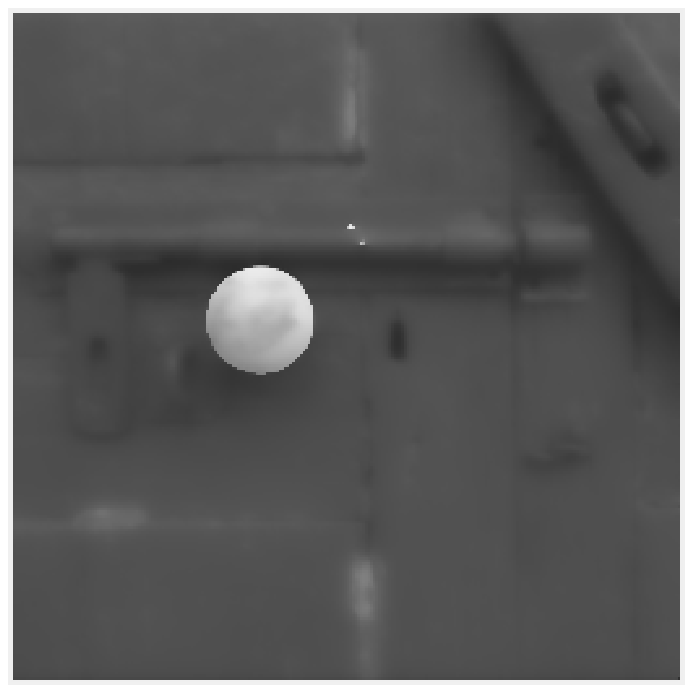}
\caption{Kodak image 2: resulting $u$ in the Ambrosio-Tortorelli model (left) and in the second order model (right), $\gamma=10^{-3},\e=3*10^{-2}$. }
\label{kodak2afig}
\end{figure}

\begin{figure}
\includegraphics[width=0.4\textwidth,frame]{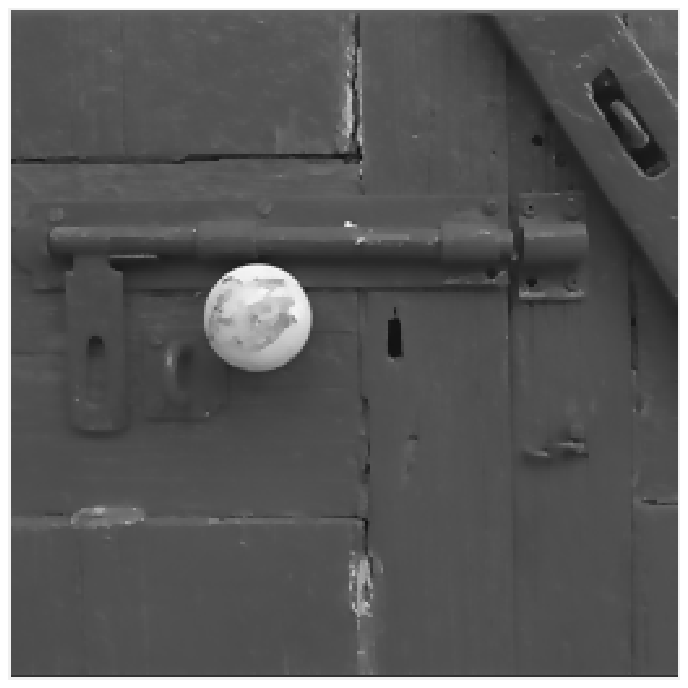}
\includegraphics[width=0.4\textwidth,frame]{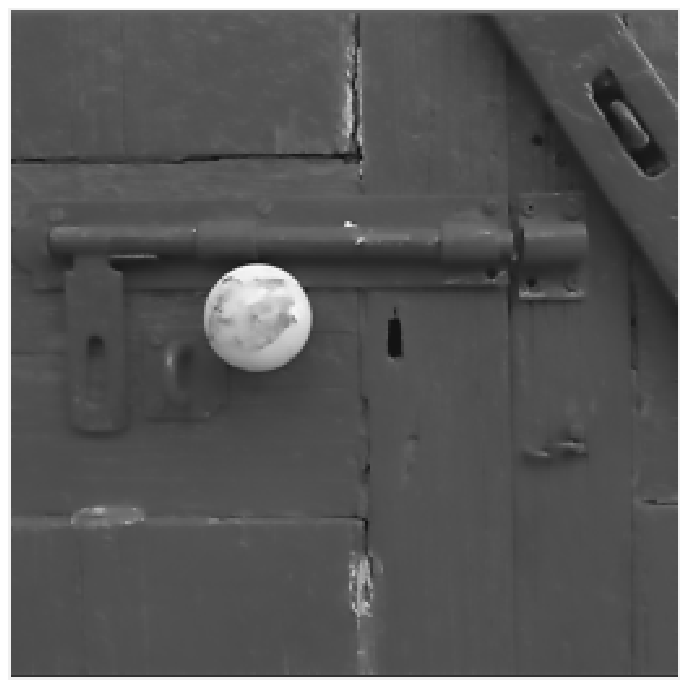}
\caption{Kodak image 2: resulting $u$ in the Ambrosio-Tortorelli model (left) and in the second order model (right), $\gamma=7*10^{-3},\e=6*10^{-2}$. }
\label{kodak2bfig}
\end{figure}

\begin{figure}
\includegraphics[width=0.4\textwidth,frame]{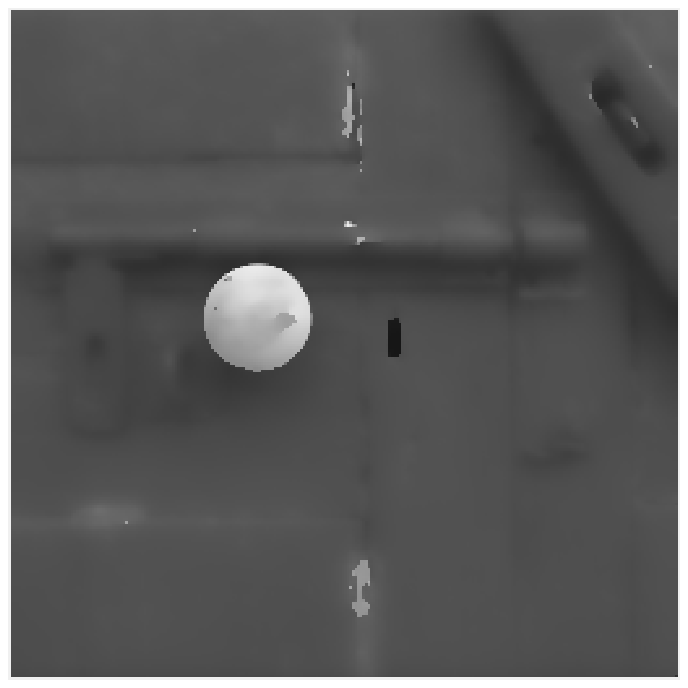}
\includegraphics[width=0.4\textwidth,frame]{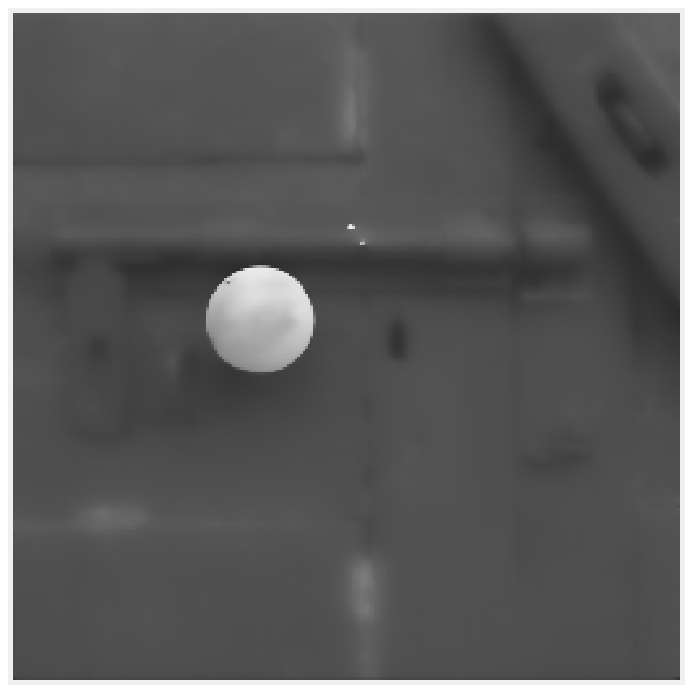}
\caption{Kodak image 2: resulting $u$ in the Ambrosio-Tortorelli model (left) and in the second order model (right), $\gamma=7*10^{-4},\e=6*10^{-2}$. }
\label{kodak2cfig}
\end{figure}

\begin{figure}
\includegraphics[width=0.4\textwidth,frame]{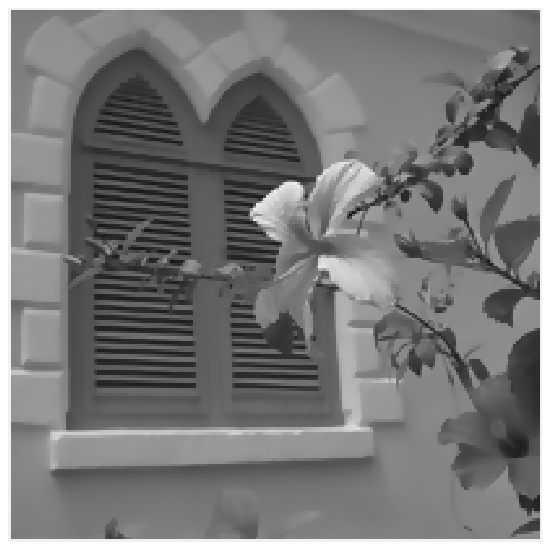}
\includegraphics[width=0.4\textwidth,frame]{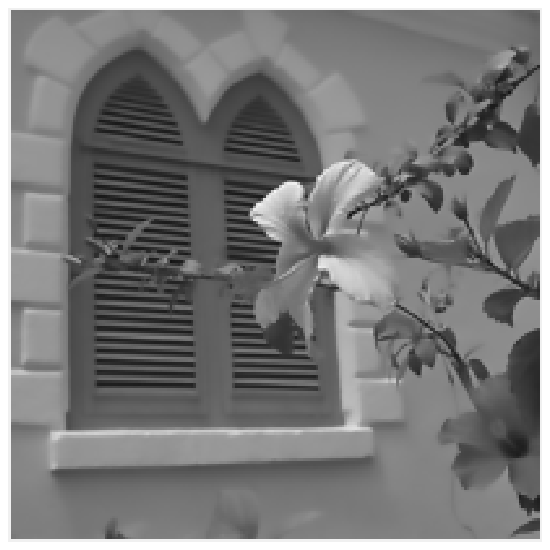}
\caption{Kodak image 7: Resulting $u$ in the Ambrosio-Tortorelli model (left) and in the second order model (right), both with $\alpha=10^{-2},\gamma=7*10^{-3}$. }
\label{kodak7afig}
\end{figure}

\begin{figure}
\includegraphics[width=0.4\textwidth,frame]{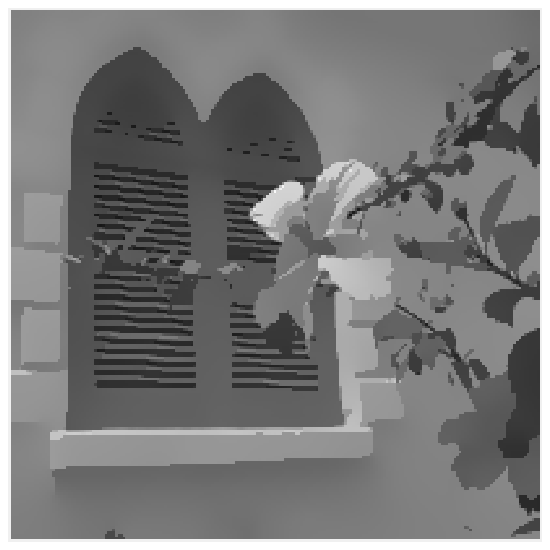}
\includegraphics[width=0.4\textwidth,frame]{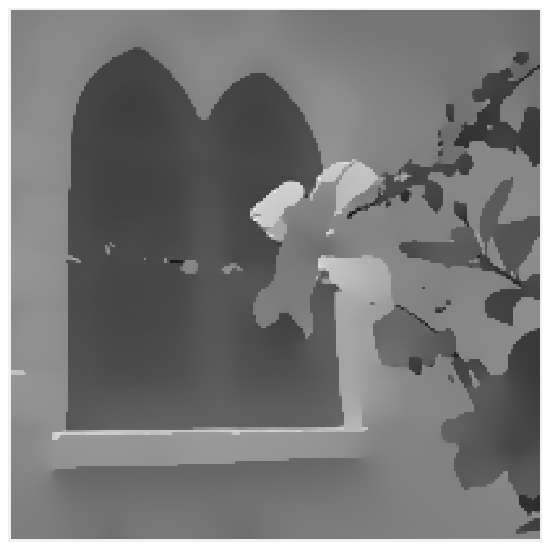}
\caption{Kodak image 7: Resulting $u$ in the Ambrosio-Tortorelli model (left) and in the second order model (right), both with $\alpha=7*10^{-2},\gamma=7*10^{-3}$. }
\label{kodak7bfig}
\end{figure}

\begin{figure}
\includegraphics[width=0.4\textwidth,frame]{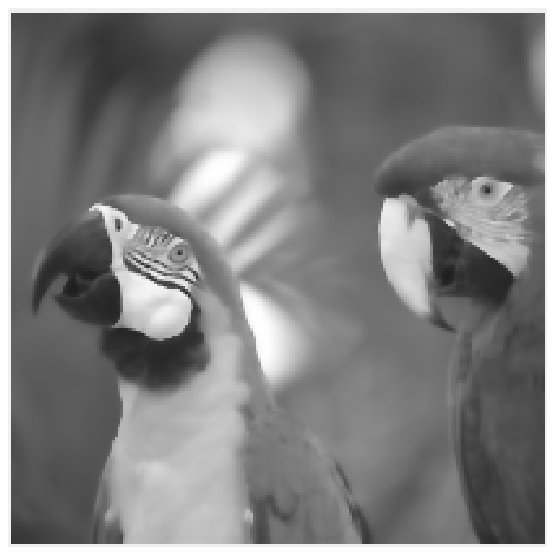}
\includegraphics[width=0.4\textwidth,frame]{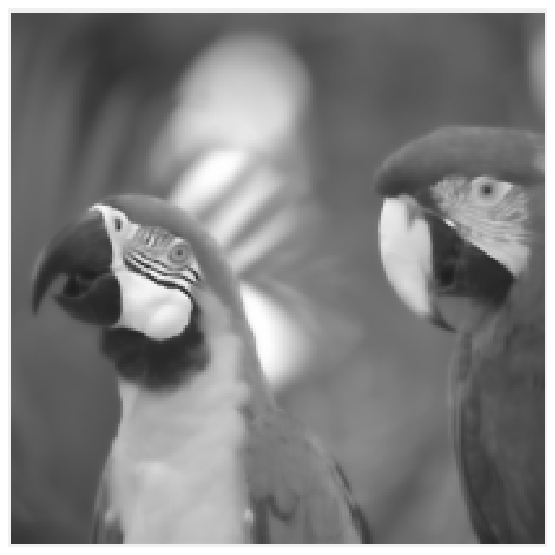}
\caption{Kodak image 23: resulting $u$ in the Ambrosio-Tortorelli model (left) and in the second order model (right). }
\label{kodak23bfig}
\end{figure}


\end{document}